\pdfoutput=1 
\documentclass[twoside,reqno]{elsarticle}
\usepackage[utf8]{inputenc}
\usepackage{color}
\usepackage{mathtools}
\usepackage{amsmath}
\usepackage{amsthm}
\usepackage{amssymb}
\usepackage{stmaryrd}
\usepackage{microtype}
\usepackage[unicode=true,
 bookmarks=false,
 breaklinks=false,pdfborder={0 0 1},backref=false,colorlinks=true]
 {hyperref}
\hypersetup{
 pdfstartview=FitV,linkcolor=black,citecolor=black}

\makeatletter
\numberwithin{equation}{section}
\numberwithin{figure}{section}
\theoremstyle{plain}
\newtheorem{thm}{\protect\theoremname}[section]
\theoremstyle{plain}
\newtheorem{cor}[thm]{\protect\corollaryname}
\theoremstyle{plain}
\newtheorem{prop}[thm]{\protect\propositionname}
\theoremstyle{remark}
\newtheorem{rem}[thm]{\protect\remarkname}
\theoremstyle{plain}
\newtheorem{lem}[thm]{\protect\lemmaname}
\theoremstyle{definition}
\newtheorem{defn}[thm]{\protect\definitionname}
\theoremstyle{plain}
\newtheorem{fact}[thm]{\protect\factname}
\theoremstyle{definition}
\newtheorem{example}[thm]{\protect\examplename}

\journal{~}
\usepackage{bbm}
\usepackage{pgfplots}
\pgfplotsset{compat=1.15}
\usepackage{mathrsfs}
\usetikzlibrary{arrows}

\usepackage{txfonts}
\usepackage{textcomp}
\usepackage{xcolor}
\usepackage{dsfont}
\usepackage{mathrsfs}
\usepackage{graphicx}
\usepackage{tikz}
\renewcommand{\d}{\mathrm{d}}
\renewcommand{\rho}{\varrho}

\newcommand{\1}{\mathbbm{1}} 

\newcommand{\e}{\mathrm{e}} 

\DeclareMathOperator{\dom}{dom}
\DeclareMathOperator{\spann}{span}
\DeclareMathOperator{\supp}{supp}
\DeclareMathOperator{\card}{card}

\renewcommand{\emptyset}{\varnothing}
\renewcommand{\phi}{\varphi}
\def\N{\mathbb N}
\def\d{\,\mathrm d}
\def\R{\mathbb{R}}

\def\E{\mathcal{E}}

\makeatother

\providecommand{\corollaryname}{Corollary}
\providecommand{\definitionname}{Definition}
\providecommand{\examplename}{Example}
\providecommand{\factname}{Fact}
\providecommand{\lemmaname}{Lemma}
\providecommand{\propositionname}{Proposition}
\providecommand{\remarkname}{Remark}
\providecommand{\theoremname}{Theorem}

\begin{document}

\begin{frontmatter}{}

\title{Spectral dimensions of Kre\u{\i}n-Feller operators and $L^{q}$-spectra}

\author{Marc~Kesseböhmer\fnref{fn1}\corref{cor1}}

\ead{mhk@uni-bremen.de}

\author{Aljoscha~Niemann\corref{cor1}}

\ead{niemann1@uni-bremen.de}

\address{Fachbereich 3 -- Mathematik und Informatik, University of Bremen,
Bibliothekstr. 1, 28359 Bremen, Germany}

\fntext[fn1]{This research was supported by the DFG grant Ke 1440/3-1.}

\cortext[cor1]{Corresponding authors}
\begin{abstract}
We study the spectral dimensions and spectral asymptotics of Kre\u{\i}n-Feller
operators for arbitrary finite Borel measures on $\left(0,1\right).$
Connections between the spectral dimension, the $L^{q}$-spectrum,
the partition entropy and the optimized coarse multifractal dimension
are established. In particular, we show that the upper spectral dimension
always corresponds to the fixed point of the $L^{q}$-spectrum of
the corresponding measure. Natural bounds reveal intrinsic connections
to the Minkowski dimension of the support of the associated Borel
measure. Further, we give a sufficient condition on the $L^{q}$-spectrum
to guarantee the existence of the spectral dimension. As an application,
we confirm the existence of the spectral dimension of self-conformal
measures with or without overlap as well as of certain measures of
pure point type. We construct a simple example for which the spectral
dimension does not exist and determine explicitly its upper and lower
spectral dimension.
\end{abstract}

\end{frontmatter}{}

\tableofcontents{}

\section{Introduction and statement of main results}

This article is devoted to the subtle connections between Kre\u{\i}n-Feller
operators, geometric measure theory, and concepts of fractality and
complexity and thus gives new partial answers to Mark Kac's famous
question\emph{ ``Can One Hear the Shape of a Drum?''}.

We investigate the spectral properties of the Kre\u{\i}n-Feller operator
$\Delta_{\nu,[a,b]}$ for a given finite Borel measure $\nu$ on the
interval $\left(a,b\right)$. The Kre\u{\i}n-Feller operator has
been introduced in \citep{Fe57,KK68,MR0042045}. We will elaborate
the abstract approach via the (Dirichlet) form $\E_{\nu,\left[a,b\right]}\left(f,g\right)\coloneqq\int_{a}^{b}\nabla_{\Lambda}f\nabla_{\Lambda}g\d\Lambda$
(with $\Lambda$ denoting the Lebesgue measure) densely defined on
$L_{\nu}^{2}=L_{\nu}^{2}\left(\left[a,b\right]\right)$ as proposed
e.~g\@. in \citep{MR2563669,MR2103098,MR2261337}. This approach
gives rise to an associated Laplace operator $\Delta_{\nu,\left[a,b\right]}$
and an orthonormal basis of eigenfunctions with non-negative eigenvalues
$\left(\lambda^{n}\left(\E_{\nu,\left[a,b\right]}\right)\right)_{n\in\N}$
tending to $\infty$. As a consequence of the Dirichlet/Neumann bracketing
(Prop. \ref{prop:D-NBracketing}) we see that the asymptotic behavior
of the eigenvalues does not depend on the chosen boundary condition
nor on the interval containing $\supp\nu$ and we state the corresponding
results neither referring to boundary conditions nor to the ambient
interval. We denote the number of eigenvalues of $\E_{\nu,\left[a,b\right]}$
not exceeding $x\geq0$ by $N_{\nu,[a,b]}\left(x\right)$ and refer
to $N_{\nu,[a,b]}$ as the \emph{eigenvalue counting function. }If
the ambient interval is clear from the context and $\supp\nu\subset\left[a,b\right]$,
we will also suppress the dependence on the interval, i.~e\@. we
write $\E_{\nu}=\E_{\nu,\left[a,b\right]}$, $N_{\nu}=N_{\nu,[a,b]}$,
$\Delta_{\nu}=\Delta_{\nu,\left[a,b\right]}$, $\lambda_{\nu}^{n}=\lambda^{n}\left(\E_{\nu,\left[a,b\right]}\right)$
etc. Using some invariance properties and bi-Lipschitz equivalence
we can assume without loss of generality that the ambient interval
is given by $\left[0,1\right]$ and we will assume this throughout
the paper. In this setup we give rigorous proofs for the sub- and
superadditivity of the eigenvalue counting function which is crucial
in our analysis. We define the upper and lower exponent of divergence
of $N_{\nu}$ by
\[
\underline{s}_{\nu}\coloneqq\liminf_{x\rightarrow\infty}\frac{\log\left(N_{\nu}(x)\right)}{\log(x)}\quad\text{and }\;\overline{s}_{\nu}\coloneqq\limsup_{x\to\infty}\frac{\log\left(N_{\nu}(x)\right)}{\log(x)}
\]
and refer to these numbers as the \emph{upper}, resp. \emph{lower},
\emph{spectral dimension} of $\E_{\nu}$, $\Delta_{\nu}$ or, just
$\nu$, respectively. If the two values coincide we denote the common
value by $s_{\nu}$ and call it the \emph{spectral dimension}. Note
that we always have $\overline{s}_{\nu}\leq1/2$ which has been shown
in \citep{MR0209733,MR0217487}. The case for measures with non-trivial
absolutely continuous part was completely solved in an elegant way
in \citep{MR0278126} using a variational approach. In this case,
for a Borel measure $\nu$ on the unit interval with absolutely continuous
non-vanishing part $\sigma\Lambda$ and singular part $\eta$, we
have 
\[
\lim_{x\to\infty}\frac{N_{\eta+\sigma\Lambda}(x)}{x^{1/2}}=\frac{1}{\pi}\int\sqrt{\sigma}\d\Lambda
\]
and particularly the spectral dimension exists and equals $1/2$.
Besides this estimate, many partial results have been obtained showing
that there is a subtle relation between spectral properties and geometric
data of $\nu$, which has been a major line of investigation since
the famous result by Hermann Weyl \citep{zbMATH02629881}. Weyl's
result includes the $d$-dimensional Lebesgue case with spectral dimension
$d/2$; for this reason, some authors prefer to set the spectral dimension
to twice the value defined here.

We will make use of the following notation. For any two functions
$f,g:\R_{\geq0}\rightarrow\R_{>0}$ we write $f\ll g$ if there exists
a positive constant $c$ such that $cf(x)\leq g(x)$ for all $x$
large; we write $f\asymp g$, if both $f\ll g$ and $g\ll f$ holds.
The asymptotic behavior of $N_{\nu}$ strongly depends on the measure
$\nu$. In general, one can not expect that $N_{\nu}$ obeys a power
law with a positive exponent $s>0$, i.~e\@. $N_{\nu}(x)\asymp x^{s}$
(for counter examples see e.~g\@. \citep{Arzt_diss} or Examples
\ref{NonExistenceSpectralDim} and \ref{exmpl:ExponentialDecay}).
However, if we restrict our attention to the spectral dimension this
connection becomes more feasible and surprisingly, we are able to
treat arbitrary Borel measures on $(0,1)$ and determine the upper
spectral dimension solely from the data provided by the measure-geometric
information carried by the $L^{q}$-spectrum of $\nu$. Under mild
regularity conditions on the measure we can guarantee the existence
of the spectral dimension. Also with the help of the $L^{q}$-spectrum
we are able to construct first examples for which the spectral dimension
does not exist. In this way we give a partial answer to Kac's question
in terms of the spectral dimension revealing how the measure theoretic
properties of $\nu$ and the topological properties of its ``fractal''
support are interwoven to determine the spectral dimension (see Theorem
\ref{Thm:MainChain_of_Inequalities+Regularity} and Corollary \ref{cor:spectralDimGeneralUpperBound}).
This striking connection will be the main line of investigation in
this paper.

Our approach lies in the construction of optimal partitions as introduced
in Section \ref{sec:OptimalPartitions} and to make use of the min-max
principle. This idea goes back to the fundamental works of \citep{MR0217487,MR0482138,MR0278126}
which deal with a systematic treatment of Sobolev embedding theorems
in tandem with variational principles. The following observation motivates
the definition of optimal $\nu$-partitions. Due to Proposition \ref{prop:D-NBracketing},
it is sufficient to consider $\Delta_{\nu}$ under Dirichlet boundary
conditions and we will therefore mainly restrict ourselves to this
case. The first key property to determine the spectral dimension is
the sub- and superadditivity of the eigenvalue counting function with
respect to interval partitions. Let $P$ be a $\nu$-\emph{partition}
of $\left[0,1\right]$, that is a partition with half open intervals
$\left\{ (c_{k},c_{k+1}]:k=1,\ldots,n-1\right\} $, $n\in\N_{\geq2}$,
such that $0=c_{1}<c_{2}\cdots<c_{n}=1$ and there are no $\nu$-atoms
on the endpoints $\nu\left(\left\{ c_{k}\right\} \right)=0$. Then
by Proposition \ref{prop:SubandSupperAdditivity}, we have
\[
\text{\ensuremath{\sum_{k=1}^{n}N_{\nu,[c_{k},c_{k+1}]}}}(x)\leq\text{\ensuremath{N_{\nu}}}(x)\leq\text{\ensuremath{\sum_{k=1}^{n}N_{\nu,[c_{k},c_{k+1}]}}}(x)+n-2.
\]
The second key observation is the Poincaré inequality, i.~e\@. the
first eigenvalue $\lambda_{\nu,[a,b]}^{1}$ of $\Delta_{\nu,[a,b]}$
can be bounded from below by $\left(\Lambda([a,b])\nu([a,b])\right)^{-1}$
(see Lemma \ref{lem:Poincare_inequality}). In particular, if $\Lambda([a,b])\nu([a,b])<1/x$,
then $N_{\nu,[a,b]}(x)=0.$ For a $\nu$-partition $P$ of $\left[0,1\right]$
such that $\Lambda(C)\nu(C)<1/x$ for all $C\in P$ and some $x>0$,
the superadditivity implies 
\[
\text{\ensuremath{N_{\nu}}}(x)\leq\card(P).
\]
Similar ideas will be employed to derive appropriate lower bounds.
These observations lead to the following definitions.\textbf{ }Let
us consider the set $\Pi$ of all finite $\nu$-partitions of $\left[0,1\right]$
and set
\[
\mathcal{N}^{R}\left(x\right)\coloneqq\inf\left\{ \card\left(P\right):P\in\Pi:\max_{C\in P}\nu\left(C\right)\Lambda\left(C\right)<1/x\right\} .
\]
The following quantities 
\[
\overline{h}_{\nu}\coloneqq\limsup_{x\to\infty}\frac{\log\mathcal{N}^{R}\left(x\right)}{\log x},\quad\underline{h}_{\nu}\coloneqq\liminf_{x\to\infty}\frac{\log\mathcal{N}^{R}\left(x\right)}{\log x},
\]
will be called the\emph{ upper, }resp.\emph{ lower, $\nu$-partition
entropy.} These quantities have interesting connections to the classical
work of Birman and Solomjak \citep{MR0209733,MR0217487,MR0281860}
as explained in Section \ref{subsec:First-observations}. Note that
ignoring intervals of measure zero in the definition of the $\nu$-partition
will not change the value of the\emph{ $\nu$-}partition\emph{ }entropy.

For effective lower bounds, we need the following variant of this
notion. Let $\Pi_{0}$ denote the set of finite disjoint collections
of subintervals of $I$ such that there are no atoms on the endpoints
of each interval, and for $m>1$, set
\[
\mathcal{N}_{m}^{L}\left(x\right)\coloneqq\sup\left\{ \card\left(P\right):P\in\Pi_{0}:\min_{C\in P}\nu\left(\left\langle I\right\rangle _{m}\right)\Lambda\left(\left\langle I\right\rangle _{m}\right)\geq\frac{4}{x(m-1)}\right\} ,
\]
where $\left\langle I\right\rangle _{m}\subset I$ denotes the interval
of length $\Lambda\left(I\right)/m$ centered in $I$. Then the \emph{lower
}and\emph{ upper $m$-reduced $\nu$-partition entropy} is given by
\[
\underline{h}_{\nu}^{m}\coloneqq\liminf_{x\to\infty}\frac{\log\mathcal{N}_{m}^{L}\left(x\right)}{\log x}\:\text{and }\overline{h}_{\nu}^{m}\coloneqq\limsup_{x\to\infty}\frac{\log\mathcal{N}_{m}^{L}\left(x\right)}{\log x}.
\]

The second important object depends on dyadic partitions of the unit
interval and gives rise to the notion of the \emph{$L^{q}$-spectrum}
of $\nu$ given by 
\[
\beta_{\nu}\left(q\right)\coloneqq\limsup_{n\rightarrow\infty}\beta_{n}^{\nu}\left(q\right)\quad\text{with\quad\ }\beta_{n}^{\nu}\left(q\right)\coloneqq\frac{1}{\log2^{n}}\log\sum_{C\in\mathcal{D}_{\nu,n}}\nu\left(C\right)^{q},
\]
where $\mathcal{D}_{\nu,n}\coloneqq\left\{ A_{n}^{k}:k\in\left\{ 1,\ldots,2^{n}\right\} ,\,\nu\left(A_{n}^{k}\right)>0\right\} $
and $A_{n}^{k}\coloneqq\left(\left(k-1\right)2^{-n},k2^{-n}\right]$
and $q\in\R$. Note that this notion has recently gained a lot of
attention in the context of fractal geometry in the work of \citep{MR3224722,MR3919361}
--~connections to our work on spectral dimensions will be provided
in the forthcoming paper \citep{KN21}. Each $\beta_{n}^{\nu}$ is
differentiable and convex with $\beta_{n}^{\nu}\left(1\right)=0$
and there exists a unique fixed point $q_{n}$, i.~e\@. $\beta_{n}^{\nu}\left(q_{n}\right)=q_{n}$.
Also $\beta_{\nu}\left(0\right)\eqqcolon\delta_{\nu}^{*}=\overline{\dim}_{M}\left(\supp\nu\right),$
where $\overline{\dim}_{M}\left(A\right)$ denotes the upper Minkowski
dimension of $A\subset\mathbb{R}$. As a shorthand notation we will
also write $\overline{\dim}_{M}\left(\nu\right)$ instead of $\overline{\dim}_{M}\left(\supp\nu\right)$.
It will turn out (see Theorem \ref{Thm:MainChain_of_Inequalities+Regularity})
that 
\[
\overline{q}_{\nu}\coloneqq\limsup q_{n}
\]
completely determines the upper spectral dimension. We also need the
following left and right derivatives of $\beta_{\nu}$ in $1$: 
\[
\overline{\delta}_{\nu}\coloneqq\lim_{t\nearrow1}\beta_{\nu}\left(t\right)/\left(1-t\right)\;\text{and }\;\underline{\delta}_{\nu}\coloneqq\lim_{t\searrow1}\beta_{\nu}\left(t\right)/\left(1-t\right).
\]
Conceptually, for the lower bound we borrow ideas from the \emph{coarse
multifractal analysis} (see \citep{MR3236784,MR1312056}). For all
$n\in\N$, we define for $\alpha>0$
\[
\overline{F}_{\nu}\left(\alpha\right)\coloneqq\limsup_{n}\frac{\log^{+}\left(\mathcal{N}_{\alpha}\left(n\right)\right)}{\log2^{n}}\,\,\text{and }\underline{F}_{\nu}\left(\alpha\right)\coloneqq\liminf_{n}\frac{\log^{+}\left(\mathcal{N}_{\alpha}\left(n\right)\right)}{\log2^{n}},
\]
where
\[
\mathcal{N}_{\alpha}\left(n\right)\coloneqq\card M_{n}\left(\alpha\right),\quad M_{n}\left(\alpha\right)\coloneqq\left\{ 0\leq k<2^{n}:\nu\left(A_{k}^{n}\right)\geq2^{-\alpha n}\right\} 
\]
and refer to the quantities 
\[
\overline{F}_{\nu}\coloneqq\sup_{\alpha>0}\frac{\overline{F}_{\nu}\left(\alpha\right)}{\left(1+\alpha\right)}\:\text{and \,\,\ensuremath{\underline{F}_{\nu}\coloneqq\sup_{\alpha>0}\frac{\underline{F}_{\nu}\left(\alpha\right)}{\left(1+\alpha\right)}}}
\]
as the \emph{upper}, resp. \emph{lower, optimized coarse multifractal
dimension}. We shall see later that the highly relevant connections
between the optimal and the dyadic partition approach are strongly
linked by ideas from the theory of large deviations.

The following list of results gives the main achievements of this
paper. The proofs are postponed to Sections \ref{sec:OptimalPartitions}
and \ref{sec:-spectrum-and-coarse}.
\begin{thm}
\label{Thm:MainChain_of_Inequalities+Regularity} For all $1<m\leq3$,
we have
\begin{equation}
\underline{F}_{\nu}\leq\underline{h}_{\nu}^{m}\leq\underline{s}_{\nu}\leq\underline{h}_{\nu}\leq\overline{h}_{\nu}=\overline{h}_{\nu}^{m}=\overline{s}_{\nu}=\overline{q}_{\nu}=\overline{F}_{\nu}.\label{eq:MainInequalities}
\end{equation}
In particular, $\overline{q}_{\nu}\leq1/2$ and the following necessary
and sufficient conditions hold for the spectral dimension to exist:
\[
\underline{s}_{\nu}=\overline{s}_{\nu}\implies\underline{h}_{\nu}=\overline{h}_{\nu}=s_{\nu}\;\quad\text{and }\quad\:\sup_{m>1}\underline{h}_{\nu}^{m}=\overline{h}_{\nu}\implies\underline{s}_{\nu}=\overline{s}_{\nu}=\overline{h}_{\nu}.
\]
We would like to emphasize that if the spectral dimension exists,
then it is given by purely measure-geometric data, which is encoded
in the $\nu$-partition entropy $\overline{h}_{\nu}=\underline{h}_{\nu}$.
We call $\nu$ \emph{regular}, if $\sup_{m>1}\underline{h}_{\nu}^{m}=\overline{h}_{\nu}$,
in which case the spectral dimension exists. If $\nu$ is \emph{MF-regular,}
that is $\underline{F}_{\nu}=\overline{F}_{\nu}$, then in the above
chain of inequalities \eqref{eq:MainInequalities} we have everywhere
equality and especially $\nu$ is\emph{ regular.} Moreover, if for
some $m>1$ we have $\underline{h}_{\nu}^{m}\geq1/2$, then $s_{\nu}=1/2=\overline{h}_{\nu}$.
\end{thm}

The following theorem shows that the $L^{q}$-spectrum is a valuable
auxiliary concept to determine the spectral behavior of a given measure
$\nu$. We call the measure $\nu$ \emph{$L^{q}$-regular}, if $\beta_{\nu}\left(q\right)=\liminf_{n}\beta_{n}^{\nu}\left(q\right)$
for $q\in(\overline{q}_{\nu}-\varepsilon,\overline{q}_{\nu})$, for
some $\varepsilon>0$, or $\beta_{\nu}\left(\overline{q}_{\nu}\right)=\liminf_{n}\beta_{n}^{\nu}\left(\overline{q}_{\nu}\right)$
and $\beta_{\nu}$ is differentiable in $\overline{q}_{\nu}$.
\begin{thm}
\label{thm:LqRegularImpliesRegular}If $\nu$ is $L^{q}$-regular,
then it is MF-regular.
\end{thm}

This result is optimal in the sense that there is an easily accessible
example of a measure $\nu$ which is not $L^{q}$-regular and for
which $\overline{s}_{\nu}=\overline{h}_{\nu}>\underline{h}_{\nu}=\underline{s}_{\nu}$
and the spectral dimension does not exist (see Section \ref{NonExistenceSpectralDim}).
It should be noted that $L^{q}$ regularity is much easier accessible
than regularity or MF regularity and that, to the best of our knowledge,
all measures examined in the literature for which the spectral dimension
is known are $L^{q}$-regular.

As a first application of these considerations, we present the case
of measures $\nu$ with absolutely continuous part based solely on
the $L^{q}$-spectrum.
\begin{cor}
\label{cor:spec_absolutely_cont}If $\nu$ has an absolutely continuous
part with respect to $\Lambda$, then for $q\in[0,1]$, the $L^{q}$-spectrum
exists as a limit with 
\[
\beta_{\nu}(q)=1-q.
\]
 In particular, $\nu$ is $L^{q}$-regular and the spectral dimension
exists and equals $1/2$.
\end{cor}

Also for measures without an absolutely continuous part we have the
following rigidity result in terms of reaching the maximum possible
value $1/2$ of the spectral dimension.
\begin{cor}
\label{cor:q-1 implies s0=00003D1/2}The following rigidity result
holds:
\begin{enumerate}
\item If $\overline{s}_{\nu}=1/2$, then $\beta_{\nu}(q)=1-q$ for all $q\in[0,1]$.
\item If $\beta_{\nu}\left(q\right)=\lim\beta_{n}^{\nu}\left(q\right)=1-q$
for some $q\in\left(0,1\right)$, then $\beta_{\nu}(q)=1-q$ for all
$q\in[0,1]$ and $s_{\nu}=1/2$.
\end{enumerate}
\end{cor}

In the following proposition we present lower bounds of the lower
spectral dimension in terms of the subdifferential $\partial\beta_{\nu}\left(q\right)\coloneqq\left\{ a\in\R:\forall t\in\R:\beta_{\nu}\left(t\right)\geq a\left(t-q\right)+\beta_{\nu}\left(q\right)\right\} $
of $\beta_{\nu}$ in $q$.
\begin{prop}
\label{prop:loverbound_by_diff_in_1}If for $q\in[0,1]$ we have $\beta_{\nu}(q)=\lim_{n}\beta_{n}^{\nu}\left(q\right)$
and $-\partial\beta_{\nu}(q)=[a,b]$, then 
\[
\frac{aq+\beta_{\nu}(q)}{1+b}\leq\underline{s}_{\nu}.
\]
\end{prop}

\begin{rem}
In the case that $\beta_{\nu}(\overline{q}_{\nu})=\lim\beta_{n}^{\nu}\left(\overline{q}_{\nu}\right)$
and $\beta_{\nu}$ is differentiable in $\overline{q}_{\nu}$, we
infer $\overline{q}_{\nu}\leq\underline{s}_{\nu}$ and hence obtain
a direct proof of the regularity statement, namely, $\overline{q}_{\nu}=\underline{s}_{\nu}=\overline{s}_{\nu}$.
\end{rem}

As a corollary to Proposition \ref{prop:loverbound_by_diff_in_1}
and Theorem \ref{Thm:MainChain_of_Inequalities+Regularity} we improve
the known general upper bound of the spectral dimension of $1/2$
as obtained in \citep{MR0209733} in terms of the upper Minkowski
dimension. Furthermore, we obtain a general lower bound of $\underline{s}_{\nu}$
in terms of the left and right-hand derivative of $\beta_{\nu}$.
\begin{cor}
\label{cor:spectralDimGeneralUpperBound}For the upper and lower spectral
dimension we have the following general upper and lower bounds depending
on the topological support of $\nu$, namely $\delta_{\nu}^{*}=\overline{\dim}_{M}\left(\nu\right)$,
and right and left derivative of $\beta_{\nu}$ in 1:

\[
\frac{\underline{\delta}_{\nu}}{1+\overline{\delta}_{\nu}}\leq\underline{s}_{\nu}\leq\overline{s}_{\nu}\leq\frac{\delta_{\nu}^{*}}{1+\delta_{\nu}^{*}}\leq\frac{1}{2}
\]
and
\[
\overline{s}_{\nu}=\frac{\delta_{\nu}^{*}}{1+\delta_{\nu}^{*}}\iff\overline{\delta}_{\nu}=\delta_{\nu}^{*}.
\]
\end{cor}

\begin{rem}
It is worth pointing out that these bounds have been first observed
in the self-similar case under the open set condition in \citep[p. 245]{MR1328700}
(in this case the Minkowski dimension and the Hausdorff of $\supp\nu$
coincide as well as $\beta_{\nu}$ is differentiable in $1$ and $\beta'_{\nu}(1)$
coincides with the Hausdorff dimension of $\nu).$ Furthermore, note
that in the case that $\nu$ has an atomic part we always have $\underline{\delta}_{\nu}=0$
(see Fact \ref{fact: atomicPart}). Hence, the lower bound $\underline{\delta}_{\nu}/(1+\overline{\delta}_{\nu})$
is only meaningful in the case of atomless measures.
\end{rem}

Regarding Kac's question, if $\beta_{\nu}$ is differentiable in $1$,
then the spectral dimension is determined by fractal-geometric quantities
as follows 
\[
\dim_{H}\left(\nu\right)=\dim_{P}\left(\nu\right)\leq\frac{\underline{s}_{\nu}}{1-\underline{s}_{\nu}}\leq\frac{\overline{s}_{\nu}}{1-\overline{s}_{\nu}}\leq\overline{\dim}_{M}\left(\nu\right).
\]

Under the additionally regularity condition imposed on $\nu$, we
have the following rigidity result:
\begin{cor}
\label{cor:Strongly regular max dim}If $\nu$ is $L^{q}$-regular,
then we have
\begin{enumerate}
\item ${\displaystyle \underline{\delta}_{\nu}=\delta_{\nu}^{*}\iff s_{\nu}=\frac{\delta_{\nu}^{*}}{1+\delta_{\nu}^{*}}}$,
\item ${\displaystyle \overline{\dim}_{M}\left(\nu\right)=\dim_{H}\left(\nu\right)\implies s_{\nu}=\frac{\delta_{\nu}^{*}}{1+\delta_{\nu}^{*}}}$,
\item $\beta_{\nu}$ is differentiable in $1$ and $s_{\nu}=\frac{\delta_{\nu}^{*}}{1+\delta_{\nu}^{*}}$
implies 
\[
\overline{\dim}_{M}\left(\nu\right)=\dim_{H}\left(\nu\right)=\underline{\dim}_{H}\left(\nu\right)\coloneqq\inf\left\{ \dim_{H}\left(E\right):\nu\left(E\right)>0\right\} =\dim_{P}\left(\nu\right),
\]
where $\dim_{H}\left(E\right)$, $\dim_{P}\left(E\right)$ denotes
the Hausdorff dimension, packing dimension of $E$, respectively.
\end{enumerate}
\end{cor}

The second implication explains that in earlier publication (e.~g\@.
\citep{Fu87,MR118891}) the spectral dimension has typically been
connected to the Hausdorff dimension of $\nu$. The implication in
(3) concerning the Hausdorff and packing dimension goes back to the
work of Ngai and Heurteaux \citep{MR1402878,MR1625871}.

The last implication is optimal in the sense that we find a measure
$\nu$ (Example \ref{exmpl:RegularBehaviour_full_dimension=000023})
with $1=\delta_{\nu}^{*}=\overline{\delta}_{\nu}$, $\beta$ is not
differentiable in $1$ and $0=\dim_{H}\left(\nu\right)=\dim_{P}\left(\nu\right)<\overline{\delta}_{\nu}=1$.
We note that the equality $\delta_{\nu}^{*}=\dim_{P}\left(\nu\right)$
is implied by $\delta_{\nu}^{*}=\dim_{H}\left(\nu\right)$, $\delta_{\nu}^{*}=h^{*}\left(\nu\right)=\limsup_{n}\left(\log2^{n}\right)^{-1}\sum_{C\in\mathcal{D}_{n}}\nu\left(C\right)\log\nu\left(C\right)$,
or by the identities $\delta_{\nu}^{*}=h_{*}\left(\nu\right)=\liminf_{n}\left(\log2^{n}\right)^{-1}\sum_{C\in\mathcal{D}_{n}}\nu\left(C\right)\log\nu\left(C\right)$.

The following proposition complements the connection of the Minkowski
dimension by establishing an upper bound of the lower spectral dimension
in terms of the lower Minkowski dimension $\underline{\dim}_{M}(\nu)$
of the suppeort of $\nu$ .
\begin{prop}
\label{prop:upperBoundForUnderline_s<Mdim}We always have
\[
\underline{s}_{\nu}\leq\underline{h}_{\nu}\leq\underline{\dim}_{M}(\nu).
\]
\end{prop}

Next we give a lower bound for the lower spectral dimension in terms
of the Minkowski dimension (if it exists) and the derivative of the
$L^{q}$-spectrum in $0$.
\begin{cor}
If the Minkowski dimension of $\supp\nu$ exists, i.~e\@. $\overline{\dim}_{M}\left(\nu\right)=\underline{\dim}_{M}\left(\nu\right)$,
and $-\partial\beta_{\nu}\left(0\right)=[a,b]$, then
\[
\frac{\delta_{\nu}^{*}}{1+b}\leq\underline{s}_{\nu}.
\]
\end{cor}

The upper spectral dimension is stable in the following sense.
\begin{cor}
\label{cor:Sum_of_Measures} For two finite Borel measures $\nu_{1}$
and $\nu_{2}$ on $\left(0,1\right)$ with upper spectral dimension
$\overline{s}_{\nu_{1}}$ and $\overline{s}_{\nu_{2}}$, we have that
the upper spectral dimension of $\nu_{1}+\nu_{2}$ is given by $\max\left\{ \overline{s}_{\nu_{1}},\overline{s}_{\nu_{2}}\right\} $.
If additionally, $\underline{s}_{\nu_{1}}=\overline{s}_{\nu_{1}}\geq\overline{s}_{\nu_{2}}$,
then its spectral dimension exists and equals $\overline{s}_{\nu_{1}}$.
If $\underline{s}_{\nu_{1}}=\overline{s}_{\nu_{1}}\geq\overline{s}_{\nu_{2}}$
and $\nu_{1}$ is regular, then so is $\nu_{1}+\nu_{2}$.
\end{cor}

Moreover, the spectral dimension (and in fact also finer spectral
asymptotic properties) is stable under bi-Lipschitz mappings.
\begin{prop}
\label{Prop:bi_Lipschitz_trafo}Let $\nu$ be a Borel measure on $(0,1)$.
Further let $g:\left[0,1\right]\to\left[a,b\right]$ be a bi-Lipschitz
mapping. Then uniformly for $i\in\N$, we have
\[
\lambda_{\nu,[0,1]}^{i}\asymp\lambda_{\nu\circ g^{-1},[a,b]}^{i}.
\]
\end{prop}

In the last section we give a sequence of examples showing how our
approach allows to analyze in detail particular classes of measures:
First, we determine the spectral dimension of measures that stem from
\emph{non-linear iterated function systems} (with or without overlap),
see Section \ref{sec:Applicaiton-to-self-conformal}.
\begin{thm}
\label{thm:self-conforma-casel}For any $C^{1}$-self-conformal measure
$\nu$, the spectral dimension of $\Delta_{\nu}$ exists and equals
$\overline{q}_{\nu}.$
\end{thm}

Second, we study\emph{ homogeneous Cantor measure}, for which we provide
examples such that the spectral dimension does not exist with $\overline{s}_{\nu}=\overline{h}_{\nu}>\underline{h}_{\nu}=\underline{s}_{\nu}$
and an other example with spectral dimension $0$. Finally, we consider
\emph{pure point measure} such that its spectral dimension exists
and attains any given value in $[0,1/2]$.

We end this introduction with some concluding remarks. Our newly established
connection to the $L^{q}$-spectrum puts the results of \citep{MR118891},
\citep{Fu87}, \citep{MR1328700} and \citep{MR2828537,MR3809018,MR3897401,MR4176086}
from an abstract point of view in the right context; for the self-similar
measure $\nu$ under the OSC, the $L^{q}$-spectrum of $\nu$ can
be expressed in terms of the associated pressure function. In the
case of self-conformal measures with overlaps the situation becomes
much more involved; nevertheless, by a result in \citep{Barral2020},
the $L^{q}$-spectrum will also exist in this situation and some explicit
formulae can be provided in the self-similar setting (see \citep{KN21}).

We remark that the operator $\Delta_{\nu,[a,b]}$ has been generalized
by considering the form $(f,g)\mapsto\int_{[a,b]}\nabla_{\mu}f\nabla_{\mu}g\d\mu$
where $\mu$ is an atomless Borel measure on $[a,b]$ with and $\nabla_{\mu}f$
denotes the weak derivative with respect to $\mu$, i.~e\@. $f(x)=f(a)+\int_{[a,x]}\nabla_{\mu}f\d\mu,\:x\in[a,b].$
The associated generalized Kre\u{\i}n-Feller operator $\Delta_{\nu,\mu,\left[a,b\right]}$
and its spectral properties have been studied in \citep{kuechler1986},
\citep{Volkmer05}, \citep{MR2787628,MR2407555,MR2110540,FZ02}. In
\citep{KSW2019} it has been shown that by a simple transformation
of measure spaces by virtue of the distribution function $F_{\mu}$
of $\mu$ the generalized Kre\u{\i}n-Feller operator can be reduced
to the classical one, namely $\Delta_{\nu\circ F_{\mu}^{-1},\Lambda,\left[F_{\mu}(a),F_{\mu}(b)\right]}$,
which has the same spectral properties. For this reason we have restricted
our attention to the case $\mu=\Lambda$.

Finally, we would like to note that the operator $\Delta_{\nu}$ has
also been generalized to higher dimensions in \citep{MR1298682,MR1338787,MR1484417,MR1721826,MR1839473,MR2261337,MR3318648,MR4241300,Ngai_2021}.
In the forthcoming paper \citep{KN2022} we elaborate the ideas of
the present  paper also for higher dimensions. 

Furthermore, the operator $\Delta_{\nu}$ can be generalized to polyharmonic
operators in higher dimensions which has been considered in \citep{MR0482138,MR0278126}.
Polyharmonic operators with respect to self-similar measures under
the OSC have been studied in \citep{MR2092208} using the same approach
as in \citep{MR1328700}. We would like to point out that our variational
approach can also be used to solve the spectral problem in these higher
dimensions as well as for polyharmonic operators; this will be the
content of the forthcoming paper \citep{KN21b}, in which we improve
known bounds on the spectral asymptotics for polyharmonic operators
with respect to any Borel measures on $(0,1)^{d}$, $d\in\N$.

\section{Spectral theory for Kre\u{\i}n-Feller operators}

In this chapter, we develop a form approach for the Kre\u{\i}n-Feller
operator that allows us to rigorously prove the sub- and superadditivity
of the eigenvalue counting function. In doing so, we follow various
ideas as found, for example, in \citep{MR2563669,MR2892328,MR2261337}.
Crucially for our purposes, and in contrast to previous work, we consider
ambient intervals for which the boundary points are not necessarily
in the support of the underlying measure, which moreover may have
a non-vanishing pure point part. As this setup is not covered in the
literature, we decided to provide the details necessary for our purpose.
In particular, we give as a key observation Proposition \ref{prop:dom_vs_H01 Minmax},
which indeed seems to be used implicitly throughout the existing literature.

\subsection{Form approach}

For a fixed finite Borel measure $\nu$ on $(a,b)$ let us define
\[
C_{\nu}([a,b])\coloneqq\left\{ f\in C([a,b])\mid f\:\text{\text{is}\:affine\:linear\:on\:the\:components\:of\:}[a,b]\setminus\supp\nu\right\} ,
\]
where $C([a,b])$ denotes the space of continuous real-valued functions
on $[a,b]$. Note that $\left(C_{\nu}([a,b]),\left\Vert \cdot\right\Vert _{\infty}\right)$
is a Banach space. Further, we define the \emph{Sobolev space }of
weakly differentiable functions
\[
H^{1}(a,b)\coloneqq\left\{ f:[a,b]\rightarrow\R\mid\exists g\in L_{\Lambda}^{2}\left(\left[a,b\right]\right):f(x)=f(a)+\int_{[a,x]}g\d\Lambda,x\in[a,b]\right\} 
\]
together with its closed subspace 
\[
H_{0}^{1}(a,b)\coloneqq\left\{ u\in H^{1}(a,b)\mid u(a)=u(b)=0\right\} .
\]
The function $g$ in the definition of $H^{1}\left(a,b\right)$ is
uniquely determined by $f$ and we write $\nabla_{\Lambda}f\coloneqq g$.
The space $H_{0}^{1}(a,b)$ with the bilinear form 
\[
(u,v)_{H_{0}^{1}}\coloneqq\int_{(a,b)}uv\d\Lambda+\int_{(a,b)}\nabla_{\Lambda}f\nabla_{\Lambda}g\d\Lambda
\]
 defines a Hilbert space. We will see that 
\[
\E_{\nu,[a,b]}(f,g)\coloneqq\int_{(a,b)}\nabla_{\Lambda}f\nabla_{\Lambda}g\d\Lambda
\]
with (Dirichlet) domain $\dom(\E_{\nu})\coloneqq H_{0}^{1}(a,b)\cap C_{\nu}([a,b])$
defines a form for $L_{\nu}^{2}$. First note that $\E_{\nu,[a,b]}$
is well defined. Indeed, for $f\in\dom(\E_{\nu,[a,b]})$ with $f=0$
$\nu$-a.~e\@. we have that $f$ vanishes on $\supp\nu$ by continuity.
Moreover, by definition of $\dom(\E_{\nu,[a,b]})$, we have that $f$
also vanishes on the complement of $\supp\nu$. Therefore, for all
$g\in\dom(\E_{\nu,[a,b]})$, it follows that $\E_{\nu,[a,b]}(f,g)=0.$
In the following we need the decomposition
\[
[a,b]\setminus\supp\nu=A_{1}\cup A_{2}\cup\bigcup_{i\in I}(a_{i},b_{i}),
\]
where $I\subset\N$ and $A_{1}\coloneqq[a,d_{1})$ if $a\notin\supp\nu$
otherwise $A_{1}=\emptyset$, $A_{2}\coloneqq(d_{2},b]$ if $b\notin\supp\nu$
and the intervals $[a,c_{1})$, $(c_{2},b]$, $(a_{i},b_{i})$, $i\in I$,
are mutually disjoint.
\begin{lem}
\label{lem:RepraesentantL2NU}The map $\tau_{\nu,\left[a,b\right]}:H_{0}^{1}(a,b)\rightarrow\dom\left(\E_{\nu}\right)$
\[
\tau_{\nu,\left[a,b\right]}(f)(x)\coloneqq\begin{cases}
f(a_{i})+\frac{f(b_{i})-f(a_{i})}{b_{i}-a_{i}}\left(x-a_{i}\right), & x\in(a_{i},b_{i}),\:i\in I,\\
\frac{f(d_{1})}{d_{1}-a}(x-a), & x\in[a,d_{1}),\:a\notin\supp\nu,\\
\frac{f(d_{2})}{b-d_{2}}\left(b-x\right), & x\in(d_{2},b],\:b\notin\supp\nu.\\
f(x), & x\in\supp\nu
\end{cases}
\]
is surjective, $\tau_{\nu,\left[a,b\right]}(f)=f$ as elements of
$L_{\nu}^{2}$, and we have
\[
\nabla_{\Lambda|_{[a,b]}}\tau_{\nu,\left[a,b\right]}(f)(x)=\begin{cases}
\frac{f(b_{i})-f(a_{i})}{b_{i}-a_{i}}, & x\in(a_{i},b_{i}),\:i\in I,\\
\frac{f(d_{1})}{d_{1}-a}, & x\in[a,d_{1}),\:a\notin\supp\nu,\\
-\frac{f(d_{2})}{b-d_{2}}, & x\in(d_{2},b],\:b\notin\supp\nu,\\
\nabla_{\Lambda}f(x), & x\in\supp\nu.
\end{cases}
\]
\end{lem}

\begin{proof}
We follow the arguments as outlined in \citep{MR2563669}. For $f\in H_{0}^{1}(a,b)$,
we set\textbf{ }
\[
\tau_{n}(f)(x)\coloneqq\begin{cases}
f(a_{i})+\frac{f(b_{i})-f(a_{i})}{b_{i}-a_{i}}\left(x-a_{i}\right), & x\in(a_{i},b_{i}),\:i\in I,i\leq n\\
\frac{f(d_{1})}{d_{1}-a}(x-a), & x\in[a,d_{1}),\:a\notin\supp\nu,\\
\frac{f(d_{2})}{b-d_{2}}\left(b-x\right), & x\in(d_{2},b],\:b\notin\supp\nu,\\
f(x), & \text{otherwise},
\end{cases}
\]
and choose a $L_{\Lambda}^{2}$ representative of $\nabla_{\Lambda}f$
to define
\[
g_{n}(x)\coloneqq\begin{cases}
\frac{f(b_{i})-f(a_{i})}{b_{i}-a_{i}}, & x\in(a_{i},b_{i}),\:i\in I,i\leq n\\
\frac{f(d_{1})}{d_{1}-a}, & x\in[a,d_{1}),\:a\notin\supp\nu,\\
-\frac{f(d_{2})}{b-d_{2}}, & x\in(d_{2},b],\:b\notin\supp\nu,\\
\nabla_{\Lambda}f(x), & \text{otherwise}.
\end{cases}
\]
Then we have $\tau_{n}\left(f\right)\in H_{0}^{1}(a,b)$ with $\nabla_{\Lambda}\tau_{n}(f)=g_{n}$
a.~e\@. and clearly we have $\tau_{n}(f)=f$ in $L_{\nu}^{2}$.
Since $f$ is uniformly continuous on the compact set $[a,b]$, for
$n$ tending to infinity, we have $\tau_{n}(f)\rightarrow\tau_{\nu,\left[a,b\right]}(f)\in C_{\nu}([a,b])$
uniformly on $[a,b]$. Now, for each $x\in[a,b]$, set $g(x)\coloneqq\lim_{n\rightarrow\infty}g_{n}(x).$
Then by Cauchy-Schwarz inequality, we have 
\begin{align*}
\int g^{2}\d\Lambda & =\sum_{i\in I}\int_{(a_{i},b_{i})}g^{2}\d\Lambda+\int_{\supp\nu}g^{2}\d\Lambda=\frac{\left(f(b_{i})-f(a_{i})\right)^{2}}{b_{i}-a_{i}}+\int_{\supp\nu}\left(\nabla_{\Lambda}f\right)^{2}\d\Lambda\\
 & =\sum_{i\in I}\frac{\left(\int_{(a_{i},b_{i})}\nabla_{\Lambda}f\d\Lambda\right)^{2}}{\left(b_{i}-a_{i}\right)}+\int_{\supp\nu}\left(\nabla_{\Lambda}f\right)^{2}\d\Lambda\\
 & \leq\sum_{i\in I}\int_{(a_{i},b_{i})}\left(\nabla_{\Lambda}f\right)^{2}\d\Lambda+\int_{\supp\nu}\left(\nabla_{\Lambda}f\right)^{2}\d\Lambda=\int\left(\nabla_{\Lambda}f\right)^{2}\d\Lambda<\infty
\end{align*}
implying $g\in L_{\Lambda|_{(a,b)}}^{2}.$ Furthermore, for all $n\in\N$,
we have
\[
\left|g_{n}(x)\right|^{2}\leq\left(\left|g_{n}(x)\right|+\left|\nabla_{\Lambda}f(x)\right|\right)^{2}\leq4\left(\left|g_{n}(x)\right|^{2}+\left|\nabla_{\Lambda}f(x)\right|^{2}\right)
\]
and by the dominated convergence theorem we conclude $g_{n}\rightarrow g$
in $L_{\Lambda|_{(a,b)}}^{2}.$ This implies $\tau_{\nu,\left[a,b\right]}(f)\in\dom\left(\E_{\nu}\right)$
and $\nabla_{\Lambda}\tau_{\nu,\left[a,b\right]}(f)=g.$
\end{proof}
\begin{prop}
\label{prop:DenseRichtig}The set $\dom(\E_{\nu,[a,b]})$ is dense
in $L_{\nu}^{2}$.
\end{prop}

\begin{proof}
Here we follow the arguments similar to \citep{MR2563669}. Let $C_{c}(a,b)$
denote the space of continuous functions from $(a,b)$ to $\R$ with
compact support and let $C_{c}^{1}(a,b)$ be the space of differentiable
functions from $(a,b)$ to $\R$ with compact support. By Lemma \ref{lem:RepraesentantL2NU}
every $f\in C_{c}^{1}\left(a,b\right)$ possesses a representative
$\tau_{\nu,\left[a,b\right]}(f)\in\dom(\E_{\nu,[a,b]})$, using $f(a)=f(b)=0.$
It is well known that $C_{c}^{1}\left(a,b\right)$ lies dense in $C_{c}\left(a,b\right)$
with respect to the uniform norm and $C_{c}\left(a,b\right)$ in $L_{\nu}^{2}$,
hence we obtain that $\dom(\E_{\nu,[a,b]})$ is dense in $L_{\nu}^{2}.$
\end{proof}
\begin{prop}
\label{Prop:HilbertspaceRichtig}The set $\dom(\E_{\nu})$ equipped
with the inner product 
\[
(f,g)_{\E_{\nu}}\coloneqq\langle f,g\rangle_{\nu}+\E_{\nu}(f,g),
\]
 defines a Hilbert space, i.~e\@. $\E_{\nu}$ is closed with respect
to $L_{\nu}^{2}$.
\end{prop}

\begin{proof}
Here, we follow ideas of \citep[Lemma 6.4.]{MR2892328}. Let $(u_{n})_{n\in\N}$
be a Cauchy sequence in $\dom(\E_{\nu,[a,b]})$ with respect to the
norm induced by $(\cdot,\cdot)_{\E_{\nu,[a,b]}}$, i.~e\@. 
\[
\left\Vert u_{n}-u_{m}\right\Vert _{L_{\nu}^{2}}^{2}+\left\Vert \nabla_{\Lambda}u_{n}-\nabla_{\Lambda}u_{m}\right\Vert _{L_{\Lambda}^{2}}^{2}\rightarrow0
\]
for $n,m\rightarrow\infty.$ Since $L_{\nu}^{2}$ and $L_{\Lambda}^{2}$
are Hilbert spaces, there exist $u\in L_{\nu}^{2}$ and $f\in L_{\Lambda}^{2}$
such that
\[
\lim_{n\rightarrow\infty}\left\Vert u_{n}-u\right\Vert _{L_{\nu}^{2}}=0\text{\,and\,}\lim_{n\rightarrow\infty}\left\Vert f-\nabla_{\Lambda}u_{n}\right\Vert _{L_{\Lambda}^{2}}=0.
\]
Using Cauchy-Schwarz inequality, we find uniformly for all $x\in[a,b]$,
\begin{align*}
\left|u_{n}(x)-u_{n}(a)-\underbrace{\int_{[a,x]}f(z)\d\Lambda(z)}_{\eqqcolon g(x)}\right| & =\left|\int_{[a,x]}\left(f(z)-\nabla_{\Lambda}u_{n}(z)\right)\d\Lambda(z)\right|\\
 & \leq\sqrt{\Lambda([a,b])}\left\Vert \nabla_{\Lambda}u_{n}-f\right\Vert _{L_{\Lambda}^{2}}\rightarrow0,
\end{align*}
for $n\rightarrow\infty.$ In particular, $u_{n}-u_{n}(a)\rightarrow g$
in $L_{\nu}^{1}$ and therefore $u_{n}(a)=\left(u_{n}(a)-u_{n}\right)+u_{n}\rightarrow u-g$
in $L_{\nu}^{1}$. But this shows that the sequence $(u_{n}(a))_{n\in\N}$
converge to $u_{0}\coloneqq\nu([a,b])^{-1}\int u-g\d\nu.$ Hence,
we conclude that $u_{n}$ convergence uniformly to $h\coloneqq u_{0}+g\in H^{1}(a,b)$
with $\nabla_{\Lambda}h=f$, in particular, uniform convergence guarantees
that $h(a)=h(b)=0$ and $h\in C_{\nu}([a,b])$.
\end{proof}
\begin{lem}
\label{lem:ContinuousEmbedding}There exists a constant $C$ such
that for every $f\in H^{1}(a,b)$
\[
\left\Vert f\right\Vert _{\infty}\leq C\left(\left\Vert \nabla_{\Lambda}f\right\Vert _{L_{\Lambda}^{2}(a,b)}^{2}+\left\Vert f\right\Vert _{L_{\nu}^{2}}^{2}\right)^{1/2}
\]
 and, for all $f\in H_{0}^{1}(a,b)$, 
\[
\int_{(a,b)}f^{2}\d\nu\leq\nu\left(\left[a,b\right]\right)(b-a)\int_{(a,b)}\left(\nabla_{\Lambda}f\right)^{2}\d\Lambda.
\]
\end{lem}

\begin{proof}
The first inequality has been proved in \citep[Lemma 1.4.]{MR2563669}.
Further, for fixed $f\in H_{0}^{1}(a,b)$, we have
\begin{align*}
\int_{(a,b)}f(x)^{2}\d\nu(x) & =\int\left(\left(\int_{(a,x)}\left(\nabla_{\Lambda}f\right)\d\Lambda\right)^{2}\right)\d\nu(x)\\
 & \leq\int\left((x-a)\int_{(a,x)}\left(\nabla_{\Lambda}f\right)^{2}\d\Lambda\right)\d\nu(x)\\
 & \leq(b-a)\nu\left(\left[a,b\right]\right)\int_{(a,b)}\left(\nabla_{\Lambda}f\right)^{2}\d\Lambda.
\end{align*}
\end{proof}
\begin{rem}
Due to the above lemma, we can equivalently consider $\dom\left(\mathcal{E}_{\nu}\right)$
as a Hilbert space whose inner product is given by $\mathcal{E}_{\nu}$.
\end{rem}

\begin{prop}
\label{Prop:CompactinclusionRichtig} The inclusion from the Hilbert
space $\left(\dom\left(\E_{\nu}\right),\mathcal{E}_{\nu}\right)$
into $L_{\nu}^{2}$ is compact.
\end{prop}

\begin{proof}
Let $U\subset\dom\left(\E_{\nu,[a,b]}\right)$ be a bounded set with
respect to $(\,,\,)_{\E_{\nu,[a,b]}}$ which means there exists $c>0$
such that
\[
\sup_{f\in U}\left(\int_{(a,b)}\left(\nabla_{\Lambda}f\right)^{2}\d\Lambda\right)^{1/2}\leq c.
\]
We have to show $U$ is pre-compact in $L_{\nu}^{2}.$ This will be
done by showing that $U$ is bounded in $C([a,b])$ with respect to
$\left\Vert \cdot\right\Vert _{\infty}$ and equicontinuous. The Arzelà--Ascoli
Theorem then guarantees that $U$ is pre-compact in $C([a,b])$. Hence
$C([a,b])$ is continuously embedded in $L_{\nu}^{2}$ noting that
$\left\Vert f\right\Vert _{L_{\nu}^{2}}\leq\nu([a,b])\left\Vert f\right\Vert _{\infty}$,
for all $f\in C([a,b])$. The set $U$ is also pre-compact in $L_{\nu}^{2}$.
Using the Cauchy-Schwarz inequality, for all $x,y\in[a,b]$, we have
\begin{align*}
\left|f(x)-f(y)\right| & =\left|\int_{(x,y)}\nabla_{\Lambda}f\d\Lambda\right|\leq\left|x-y\right|^{1/2}\left(\,\int_{(a,b)}\left(\nabla_{\Lambda}f\right)^{2}\d\Lambda\right)^{1/2}\leq c\left|x-y\right|^{1/2}.
\end{align*}
Hence, $U$ is equicontinuous. Lemma \ref{lem:ContinuousEmbedding}
implies that $U$ is also bounded in $\left(C([a,b]),\left\Vert \cdot\right\Vert _{\infty}\right).$
\end{proof}
If $\nu\left(\left\{ a,b\right\} \right)=0$ then Proposition \ref{Prop:HilbertspaceRichtig}
and Proposition \ref{prop:DenseRichtig} show that the densely defined
form $\E_{\nu,[a,b]}$ is closed. Hence, by \citep[Theorem 4.4.2]{davies_1995}
or \citep[Theorem 2.1., p. 322]{MR0407617} there exists a non-negative
self-adjoint operator $\Delta_{\nu,[a,b]}$ which is characterized
as follows; $f\in\dom(\E_{\nu,[a,b]})$ lies in the domain $\dom\left(\Delta_{\nu,[a,b]}\right)$
of the Laplacian $\Delta_{\nu,[a,b]}$, if and only if there exists
$h\in L_{\nu}^{2}$ such that
\[
\E_{\nu,[a,b]}(g,f)=\langle g,h\rangle_{\nu},\;\;\mbox{ for all }\;\;g\in\dom(\E_{\nu,[a,b]})
\]
with $h\eqqcolon\Delta_{\nu,\left[a,b\right]}f\eqqcolon\Delta_{\nu}f.$

Observe that $H_{0}^{1}(a,b)$ is compactly embedded in $L_{\nu}^{2}$,
consequently the closed subspace $\dom(\E_{\nu})$ is also compactly
embedded in $L_{\nu}^{2}.$ Consequently, we conclude (see \citep[Theorem 4.5.1. and p. 258]{MR1193032})
that there exists an orthonormal system of eigenfunctions of $\Delta_{\nu}$
of $L_{\nu}^{2}$ with non-negative eigenvalues $\left(\lambda_{\nu}^{n}\right)_{n\in\N}$
tending to $\infty$ given $\supp\nu$ is not finite.

\subsection{Spectral problem, min-max principle, and sub- and superadditivity}
\begin{defn}
\label{def:WeakEV and EF} An element $f\in\dom(\E_{\nu})\setminus\left\{ 0\right\} $
is called \emph{eigenfunction }of $\E_{\nu}$ with \emph{eigenvalue}
$\lambda$, if for all $g\in\dom(\E_{\nu})$, we have
\[
\E_{\nu}(f,g)=\lambda\cdot\langle f,g\rangle_{\nu}.
\]
We denote the number of eigenvalues of $\E_{\nu}$ not exceeding $x$
by $N_{\nu}\left(x\right)$ and refer to $N_{\nu}=N_{\nu,[a,b]}$
as the \emph{eigenvalue counting function}.
\end{defn}

\begin{rem}
Note that for any fixed $c>0$, $\lambda$ is eigenvalue of $\E_{\nu}$
if and only if $c\lambda$ is an eigenvalue of $\E_{c\nu}$. Therefore,
the spectral dimension does not change after rescaling the measure
and we may assume that $\nu$ is Borel probability measure.
\end{rem}

\begin{rem}
Let \textbf{$\E$} be a closed form with domain \textbf{$\dom\left(\E\right)$}
densely defined on $L_{\nu}^{2}$, in particular, \textbf{$\dom\left(\E\right)$
}defines a Hilbert space with respect to $(f,g)_{\E}\coloneqq\left\langle f,g\right\rangle {}_{\nu}+\E(f,g)$,
and assume that the inclusion from $\left(\dom\left(\E\right),\left(\cdot,\cdot\right)_{\E}\right)$
into $L_{\nu}^{2}$ is compact. Then the \emph{Poincaré--Courant--Fischer--Weyl
min-max principle} is applicable, that is for the $i$-th eigenvalue
$\lambda^{i}\left(\E\right)$ of \textbf{$\E$}, $i\in\N$, we have
(see also \citep{davies_1995,MR1243717}) 
\begin{align*}
\lambda^{i}\left(\E\right) & =\inf\left\{ \sup\left\{ \frac{\E(\psi,\psi)}{\langle\psi,\psi\rangle_{\nu}}:\psi\in G^{\star}\right\} \colon G<_{i}\dom\left(\E\right)\right\} ,
\end{align*}
where $J<_{i}H$ means that $J$ is a linear subspace of the Hilbert
space $H$ and the vector space dimension of $J$ is equal to $i\in\N$
and we set $J^{\star}\coloneqq J\setminus\left\{ 0\right\} $. 
\end{rem}

Let $P\coloneqq\left\{ \left(c_{k},c_{k+1}\right]\colon k=1,\ldots,n-1\right\} $
a $\nu$-partition of $\left(a,b\right]$, $n\in\N_{\geq2}$, $a=c_{1}<c_{2}\cdots<c_{n}=b$.
For a subinterval $\left(c_{k},c_{k+1}\right)$ we may define the
restricted form $\E_{\nu,\left[c_{k},c_{k+1}\right]}$ on $L_{\nu|_{\left[c_{k},c_{k+1}\right]}}^{2}$
with domain $\dom\left(\E_{\nu,\left[c_{k},c_{k+1}\right]}\right)\coloneqq H_{0}^{1}\left(c_{k},c_{k+1}\right)\cap C_{\nu}\left(c_{k},c_{k+1}\right)$
by 
\[
\E_{\nu,\left[c_{k},c_{k+1}\right]}\left(f,g\right)=\int_{c_{k}}^{c_{k+1}}\nabla_{\Lambda}f\nabla_{\Lambda}g\d\Lambda.
\]
Consider the natural direct sum Hilbert space
\[
L_{\nu,P}^{2}\coloneqq\bigoplus_{k=1}^{n}L_{\nu|_{\left[c_{k},c_{k+1}\right]}}^{2}
\]
with inner product
\[
\left\langle \left(f_{k}\right),\left(g_{k}\right)\right\rangle _{\nu,P}\coloneqq\sum_{k=1}^{n}\left\langle f_{k},g_{k}\right\rangle _{\nu|_{\left[c_{k},c_{k+1}\right]}}
\]
and corresponding Sobolev spaces
\[
H_{P}^{1}\coloneqq\bigoplus_{k=1}^{n}H_{0}^{1}\left(c_{k},c_{k+1}\right).
\]
 We define the \emph{direct sum form }
\[
\E_{\nu,P}\left(\left(f_{k}\right),\left(g_{k}\right)\right)\coloneqq\sum_{k=1}^{n}\E_{\nu,\left[c_{k},c_{k+1}\right]}\left(f_{k},g_{k}\right)
\]
with domain 
\[
\dom\left(\E_{\nu,P}\right)\coloneqq\bigoplus_{k=1}^{n}\dom\left(\E_{\nu,\left[c_{k},c_{k+1}\right]}\right).
\]
Both $\dom\left(\E_{\nu,P}\right)$ and $H_{P}^{1}$ are Hilbert spaces
with respect to the inner product given by $\E_{\nu,P}$.
\begin{lem}
The form $\E_{\nu,P}$ is densely defined and closed.
\end{lem}

\begin{proof}
This follows from Proposition \ref{Prop:HilbertspaceRichtig} and
\ref{prop:DenseRichtig} applied to each component.
\end{proof}
The eigenvalue counting function with respect to $\E_{\nu,P}$ will
be denoted by $N_{\nu,P}$.

\begin{lem}
\label{lem:subdivision_EV_Counting function}For a $\nu$-partition
$P\coloneqq\left\{ \left(c_{k},c_{k+1}\right]\colon k=1,\ldots,n-1\right\} $
and all $x>0$, we have 
\[
N_{\nu,P}\left(x\right)=\text{\ensuremath{\sum_{k=1}^{n}N_{\nu,\left[c_{k},c_{k+1}\right]}}}\left(x\right).
\]
\end{lem}

\begin{proof}
Let $(f_{k})_{k}\in\dom\left(\E_{\nu,P}\right)$ be an eigenfunction
with eigenvalue $\lambda>0$, i.~e\@. for $(g_{k})_{k}\in\dom\left(\E_{\nu,P}\right)$
\[
\sum_{k=1}^{n}\E_{\nu,\left[c_{k},c_{k+1}\right]}\left(f_{k},g_{k}\right)=\lambda\sum_{k=1}^{n}\left\langle f_{k},g_{k}\right\rangle _{\nu|_{\left[c_{k},c_{k+1}\right]}}.
\]
Hence, if $f_{k}\neq0$ it follows all $g\in\dom\left(\E_{\nu|_{\left[c_{k},c_{k+1}\right]}}\right)$
that $(g_{k})_{k}=(0,\dots,0,g,0,\dots,0)\in\dom\left(\E_{\nu,P}\right)$
\[
\E_{\nu,\left[c_{k},c_{k+1}\right]}\left(f_{k},g\right)=\lambda\left\langle f_{k},g\right\rangle _{\nu|_{\left[c_{k},c_{k+1}\right]}}.
\]
This shows that $f_{k}$ is an eigenfunction of $\E_{\nu,\left[c_{k},c_{k+1}\right]}$
with eigenvalue $\lambda$, which implies 
\[
N_{\nu,P}\left(x\right)\leq\text{\ensuremath{\sum_{k=1}^{n}N_{\nu,\left[c_{k},c_{k+1}\right]}}}\left(x\right).
\]
In a similar way it follows $N_{\nu,P}\left(x\right)\geq\text{\ensuremath{\sum_{k=1}^{n}N_{\nu,\left[c_{k},c_{k+1}\right]}}}\left(x\right).$
\end{proof}
\begin{lem}
\label{lem:Dom=00003DorthgonalComplement}For a $\nu$-partition $P\coloneqq\left\{ \left(c_{k},c_{k+1}\right]\colon k=1,\ldots,n-1\right\} $
we have 
\begin{align*}
\dom\left(\E_{\nu,P}\right)^{\perp} & \coloneqq\left\{ f\in H_{P}^{1}\colon\forall g\in\dom\left(\E_{\nu,P}\right):\E_{\nu,P}(f,g)=0\right\} =\left\{ f\in H_{P}^{1}\colon\left\Vert f\right\Vert _{L_{\nu,P}^{2}}=0\right\} .
\end{align*}
\end{lem}

\begin{proof}
Pick $f\coloneqq\left(f_{j}\right)_{j=1,\dots,n}\in\left(\bigoplus_{k=1}^{n}\dom\left(\E_{\nu,\left[c_{k},c_{k+1}\right]}\right)\right)^{\perp}$
and $i\in\left\{ 1,\dots,n\right\} $ with $\nu(\left[c_{i},c_{i+1}\right])>\text{0}$.
Then define for $x\in\supp\nu\cap\left[c_{i},c_{i+1}\right]$
\[
g_{x}(y)\coloneqq f_{i}(x)\frac{(y-c_{i})}{x-c_{i}}\1_{[c_{i},x]}+f_{i}(x)\frac{c_{i+1}-y}{c_{i+1}-x}\1_{(x,c_{i+1}]}.
\]
Hence, 
\begin{align*}
0 & =\E_{\nu,\left[c_{i},c_{i+1}\right]}(f_{i},g_{x})=\int_{[c_{i},x]}\nabla_{\Lambda}f_{i}\nabla_{\Lambda}g_{x}\d\Lambda+\int_{(x,c_{i+1}]}\nabla_{\Lambda}f_{i}\nabla_{\Lambda}g_{x}\d\Lambda\\
 & =\frac{f_{i}(x)}{x-c_{i}}(f_{i}(x)-f_{i}(c_{i}))+\frac{-f_{i}(x)}{c_{i+1}-x}(f_{i}(c_{i+1})-f_{i}(x))=\frac{f_{i}(x)^{2}}{x-c_{i}}+\frac{f_{i}(x)^{2}}{c_{i+1}-x},
\end{align*}
and consequently, for all $x\in\supp\nu\cap(c_{i},c_{i+1}),$ we have
$f_{i}(x)=0$; in particular, $\left\Vert f\right\Vert _{L_{\nu,P}^{2}}=0$.
On the other hand, for $f\in\left\{ g\in H_{P}^{1}:\left\Vert g\right\Vert _{L_{\nu,P}^{2}}=0\right\} $
and each $k\in\left\{ 1,\ldots,n\right\} $, $f_{k}$ vanishes $\nu|_{\left(c_{k},c_{k+1}\right)}$-a.\,e\@.
and, using the continuity of $f_{k}$, we obtain $f_{k}=0$ on $\supp\nu|_{\left[c_{i},c_{i+1}\right]}.$
To simplify notation, assume $c_{k},c_{k+1}\in\supp\nu.$ Now, with
$\bigcup_{i\in I}\left(a_{i},b_{i}\right)=\left[c_{k},c_{k+1}\right]\setminus\supp\nu|_{\left[c_{k},c_{k+1}\right]}$,
one easily verifies that $\nabla_{\Lambda}f_{k}=\sum_{i\in I}\1_{(a_{i},b_{i})}\nabla_{\text{\ensuremath{\Lambda}}}f_{k}\in L_{\Lambda}^{2}$
and we obtain for all $g\in\dom\left(\E_{\nu,\left[c_{k},c_{k+1}\right]}\right)$,
\begin{align*}
\E_{\nu,\left[c_{k},c_{k+1}\right]}(f_{k},g) & =\int_{\left[c_{k},c_{k+1}\right]}\nabla_{\Lambda}f_{k}\nabla_{\Lambda}g\d\Lambda\\
 & =\sum_{i\in I}\int_{(a_{i},b_{i})}\nabla_{\Lambda}f_{k}\nabla_{\Lambda}g\d\Lambda\\
 & =\sum_{i\in I}\frac{g(b_{i})-g(a_{i})}{b_{i}-a_{i}}(f_{k}(b_{i})-f_{k}(a_{i}))=0.
\end{align*}
Since $k$ was arbitrary the definition of $\E_{\nu,P}$ gives $f\text{\ensuremath{\in}}\dom(\E_{\nu,P}){}^{\perp}$.
\end{proof}
The following extended min-max principle will be crucial for the proof
of sub- and superadditivity, and it seems to have been used implicitly
in earlier work along these lines.
\begin{prop}
\label{prop:dom_vs_H01 Minmax}For a $\nu$-partition $P\coloneqq\left((c_{k},c_{k+1}]\colon k=1,\ldots,n\right)$
and all $i\in\N$, we have
\begin{align*}
\lambda_{i}\left(\E_{\nu,P}\right) & =\inf\left\{ \sup\left\{ \frac{\E_{\nu,P}(\psi,\psi)}{\langle\psi,\psi\rangle_{\nu,P}}:\psi\in G^{\star}\right\} \colon G<_{i}\dom\left(\E_{\nu,P}\right)\right\} \\
 & =\inf\left\{ \sup\left\{ \frac{\E_{\nu,P}(\psi,\psi)}{\langle\psi,\psi\rangle_{\nu,P}}:\psi\in G^{\star}\right\} \colon G<_{i}H_{P}^{1}\right\} .
\end{align*}
\end{prop}

\begin{proof}
The part `$\geq$' follows from the inclusion $\mathcal{D}\coloneqq\dom\left(\E_{\nu,P}\right)=\bigoplus_{k=1}^{n}\dom\E_{\nu,\left[c_{k},c_{k+1}\right]}\subset\bigoplus_{k=1}^{n}H_{0}^{1}\left(c_{k},c_{k+1}\right)$.
For the reverse inequality we consider an $i$-dimensional subspace
$G=\spann(f_{1},\dots,f_{i})\subset\bigoplus_{k=1}^{n}H_{0}^{1}\left(c_{k},c_{k+1}\right)$.
There exists a unique decomposition $f_{j}=f_{1,j}+f_{2,j}$ with
$f_{1,j}\in\mathcal{D}$ and $f_{2,j}\in\mathcal{D}{}^{\perp}$, $j=1,\ldots,i$.
Suppose that $\left(f_{1,j}\right)_{j}$ are not linearly independent,
then there exists a non-zero element $g\in G\cap\mathcal{D}{}^{\perp}$.
To see this fix $(\lambda_{1},\dots,\lambda_{n})\neq(0,\dots,0)$
with $\lambda_{1}f_{1,1}+\dots+\lambda_{i}f_{1,i}=0.$ Then
\begin{align*}
\underbrace{\lambda_{1}\left(f_{1,1}+f_{1,2}\right)+\dots+\lambda_{i}\left(f_{1,i}+f_{2,i}\right)}_{\in G^{\star}}=\underbrace{\lambda_{1}f_{2,1}+\dots+\lambda_{i}f_{2,i}}_{\in\dom(\E_{\nu,P})^{\perp}} & \eqqcolon g.
\end{align*}
Using $\mathcal{D}{}^{\perp}\subset\left\{ f\in H_{P}^{1}\colon\left\Vert f\right\Vert _{L_{\nu,P}^{2}}=0\right\} $
by Lemma \ref{lem:Dom=00003DorthgonalComplement} and $\E_{\nu,P}(g,g)>0$,
we get in this case
\[
\sup\left\{ \frac{\E_{\nu,P}(\psi,\psi)}{\langle\psi,\psi\rangle_{\nu,P}}:\psi\in G^{\star}\right\} =\infty.
\]
Otherwise, for every vector $\left(a_{j}\right)\in\R^{i}\setminus\left\{ 0\right\} $,
we have 
\begin{align*}
 & \!\!\!\!\!\!\!\!\!\!\!\!\!\!\!\!\!\!\!\!\!\!\!\!\!\frac{\E_{\nu,P}(\sum_{j}a_{j}f_{1,j}+\sum_{j}a_{j}f_{2,j},\sum_{j}a_{j}f_{1,j}+\sum_{j}a_{j}f_{2,j})}{\left\langle \sum_{j}a_{j}f_{1,j}+\sum_{j}a_{j}f_{2,j},\sum_{j}a_{j}f_{1,j}+\sum_{j}a_{j}f_{2,j}\right\rangle _{\nu,P}}\\
 & =\frac{\E_{\nu,P}(\sum_{j}a_{j}f_{1,j},\sum_{j}a_{j}f_{1,j})+\E_{\nu,P}(\sum_{j}a_{j}f_{2,j},\sum_{j}a_{j}f_{2,j})}{\left\langle \sum_{j}a_{j}f_{1,j},\sum_{j}a_{j}f_{1,j}\right\rangle _{\nu,P}}\\
 & \geq\frac{\E_{\nu,P}(\sum_{j}a_{j}f_{1,j},\sum_{j}a_{j}f_{1,j})}{\left\langle \sum_{j}a_{j}f_{1,j},\sum_{j}a_{j}f_{1,j}\right\rangle _{\nu,P}}.
\end{align*}
Note that $\spann(f_{1,1},\dots,f_{1,i})\subset\dom\left(\E_{\nu,P}\right)$
is also $i$-dimensional subspace in $L_{\nu,P}^{2}$ by our definition
of $\dom\left(\E_{\nu,P}\right)$. Hence, in any case the reverse
inequality follows.
\end{proof}
\begin{lem}
\label{lem general}Let $\theta,\vartheta$ Borel measures on $\left(0,1\right)$
with $\card\left(\supp\vartheta\right)=\card\left(\supp\theta\right)=\infty$.
Then for each $i\in\N$,
\[
\lambda_{\vartheta+\theta}^{i}\leq\lambda_{\theta}^{i},
\]
and for all $x\geq0$
\[
\max\left(N_{\vartheta}(x),N_{\theta}(x)\right)\leq N_{\vartheta+\theta}(x).
\]
\end{lem}

\begin{proof}
For all $f\in H_{0}^{1}(a,b)\setminus\left\{ 0\right\} $, we have
\[
\frac{\int\left(\nabla_{\Lambda}f\right){}^{2}\d\Lambda}{\int f^{2}\d(\theta+\vartheta)}\leq\frac{\int\left(\nabla_{\Lambda}f\right)^{2}\d\Lambda}{\int f^{2}\d\vartheta},\frac{\int\left(\nabla_{\Lambda}f\right)^{2}\d\Lambda}{\int f^{2}\d(\theta+\vartheta)}\leq\frac{\int\left(\nabla_{\Lambda}f\right)^{2}\d\Lambda}{\int f^{2}\d\theta}
\]
and the claim follows from Proposition \ref{prop:dom_vs_H01 Minmax}.
\end{proof}
\begin{rem}
We can reformulate the above lemma as follows. If we assume $\theta-\vartheta$
is again a positive measure, then, for each $i\in\N$, we have
\[
\lambda_{\theta}^{i}\leq\lambda_{\vartheta}^{i},
\]
and for all $x\geq0$,
\[
N_{\vartheta}(x)\leq N_{\theta}(x).
\]
\end{rem}

Now we are in the position also to state the proof of Proposition
\ref{Prop:bi_Lipschitz_trafo}.
\begin{proof}
[Proof of Proposition  \ref{Prop:bi_Lipschitz_trafo}] For $\E_{\nu}(\varphi,\psi)=\int_{[0,1]}\nabla_{\Lambda}\varphi\nabla_{\Lambda}\psi\:\mathrm{d}\Lambda$
using the bi-Lipschitz property, we have 
\begin{align*}
\frac{\E_{\nu\circ g^{-1},[a,b]}(\psi,\psi)}{\langle\psi,\psi\rangle_{\nu\circ g^{-1}}} & =\frac{\int_{[a,b]}\left(\nabla_{\Lambda}\psi\right)^{2}\d\Lambda}{\left\langle \psi\circ g,\psi\circ g\right\rangle _{\nu}}=\frac{\int_{[0,1]}\left(\left(\nabla_{\Lambda}\psi\right)\circ g\right)^{2}\left|\nabla_{\Lambda}g\right|\d\Lambda}{\left\langle \psi\circ g,\psi\circ g\right\rangle _{\nu}}\\
 & =\frac{\int_{[0,1]}\left(\nabla_{\Lambda}\left(\psi\circ g\right)\right)^{2}\left(\left|\nabla_{\Lambda}g\right|\right)^{-1}\d\Lambda}{\left\langle \psi\circ g,\psi\circ g\right\rangle _{\nu}}\asymp\frac{\E_{\nu,[0,1]}(\psi\circ g,\psi\circ g)}{\langle\psi\circ g,\psi\circ g\rangle_{\nu}}.
\end{align*}
Since $\Phi:H_{0}^{1}(a,b)\rightarrow H_{0}^{1}(0,1),$ $\Phi(f)\coloneqq f\circ g$
defines a linear bijection with $\nabla_{\Lambda}\Phi(f)=\left(\nabla_{\Lambda}f\right)\circ g\left|\nabla_{\Lambda}g\right|$
and inverse $\Phi^{-1}(f)=f\circ g^{-1}$. We see that for any linear
subspace $G\subset H_{0}^{1}(a,b)$ with $\dim(G)=i$ with respect
to $\E_{\nu,[0,1]}$ if, and only if
\[
\dim(\Phi(G))=i
\]
with respect to $\E_{\nu\circ g^{-1},[a,b]}$. The extended min-max
principle from Proposition \ref{prop:dom_vs_H01 Minmax} proves the
claim.
\end{proof}
\begin{prop}
\label{prop:SubandSupperAdditivity} For a $\nu$-partition $P\coloneqq\left((c_{k},c_{k+1}]\colon k=1,\ldots,n\right)$
and all $x\geq0$, we have
\[
\text{\ensuremath{\sum_{k=1}^{n}N_{\nu,[c_{k},c_{k+1}]}}}(x)\leq\text{\ensuremath{N_{\nu}}}(x)\leq\text{\ensuremath{\sum_{k=1}^{n}N_{\nu,[c_{k},c_{k+1}]}}}(x)+n-2.
\]
\end{prop}

\begin{proof}
For the $i$-th eigenvalue we have with Proposition \ref{prop:dom_vs_H01 Minmax}
\begin{align*}
\lambda_{i}\left(\E_{\nu,P}\right) & =\inf\left\{ \sup\left\{ \frac{\sum_{k}\E_{\nu,[c_{k},c_{k+1}]}\left(\psi_{k},\psi_{k}\right)}{\sum_{k}\langle\psi_{k},\psi_{k}\rangle_{\nu|_{(c_{k},c_{k+1})}}}:\left(\psi_{k}\right)\in G^{\star}\right\} \colon G<_{i}\dom\left(\E_{\nu,P}\right)\right\} \\
 & \geq\inf\left\{ \sup\left\{ \frac{\E_{\nu}(\psi,\psi)}{\langle\psi,\psi\rangle_{\nu}}:\psi\in G^{\star}\right\} \colon G<_{i}H_{0}^{1}(a,b)\right\} \\
 & =\inf\left\{ \sup\left\{ \frac{\E_{\nu}(\psi,\psi)}{\langle\psi,\psi\rangle_{\nu}}:\psi\in G^{\star}\right\} \colon G<_{i}\dom\left(\E_{\nu}\right)\right\} =\lambda_{i}\left(\E_{\nu}\right).
\end{align*}
This gives
\[
\text{\ensuremath{\sum_{k=1}^{n}N_{\nu,[c_{k},c_{k+1}]}}}(x)=N_{\nu,P}\left(x\right)\leq\text{\ensuremath{N_{\nu}}}(x).
\]
For the second inequality we make use of $H_{0,P}^{1}\coloneqq\left\{ f\in H_{0}^{1}(a,b)\colon f(c_{k})=0,k=1,\ldots,n\right\} \simeq H_{P}^{1}$
and $\dim\left(H_{0}^{1}(a,b)/H_{0,P}^{1}\right)=n-2$. We now estimate
as in the proof of \citep[Theorem 4.1.7.]{kigami_2001}:
\begin{align*}
\lambda_{i+n-2}\left(\E_{\nu}\right) & =\inf\left\{ \sup\left\{ \frac{\E_{\nu}(\psi,\psi)}{\langle\psi,\psi\rangle_{\nu}}:\psi\in G^{\star}\right\} \colon G<_{i+n-2}H_{0}^{1}(a,b)\right\} \\
 & \geq\inf\left\{ \sup\left\{ \frac{\E_{\nu}(\psi,\psi)}{\langle\psi,\psi\rangle_{\nu}}:\psi\in G^{\star}\right\} \colon G<_{i}H_{0,P}^{1}\right\} \\
 & =\inf\left\{ \sup\left\{ \frac{\sum\E_{\nu,[c_{k},c_{k+1}]}\left(\psi_{k},\psi_{k}\right)}{\sum\langle\psi_{k},\psi_{k}\rangle_{\nu|_{(c_{k},c_{k+1})}}}:\left(\psi_{k}\right)\in G^{\star}\right\} \colon G<_{i}H_{P}^{1}\right\} \\
 & =\inf\left\{ \sup\left\{ \frac{\sum\E_{\nu,[c_{k},c_{k+1}]}\left(\psi_{k},\psi_{k}\right)}{\sum\langle\psi_{k},\psi_{k}\rangle_{\nu|_{(c_{k},c_{k+1})}}}:\left(\psi_{k}\right)\in G^{\star}\right\} \colon G<_{i}\dom\left(\E_{\nu,P}\right)\right\} \\
 & =\lambda_{i}\left(\E_{\nu,P}\right).
\end{align*}
\end{proof}
\begin{cor}
\label{cor:enlarge_Interval_does_not_change_spectral_dimension}Let
$\nu$ be a Borel measure with $\supp\nu\subset(a,b)\subset[c,d]$.
Then 
\[
N_{\nu,[a,b]}(x)\leq N_{\nu,[c,d]}(x)\leq N_{\nu,[a,b]}(x)+2.
\]
\end{cor}

\subsection{Neumann boundary conditions}

Although we have shown that Dirichlet boundary conditions alone are
sufficient to determine spectral asymptotics, for completeness we
give the relationship to Neumann boundary conditions, known as Neumann-Dirichlet
bracketing. Let $I_{\nu,a}$ and $I_{\nu,b}$ denote possibly empty
components of $[a,b]\setminus\supp\nu$ containing $a$, respectively
$b$ and set
\[
C_{\nu}^{N}([a,b])\coloneqq\left\{ f\in C_{\nu}([a,b])\mid f\:\text{\text{is}\:const. on \:}I_{\nu,a}\:\text{and}\:I_{\nu,b}\right\} .
\]
Note that $\left(C_{\nu}^{N}([a,b]),\left\Vert \cdot\right\Vert _{\infty}\right)$
is a Banach space. We define the form
\[
\E_{\nu,\left[a,b\right]}^{N}(f,g)\coloneqq\int_{\left[a,b\right]}\nabla_{\Lambda}f\nabla_{\Lambda}g\d\Lambda
\]
with Neumann domain $\dom(\E_{\nu,\left[a,b\right]}^{N})\coloneqq H^{1}(a,b)\cap C_{\nu}^{N}([a,b])$.
Again $\E_{\nu}^{N}$ is well defined. The set $\dom\left(\E_{\nu,[a,b]}^{N}\right)$
is dense in $L_{\nu}^{2}$ and equipped it with the inner product
$(f,g)_{\E_{\nu,[a,b]}^{N}}\coloneqq\langle f,g\rangle_{\nu}+\E_{\nu,[a,b]}^{N}(f,g)$,
defines a Hilbert space.

Using \citep[Theorem 4.4.2]{davies_1995} or \citep[Theorem 2.1., p. 322]{MR0407617},
there exists a non-negative self-adjoint operator $\Delta_{\nu,[a,b]}^{N}$
which is characterized as follows; $f\in\dom\left(\E_{\nu,[a,b]}^{N}\right)$
lies in the domain $\dom\left(\Delta_{\nu,[a,b]}^{N}\right)$ of the
Laplacian $\Delta_{\nu,[a,b]}^{N}$ if and only if there exists $h\in L_{\nu}^{2}$
such that
\[
\E_{\nu,[a,b]}(g,f)=\langle g,h\rangle_{\nu},\;\;\mbox{ for all }\;\;g\in\dom(\E_{\nu,[a,b]}^{N})
\]
with $h\eqqcolon\Delta_{\E_{\nu,[a,b]}^{N}}f\eqqcolon\Delta_{\nu,[a,b]}^{N}f.$
Furthermore, let $N_{\nu}^{N}$ denote the eigenvalue counting function
of $\Delta_{\nu,[a,b]}^{N}.$

Also in our situation the min-max principle gives the usual Dirichlet-Neumann
bracketing.
\begin{prop}
\label{prop:D-NBracketing}Let $\nu$ be a Borel measure on $(a,b)$,
then, for all $x\geq0$, we have 
\[
\text{\ensuremath{N_{\nu,[a,b]}}}(x)\leq\text{\ensuremath{N_{\nu,[a,b]}^{N}}}(x)\leq\text{\ensuremath{N_{\nu,[a,b]}}}(x)+2.
\]
\end{prop}

\begin{proof}
For $g\in\dom(\E_{\nu})$ we define
\[
\widehat{g}\left(x\right)\coloneqq g\left(x\right)+\1_{\left[a,a'\right)}\left(\left(a'-x\right)g\left(a'\right)/\left(a'-a\right)\right)+\1_{\left(b',b\right]}\left(\left(x-b'\right)g\left(b'\right)/\left(b-b'\right)\right)\in\dom\left(\E_{\nu}^{N}\right).
\]
Since $\E_{\nu}(g,g)\geq\E_{\nu}(\widehat{g},\widehat{g})$ and $\langle g,g\rangle_{\nu}=\langle\widehat{g},\widehat{g}\rangle_{\nu}$
this gives
\begin{align*}
\lambda_{i}\left(\E_{\nu,[a,b]}\right) & =\inf\left\{ \sup\left\{ \frac{\E_{\nu,[a,b]}(g,g)}{\langle g,g\rangle_{\nu}}:g\in G^{\star}\right\} \colon G<_{i}\dom\left(\E_{\nu,[a,b]}\right)\right\} \\
 & \geq\inf\left\{ \sup\left\{ \frac{\E_{\nu,[a,b]}\left(\widehat{g},\widehat{g}\right)}{\langle\widehat{g},\widehat{g}\rangle_{\nu}}:g\in G^{\star}\right\} \colon G<_{i}\dom\left(\E_{\nu,[a,b]}\right)\right\} \\
 & \geq\inf\left\{ \sup\left\{ \frac{\E_{\nu,[a,b]}(g,g)}{\langle g,g\rangle_{\nu}}:g\in G^{\star}\right\} \colon G<_{i}\dom\left(\E_{\nu,[a,b]}^{N}\right)\right\} \\
 & =\lambda_{i}\left(\E_{\nu,[a,b]}^{N}\right).
\end{align*}
For the reversed inequality we note that the dimension of the quotient
vector space $\dom\left(\E_{\nu,[a,b]}^{N}\right)/\left(\dom\left(\E_{\nu,[a,b]}\right)\cap\dom\left(\E_{\nu,[a,b]}^{N}\right)\right)$
is equal to two. Then using again the proof of \citep[Theorem 4.1.7.]{kigami_2001},
we deduce
\begin{align*}
\lambda^{i+2}\left(\E_{\nu,[a,b]}^{N}\right) & =\inf\left\{ \sup\left\{ \frac{\E_{\nu,[a,b]}(g,g)}{\langle g,g\rangle_{\nu}}:g\in G^{\star}\right\} \colon G<_{i+2}\dom\left(\E_{\nu,[a,b]}^{N}\right)\right\} \\
 & \geq\inf\left\{ \sup\left\{ \frac{\E_{\nu,[a,b]}(g,g)}{\langle g,g\rangle_{\nu}}:g\in G^{\star}\right\} \colon G<_{i}\dom\left(\E_{\nu,[a,b]}\right)\cap\dom\left(\E_{\nu,[a,b]}^{N}\right)\right\} \\
 & \geq\inf\left\{ \sup\left\{ \frac{\E_{\nu,[a,b]}(g,g)}{\langle g,g\rangle_{\nu}}:g\in G^{\star}\right\} \colon G<_{i}\dom\left(\E_{\nu,[a,b]}\right)\right\} =\lambda^{i}\left(\E_{\nu,[a,b]}\right).
\end{align*}
\end{proof}

\section{Optimal partitions and partition entropy \label{sec:OptimalPartitions}}

Recall the definition of $\Pi$, $\Pi_{0}$, $\mathcal{N}^{R}\left(x\right)$,
$\mathcal{N}_{m}^{L}\left(x\right)$, $\overline{h}_{\nu}$, $\underline{h}_{\nu}$,
$\underline{h}_{\nu}^{m}$, $\overline{h}_{\nu}^{m}$ as provided
in the introduction.

\subsection{\label{subsec:First-observations}Basic observations}

In this section we\textbf{ }establish a connection of our approach
and the classical work of Birman and Solomjak \citep{MR0209733,MR0217487,MR0281860}.
This connection will come into play again in Example \ref{Example:fastDecay}.

Let $\Upsilon_{n}$ denote the set of all finite $\nu$-partitions
consisting of at most $n\in\N$ half open intervals of $(a,b]$ with
$\nu\left(\left\{ a,b\right\} \right)=0$ and let us define 
\begin{align*}
\gamma_{n} & \coloneqq\gamma_{n,\nu}\coloneqq\inf_{\Xi\in\Upsilon_{n}}\max_{I\in\Xi}\Lambda(I)\nu(I),
\end{align*}
and
\begin{equation}
\underline{\alpha}_{\nu}\coloneqq\liminf_{n\rightarrow\infty}\frac{\log\left(\gamma_{n}\right)}{-\log(n)},\:\overline{\alpha}_{\nu}\coloneqq\limsup_{n\rightarrow\infty}\frac{\log\left(\gamma_{n}\right)}{-\log(n)}.\label{eq:PartitionRate}
\end{equation}

\begin{rem}
If $\nu$ has an absolutely continuous part, then the limit in \eqref{eq:PartitionRate}
is equal to $2$ and in general, we have that $\underline{\alpha}_{\nu}\in$$\left[2,\infty\right]$.

Note also that in the definition of $\gamma_{n}$ for general $\nu$
(i.~e\@. also allowing for a pure point part) one cannot replace
the infimum with the minimum. Indeed, consider the Dirac measure $\delta_{1/2}$
on $[0,1].$ Then for every $1/2>\epsilon>0$, $\left\{ (0,1/2-\epsilon],(1/2-\epsilon,1/2+\epsilon],(1/2+\epsilon,1]\right\} $
is $\nu$ partition, hence we obtain for $n\geq3$ that $\gamma_{n}\leq2\varepsilon$,
$\gamma_{n}=0$ and the minimum does not exist.
\end{rem}

\begin{lem}
\label{lem:induction_betaN}For all $n\in\N$ we have $\gamma_{2^{n+1}}\leq\frac{1}{2}\gamma_{2^{n}}<\gamma_{2^{n}}.$
If $\nu$ is atomless, then $\gamma_{2^{n+1}}\leq\frac{1}{4}\gamma_{2^{n}}.$
\end{lem}

\begin{proof}
For $\Xi\in\Upsilon,\,\card\Xi\leq2^{n}$ we can divide each $I\in\Xi$
into two sub-intervals $I_{1},I_{2}$ of same length. The new resulting
partition denoted by $\Xi'$ satisfies $\card(\Xi')\leq2^{n+1}$ and
\[
\max_{I\in\Xi'}\nu(I)\Lambda(I)\leq\frac{1}{2}\max_{I\in\Xi}\nu(I)\Lambda(I)<\max_{I\in\Xi}\nu(I)\Lambda(I).
\]
This implies the first claim of the Lemma. For the second statement
fix $I\in\Xi$ and let $F$ denote the (continuous) distribution function
of the normalized restriction of $\nu$ to $I$. Then choose $t\in\left(1/4,1\right)$
such that $tF\left(t\right)=1/4$. An elementary calculation then
shows that for such $t$ we have $\left(1-t\right)\left(1-F\left(t\right)\right)=5/4-t-1/\left(4t\right)\leq1/4$.
Now we subdivide $I$ in a left interval $I_{1}$ of length $t\Lambda\left(I\right)$
and the remaining right interval $I_{2}$. The resulting partition
now fulfills the inequality as above with $1/2$ replaced by $1/4$.
\end{proof}
\begin{rem}
If the measure $\nu$ is atomless, we deduce from Lemma \ref{lem:induction_betaN}
by induction $\gamma_{2^{n}}\leq4^{-n+1}\gamma_{1}=4^{-n+1}\nu([a,b])(b-a)$
and as a simple consequence we find $\underline{\alpha}_{\nu}\geq2$.
This will give rise to the famous upper bound $1/2$ on the spectral
dimension as obtained in \citep{MR0209733,MR0217487} also for arbitrary
measure $\nu$.
\end{rem}

\begin{lem}
\label{lem: SolomyakEquivalence}We have $\overline{h}_{\nu}=1/\underline{\alpha}_{\nu}$
and $\underline{h}_{\nu}=1/\overline{\alpha}_{\nu}$.
\end{lem}

\begin{proof}
The proof follows along the same lines as the proof of the elementary
Lemma \citep[Lemma 2.2]{MR2083820}. First note that for $0<\varepsilon<1$
we have $\mathcal{N}^{R}\left(1/\epsilon\right)=\inf\left\{ n\in\N\mid\gamma_{n}<\epsilon\right\} $.
By Lemma \ref{lem:induction_betaN} we have that $\left(\gamma_{2^{n}}\right)_{n}$
is a strictly decreasing null sequences or eventually constant zero.
The latter case is immediate. For the first case the strict monotonicity
gives as in \citep[Lemma 2.2]{MR2083820} for $B\left(\epsilon\right)\coloneqq\inf\left\{ n\in\N\mid\gamma_{2^{n}}<\epsilon\right\} $
\[
\limsup_{n\to\infty}\frac{n\log2}{-\log\gamma_{2^{n}}}=\limsup_{\epsilon\searrow0}\frac{B\left(\epsilon\right)\log2}{-\log\epsilon}.
\]
The left hand side gives $1/\underline{\alpha}_{\nu}$ by squeezing
$m\to\infty$ on subsequences $2^{n-1}<m\leq2^{n}$ and the right
hand side $\overline{h}_{\nu}$ by noting that $2^{B\left(\epsilon\right)-1}\leq\mathcal{N}^{R}\left(1/\epsilon\right)\leq2^{B\left(\epsilon\right)}$.
The same result holds for $\liminf$.
\end{proof}

\subsection{Some global bounds}
\begin{lem}[Poincaré inequality]
\label{lem:Poincare_inequality}Let $\nu$ be a Borel measure on
$(a,b)$. Then for the first eigenvalue of $\E_{\text{\ensuremath{\nu}},[a,b]}$,
we have
\[
\lambda_{\nu}^{1}\geq\frac{1}{\nu([a,b])\Lambda([a,b])}.
\]
\end{lem}

\begin{proof}
Here, we follow \citep[Lemma 4.6.]{MR4048458}. Let $f$ be the eigenfunction
with eigenvalue $\lambda_{\nu}^{1}$ with $\left\Vert f\right\Vert _{\nu}=1.$
Then by the Cauchy-Schwarz inequality, for all $x\in[a,b]$, we have
\[
f(x)^{2}=\left|f(x)-f(a)\right|^{2}=\left|\int_{[a,x]}\nabla_{\Lambda}f\d\Lambda\right|^{2}\leq\Lambda\left(\left[a,b\right]\right)\left\Vert \nabla_{\Lambda}f\right\Vert _{\Lambda}^{2}
\]
Now, integrating with respect to $\nu$ gives
\[
1\leq\nu\left(\left[a,b\right]\right)\Lambda\left(\left[a,b\right]\right)\left\Vert \nabla_{\Lambda}f\right\Vert _{\Lambda}^{2}=\nu\left(\left[a,b\right]\right)\Lambda\left(\left[a,b\right]\right)\lambda_{\nu}^{1}.
\]
\end{proof}
In the following we will use the approximation theory of Birman and
Solomjak developed in \citep[Theorem 2.1,Theorem 2.2]{MR0217487},
see also \citep{MR0209733,MR0301568}. Let $\nu$ be a Borel measure
on $[a,b]$. For every $n\in\N$ there exists a partition of half
open intervals of $(a,b]$ denoted by $\Xi$ such that $\card\left(\Xi\right)\leq n$
and
\[
\max_{I\in\Xi}\Lambda(I)\nu(I)\leq C_{1}n^{-2}\nu([a,b])(b-a),
\]
where $C_{1}>0$ is independent of $\nu.$ Assuming now $\nu$ to
be atomless the constant is given by $C_{1}=1.$ Putting $n\coloneqq\left\lceil \sqrt{x(b-a)\nu([a,b])}\right\rceil $
and note that by Lemma \ref{lem:Poincare_inequality} for every $C\in\Xi$
with $\nu(C)>0$, we have
\[
\lambda_{\nu,C}^{1}\geq\frac{1}{\nu(C)\Lambda(C)}\geq\frac{1}{\max_{\Gamma\in\Xi}\Lambda(\Gamma)\nu(\Gamma)}\geq\frac{n^{2}}{\nu([a,b])(b-a)},
\]
hence applying Proposition \ref{prop:SubandSupperAdditivity} delivers
\[
N_{\nu}(x)\leq\sum_{C\in\varXi\:\nu(C)>0}N_{\nu,\Lambda,C}(x)+n\leq n\leq1+\sqrt{\nu([a,b])(b-a)x},
\]
where we used that $\lambda_{\nu,C}^{1}\geq n^{2}\left(C_{1}\nu([a,b])(b-a)\right)^{-1}>x$
by the choice of $n$.

The estimate can be improved in the case that $\nu$ is singular with
respect to $\Lambda$. To prove this we need the following refinement
which is developed in \citep{Borzov1971}:

\emph{There is a constant $C_{1}>\text{0}$ such that for every measure
$\nu$ singular with respect to $\Lambda$ and every $\varepsilon>0$
there exists $n_{\nu,\varepsilon}\in\N$ such that for every $n>n_{\nu,\varepsilon}$
there is a partition of half open intervals of $(a,b]$ denoted by
$\Xi$ such that $\card\left(\Xi\right)\leq n$ and 
\[
\max_{I\in\Xi}\Lambda(I)\nu(I)\leq\varepsilon C_{1}\frac{\nu([a,b])(b-a)}{\left(1-n_{\nu,\varepsilon}/n\right)^{2}n^{2}}.
\]
} Hence, for $\nu$ atomless and singular with respect to $\Lambda$,
and for every $\varepsilon>0$ and all $x>0$ large enough, we have
\[
N_{\nu}(x)\leq\sqrt{\varepsilon x}.
\]
In particular, $\lim_{x\rightarrow\infty}x^{-1/2}N_{\nu}(x)=0.$
\begin{lem}
\label{lem:UpperBoundForDirichlet} For the subinterval $C=[c,d]$
of $[a,b]$ assume $\nu\left(\left\{ c,d\right\} \right)=0$ and fix
$\alpha_{1},\alpha_{2}\in C$ such that $c<\alpha_{1}<\alpha_{2}<d$
and $\nu([\alpha_{1},\alpha_{2}])>0$. Then 
\[
\lambda_{\nu,C}^{1}\leq\frac{\left(\frac{\alpha_{2}-\alpha_{1}}{\alpha_{1}-c}+\frac{\alpha_{2}-\alpha_{1}}{d-\alpha_{2}}\right)}{\nu\left([\alpha_{1},\alpha_{2}]\right)\Lambda\left([\alpha_{1},\alpha_{2}]\right)}.
\]
In particular, if $\nu\left(\left\langle C\right\rangle _{m}\right)>0$
for some $m>1$, then
\[
\lambda_{\nu,C}^{1}\leq\frac{4/(m-1)}{\nu\left(\left\langle C\right\rangle _{m}\right)\Lambda\left(\left\langle C\right\rangle _{m}\right)}.
\]
\end{lem}

\begin{proof}
Define 
\[
f(x)\coloneqq\frac{x-c}{\alpha_{1}-c}\1_{\left[c,\alpha_{1}\right)}+\1_{\left[\alpha_{1},\alpha_{2}\right]}+\frac{d-x}{d-\alpha_{2}}\1_{\left(\alpha_{2},d\right]},
\]
then $f\in H_{0}^{1}(c,d)$ and the min-max principle gives
\[
\lambda_{\nu,C}^{1}\leq\frac{\int\left(f'\right)^{2}\d\Lambda}{\int f^{2}\d\nu}=\frac{\left(\frac{1}{\alpha_{1}-c}+\frac{1}{d-\alpha_{2}}\right)}{\int f^{2}\d\nu}\leq\frac{\left(\alpha_{2}-\alpha_{1}\right)\left(\frac{1}{\alpha_{1}-c}+\frac{1}{d-\alpha_{2}}\right)}{\nu\left([\alpha_{1},\alpha_{2}]\right)\Lambda\left([\alpha_{1},\alpha_{2}]\right)}.
\]
For $\alpha_{1}=\frac{1}{2}\left(c+d-\frac{d-c}{m}\right)$, $\alpha_{2}=\frac{1}{2}\left(c+d+\frac{d-c}{m}\right)$
and provided $\nu\left(\left\langle C\right\rangle _{m}\right)>0$,
one obtains $\lambda_{\nu,\Lambda,C}^{1}\leq4/(m-1)\left(\nu\left(\left\langle C\right\rangle _{m}\right)\Lambda\left(\left\langle C\right\rangle _{m}\right)\right)^{-1}$.
\end{proof}
\begin{prop}
\label{prop:estimateOfSpectralDim_N^L_m_N^R} Let $\nu$ be a Borel
measure on $(a,b)$. Then for each $m>1$ and all $x>0$, we have
\[
\mathcal{N}_{m}^{L}\left(x\right)\leq N_{\nu}(x)\leq\mathcal{N}^{R}\left(x\right).
\]
In particular,
\[
\underline{h}_{\nu}^{m}\leq\underline{s}_{\nu}\leq\overline{s}_{\nu}\leq\overline{h}_{\nu}\;\text{and \, }\underline{s}_{\nu}\leq\underline{h}_{\nu}\;\text{and \, \ensuremath{\overline{h}_{\nu}^{m}\leq\overline{s}_{\nu}}}
\]
and if
\[
\sup_{m>1}\underline{h}_{\nu}^{m}=\overline{h}_{\nu},
\]
then the spectral dimension exists and equals $\overline{h}_{\nu}$.
\end{prop}

\begin{proof}
This follows immediately from the sub- and superadditivity of the
eigenvalue counting functions (Proposition \ref{prop:SubandSupperAdditivity})
in tandem with Lemma \ref{lem:Poincare_inequality} and Lemma \ref{lem:UpperBoundForDirichlet}.
\end{proof}

\section{$L^{q}$-spectrum and coarse multifractal formalism\label{sec:-spectrum-and-coarse}}

In this section we establish a connection between the spectral dimension,
$L^{q}$-spectrum and the coarse multifractal formalism.

\subsection{The $L^{q}$-spectrum}

In this section we collect some important facts about the $L^{q}$-spectrum
providing proofs only for those which are not straight forward or
have not been covered in the literature.

We let $\nu$ be an arbitrary finite Borel measure with bounded support.\textbf{
}For $k\in\mathbb{Z}$ and $n\in\N$ we let $A_{k}^{n}\coloneqq\left(k2^{-n},\left(k+1\right)2^{-n}\right]$
and set $\mathcal{D}_{\nu,n}\coloneqq\left\{ A_{k}^{n}:k\in\mathbb{Z},\nu\left(A_{k}^{n}\right)>0\right\} $,
$\mathcal{D}_{\nu}\coloneqq\bigcup_{n}\mathcal{D}_{\nu,n}$. For $q\in\R$,
the $L^{q}$\emph{-spectrum} of $\nu$ is defined by
\[
\beta_{\nu}\left(q\right)\coloneqq\limsup_{n\to\infty}\beta_{n}(q)\;\text{with \, \,}\quad\beta_{n}^{\nu}\left(q\right)\coloneqq\beta_{n}\left(q\right)\coloneqq\frac{1}{n\log2}\log\sum_{C\in\mathcal{D}_{\nu,n}}\nu\left(C\right)^{q}.
\]
Let us begin with some basic facts.
\begin{fact}
\label{Fact IndependenceOFChoiceCordinates_and_factor} The function
$\beta_{\nu}$ will not alter when we take $d$-adic intervals instead
of dyadic ones \textbf{(}see e.~g\@. \citep[Proposition 2 and Remarks, p. 466]{MR1312056}
or \citep[Proposition 1.6.]{Riedi_diss}\textbf{ }and note that the
definition in \citep[Proposition 1.6.]{Riedi_diss}\textbf{ }coincides
with our definition for $q\geq0$\textbf{). }More precisely, for fixed
$\delta>\text{0}$, set
\[
G_{\nu,\delta}\coloneqq G_{\delta}\coloneqq\left\{ \left(l\delta,(l+1)\delta\right]\mid l\in\N,\,\nu\left(\left(l\delta,(l+1)\delta\right]\right)>0\right\} 
\]
and let $(\delta_{n})$ be an \emph{admissible} sequence, i.~e\@.
$\delta_{n}\in(0,1)^{\N}$, $\delta_{n}\rightarrow0$ and there exists
a constant $C>0$ such that $C\delta_{n}\leq\delta_{n+1}\leq\delta_{n}$,
\textup{for all} $n\in\N$. Then for $q\geq0$, 
\[
\limsup_{\delta\downarrow0}\frac{1}{-\log(\delta)}\log\sum_{C\in G_{\delta}}\nu\left(C\right)^{q}=\limsup_{m\rightarrow\infty}\frac{1}{-\log(\delta_{m})}\log\sum_{C\in G_{\delta_{m}}}\nu\left(C\right)^{q}.
\]
In particular, for $\delta_{m}\coloneqq2^{-m}$, the above expression
equals $\beta_{\nu}(q).$
\end{fact}

\begin{fact}
\label{Fact: ObdaProbMeasure}For $c>0$ we have $\beta_{c\nu}=\beta_{\nu}$.
Therefore, we can assume without loss of generality that $\nu$ is
a probability measure.
\end{fact}

\begin{fact}
The function $\beta_{\nu}$ is ---~as a pointwise limit superior
of convex function~--- again convex and we have (see also \citep[Corollary ]{MR1312056})
\[
\beta_{\nu}\left(0\right)=\overline{\dim}_{M}\left(\nu\right)\eqqcolon\delta_{\nu}^{*}\quad\text{and }\quad\beta_{\nu}\left(1\right)=0.
\]
Hence, the left, respectively right, derivative in $1$ exists and
its negative value will be denoted by $\overline{\delta}_{\nu}\coloneqq\lim_{q\nearrow1}-\beta_{\nu}\left(q\right)/\left(1-q\right)$
and $\underline{\delta}_{\nu}\coloneqq\lim_{q\searrow1}-\beta_{\nu}\left(q\right)/\left(1-q\right)$.
For later use we also define the right derivative of $\beta$ in $0$
denoted by $\varrho_{\nu}\coloneqq\lim_{t\searrow0}\left(\beta_{\nu}\left(t\right)-\delta_{\nu}^{*}\right)/t.$
Note that in many situations we have $\overline{\delta}_{\nu}=\dim_{H}\left(\nu\right)$.
The function $\beta_{\nu}$ is non-increasing and non-negative on
$\R_{<1}$ and $\liminf_{n}\beta_{n}(q)\geq-q$ for all $q\geq0.$
\end{fact}

\begin{fact}
Since $\beta_{\nu}$ is a proper convex function, its \emph{Legendre
transform} exists, given by 
\[
\widehat{\beta}_{\nu}\left(\alpha\right)\coloneqq\inf_{q\in\R}\beta_{\nu}\left(q\right)+\alpha q.
\]
\end{fact}

\begin{fact}
\label{fact:MGF_for_sum} For any two finite Borel measures $\mu_{1}$,
$\mu_{2}$ with bounded support and for $q\geq0$, we have
\[
\beta_{\mu_{1}+\mu_{2}}(q)=\max\left(\beta_{\mu_{1}},\beta_{\mu_{2}}\right)(q).
\]
If $\beta_{\mu_{1}}\left(q\right)\geq\beta_{\mu_{2}}\left(q\right)$
and $\beta_{\mu_{1}}\left(q\right)$ exists as a limit, then so does
$\beta_{\mu_{1}+\mu_{2}}\left(q\right)=\beta_{\mu_{1}}\left(q\right)$.
\end{fact}

\begin{proof}
We readily check, for $q\geq0$,
\[
\sum_{C\in\mathcal{D}_{\mu_{i},n}}\left(\mu_{i}\left(C\right)\right)^{q}\leq\sum_{C\in\mathcal{D}_{\mu_{1}+\mu_{2},n}}\left(\mu_{1}\left(C\right)+\mu_{2}\left(C\right)\right)^{q}\leq2^{q}\left(\sum_{C\in\mathcal{D}_{\mu_{1},n}}\mu_{1}\left(C\right)^{q}+\sum_{C\in\mathcal{D}_{\mu_{2},n}}\mu_{2}\left(C\right)^{q}\right)
\]
implying our claim.
\end{proof}
\begin{fact}
\label{fact: atomicPart} If $\nu$ has an atomic part, then $\beta_{\nu}(q)=0$,
for all $q\geq1$.
\end{fact}

\begin{proof}
Assume that $\nu$ has an atom in $x_{0}\in(0,1)$ and let $q>1$.
Then for every $n\in\N$, we have that $0<\nu\left(\left\{ x_{0}\right\} \right)^{q}\leq\sum_{C\in\mathcal{D}_{\nu,n}}\nu(C)^{q}$
and hence $0\leq\beta_{\nu}(q)\leq\beta_{\nu}(1)=0$.
\end{proof}
\begin{fact}
\label{fact:SmallSupport}If $\delta_{\nu}^{*}=0$, then $\beta_{\nu}\left(q\right)=0,$
for all $q\geq0$ and $\dim_{H}\left(\nu\right)=\dim_{P}\left(\nu\right)=\underline{\delta}_{\nu}=\overline{\delta}_{\nu}=0$.
\end{fact}

\begin{proof}
Note that $\beta$ is a proper convex non-increasing function with
$\beta_{\nu}\left(0\right)=\beta_{\nu}\left(1\right)=0$. This implies
the first claim as well as $\underline{\delta}_{\nu}=\overline{\delta}_{\nu}=0$.
Further, recall that
\[
\dim_{H}\left(\nu\right)=\inf\left\{ \dim_{H}\left(A\right):\nu\left(A^{\complement}\right)=0\right\} \;\text{and }\dim_{P}\left(\nu\right)=\inf\left\{ \dim_{P}\left(A\right):\nu\left(A^{\complement}\right)=0\right\} .
\]
Since $\nu\left(\left(\supp\nu\right)^{\complement}\right)=0$ and
$\dim_{H}\left(\supp\nu\right)\leq\dim_{P}\left(\supp\nu\right)\leq\overline{\dim}_{M}\left(\supp\nu\right)=0$
the second claim follows.
\end{proof}
Recall from the introduction the definition of $\overline{q}_{\nu}\coloneqq\limsup_{n\to\infty}q_{n}^{\nu},$
where $q_{n}^{\nu}$ denotes the unique fixed point of $\beta_{n}^{\nu}$.
\begin{fact}
\label{fact:CriticalExponentBETA} We have 
\begin{align*}
\overline{q}_{\nu} & =\inf\left\{ q>0:\beta_{\nu}\left(q\right)-q\leq0\right\} \\
 & =\inf\left\{ q>0\colon\limsup_{n\to\infty}\frac{1}{n}\log\sum_{C\in\mathcal{D}_{\nu,n}}\left(\nu\left(C\right)\Lambda\left(C\right)\right)^{q}\leq0\right\} \\
 & =\inf\left\{ q>0:\sum_{C\in\mathcal{\mathcal{D}_{\nu}}}\left(\nu\left(C\right)\Lambda\left(C\right)\right)^{q}<\infty\right\} \\
 & =\sup_{\alpha\geq0}\frac{\widehat{\beta}_{\nu}\left(\alpha\right)}{1+\alpha}.
\end{align*}
If $\overline{q}_{\nu}>0$, then $\overline{q}_{\nu}$ is the unique
fixed point of $\beta_{\nu}$ in $\left(0,1\right)$, i.~e\@. $\beta_{\nu}\left(\overline{q}_{\nu}\right)=\overline{q}_{\nu}\in\left(0,1\right)$.
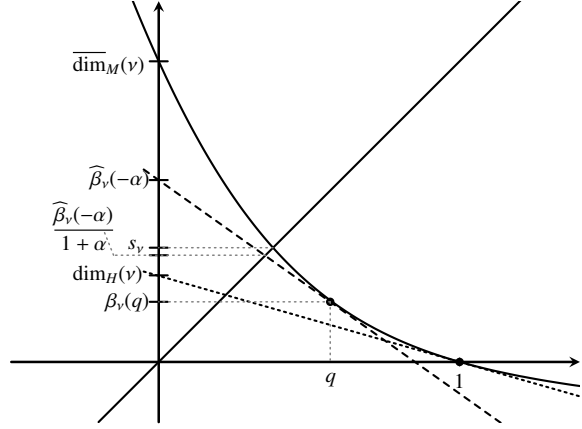
\begin{figure}
\center{\begin{tikzpicture}[scale=0.8, every node/.style={transform shape},line cap=round,line join=round,>=triangle 45,x=1cm,y=1cm] \begin{axis}[ x=5cm,y=5cm, axis lines=middle, axis line style={very thick},ymajorgrids=false, xmajorgrids=false, grid style={thick,densely dotted,black!20}, xmin=-0.49 , xmax=1.4 , ymin=-0.2, ymax=1.2, xtick={0,1}, ytick={0},] \clip(-0.49,-0.2) rectangle (1.4,1.4); 
\draw[line width=1pt,smooth,samples=180,domain=-0.78:1.4] plot(\x,{log10(0.05^((\x))+0.95^((\x)))/log10(2)}); 
\draw[line width=1pt, smooth,samples=100,domain=-0.2:1.4] plot(\x,{(\x)}); 
\draw [line width=1pt,dashed, domain=-0.05 :1.4] plot(\x,{0.71*(0.57-\x)+0.2}); 
\draw [line width=01pt,dotted, domain=-0.05 :1.4] plot(\x,{-((0.05*log10(0.05)+0.95*log10(0.95))/(log10(2))*(1-\x))});
 
\fill (0.57 ,0.2) circle[radius=2pt];
\fill (1 ,0 ) circle[radius=2pt];

\draw [line width=1pt ] (0.03 ,0.2)--(-0.03,0.2);
\draw [line width=1pt ] (0.03 ,0.355)--(-0.03,0.355);
\draw [line width=1pt ] (0.03 ,1)--(-0.03,1);
\draw [line width=1pt ] (0.03 ,.38)--(-0.03,0.38);
\draw [line width=1pt ] (0.03 ,.288)--(-0.03,.288);
\draw [line width=1pt ] (0.03 ,0.605)--(-0.03,0.605 );

\draw [line width=.7pt,dotted, gray] (0.57 ,0.2)--(0.57,0); 
\draw [line width=.7pt,dotted, gray] (0.570 ,0.20 )-- (0,0.2);
\draw [line width=.7pt,dotted, gray] (0.38  ,0.38 )--(0,0.38 );
\draw [line width=.7pt,dotted, gray] (0.3550 ,0.3550)--(-0.15,0.355 );
\draw [line width=.7pt,dotted, gray] (-0.15,0.355 )--(-0.18,0.43 );
\draw (0.53 ,-0.01) node[anchor=north west] {$\Large q$}; 
\draw (-0.2,0.255) node[anchor=north west] {$\Large \beta_\nu(q)$}; 
\draw (-0.25 ,0.68 ) node[anchor=north west] {$\Large \widehat\beta_\nu(-\alpha)$}; 
\draw (-0.13 ,0.43 ) node[anchor=north west] {$\Large \displaystyle{s_{\nu}}$};
\draw (-0.38 ,0.56 ) node[anchor=north west] {$\Large \displaystyle{\frac{\widehat\beta_\nu(-\alpha)}{1+\alpha}}$};
\draw (-0.31 ,0.345 ) node[anchor=north west] {$\Large \displaystyle{\dim_H(\nu) }$};
\draw (-0.31 ,1.06 ) node[anchor=north west] {$\Large \displaystyle{\overline{\dim}_M(\nu)}$}; 
\end{axis} 
\end{tikzpicture}}

\caption{\label{fig:Moment-generating-function}Moment generating function
$\beta_{\nu}$ for $\nu$ with (dashed) tangent to $\beta_{\nu}$
of slope $-\alpha$ in $\left(q,\beta_{\nu}\left(q\right)\right)$
which intersects the $y$-axis in its Legendre transform $\widehat{\beta}_{\nu}\left(\alpha\right)$
and the bisector of the first quadrant in $\widehat{\beta}_{\nu}\left(\alpha\right)/\left(1+\alpha\right)$.
Here $\nu$ is chosen to be the $\left(0.05,0.95\right)$-Salem measure
with full support $\supp\nu=[0,1]$. The intersection of $\beta_{\nu}$
with the $y$-axis gives the Minkowski dimension of $\supp\nu$, namely
$1$, and the intersection with the (dotted) tangent to $\beta_{\nu}$
in $\left(0,1\right)$ gives the Hausdorff dimension $\dim_{H}\left(\nu\right)$
of the measure which equals $\left(0.05\log\left(0.05\right)+0.95\log\left(0.95\right)\right)/\log\left(2\right)$.
The spectral dimension $s_{\nu}$ is equal to the maximum over all
$\widehat{\beta}_{\nu}\left(\alpha\right)/\left(1+\alpha\right)$
given by $\overline{q}_{\nu}$.}
\end{figure}
\end{fact}

\begin{proof}
In view of Fact \ref{Fact: ObdaProbMeasure} we assume without loss
of generality that $\nu$ is a probability measure. Let us begin with
the first equality assuming $\inf\left\{ q>0:\beta_{\nu}\left(q\right)-q\leq0\right\} >0$.
Recall that $\beta_{\nu}$ defines a convex and real valued function
on $(0,1]$ and hence \textbf{$\beta_{\nu}$} is also continuous on
$(0,\infty)$ (see e.~g\@. \citep[Lemma 5 ]{MR1312056}). Since
$\beta_{\nu}$ is also decreasing with $\lim_{t\searrow0}\beta_{\nu}\left(t\right)>0$
and $\beta_{\nu}\left(1\right)=0$ there exists a unique fixed point
of $\beta_{\nu}$ which coincides with $\inf\left\{ q>0:\beta_{\nu}\left(q\right)-q\leq0\right\} $.
This fixed point also coincides with $\overline{q}_{\nu}$, which
can be checked easily by contradiction. This proves also the last
claim. If, on the other hand, $\inf\left\{ q>0:\beta_{\nu}\left(q\right)-q\leq0\right\} =0$,
then $\beta_{\nu}\left(q\right)=0$ for all $q>0$ and again by contradiction
$\overline{q}_{\nu}=0$. This proves the first equality.

The fact that $\limsup_{n\to\infty}\left(n\log(2)\right)^{-1}\log\sum_{C\in\mathcal{D}_{\nu,n}}\left(\nu\left(C\right)\Lambda\left(C\right)\right)^{q}=\beta_{\nu}\left(q\right)-q$
for all $q\in\R$ gives the second equality.

The third equality follows similar to \citep{MR3190211}. Let us assume
$\overline{q}_{\nu}>0$ and fix $0<q<\overline{q}_{\nu}$. Then by
the second characterization, we have $\limsup_{n\to\infty}n^{-1}\log\sum_{C\in\mathcal{D}_{\nu,n}}\left(\nu\left(C\right)\Lambda\left(C\right)\right)^{q}=\epsilon>0$
and in particular, for infinitely many $n$, for $\left(n_{j}\right)$
say, we have $\sum_{C\in\mathcal{D}_{\nu,n_{j}}}\left(\nu\left(C\right)\Lambda\left(C\right)\right)^{q}>\e^{n_{j}\epsilon/2}$.
Consequently, 
\begin{align*}
\sum_{C\in\mathcal{D}_{\nu}}\left(\nu\left(C\right)\Lambda\left(C\right)\right)^{q} & \geq\sum_{j=0}^{\infty}\sum_{C\in\mathcal{D}_{\nu,n_{j}}}\left(\nu\left(C\right)\Lambda\left(C\right)\right)^{q}\geq\sum_{j=0}^{\infty}\e^{n_{j}\epsilon/2}=\infty.
\end{align*}
Hence, $\overline{q}_{\nu}\leq\inf\left\{ q>0:\sum_{C\in\mathcal{\mathcal{D}_{\nu}}}\left(\nu\left(C\right)\Lambda\left(C\right)\right)^{q}<\infty\right\} $,
which is trivially fulfilled if $\overline{q}_{\nu}=0$. For the reverse
inequality and $\overline{q}_{\nu}\geq0$ arbitrary, fix $\epsilon>0$
and $q>\overline{q}_{\nu}$ such that for all $n>N$, we have $\sum_{C\in\mathcal{D}_{\nu,n}}\left(\nu\left(C\right)\Lambda\left(C\right)\right)^{q}<2^{n\epsilon/2}$.
Then 
\begin{align*}
\sum_{C\in\mathcal{\mathcal{D}_{\nu}}}\left(\nu\left(C\right)\Lambda\left(C\right)\right)^{q+\epsilon} & \leq\sum_{C\in\mathcal{D}}\left(\nu\left(C\right)\Lambda\left(C\right)\right)^{q}\Lambda\left(C\right)^{\epsilon}=\sum_{n=0}^{\infty}\sum_{C\in\mathcal{D}_{\nu,n}}\left(\nu\left(C\right)\Lambda\left(C\right)\right)^{q}2^{-n\epsilon}\\
 & \leq\left(\sum_{n=0}^{N}\sum_{C\in\mathcal{D}_{\nu,n}}\left(\nu\left(C\right)\Lambda\left(C\right)\right)^{q}2^{-n\epsilon}+\sum_{n=N}^{\infty}\sum_{C\in\mathcal{D}_{\nu,n}}\left(\nu\left(C\right)\Lambda\left(C\right)\right)^{q}2^{-n\epsilon}\right)\\
 & \leq\left(\sum_{n=0}^{N}\sum_{C\in\mathcal{D}_{\nu,n}}\left(\nu\left(C\right)\Lambda\left(C\right)\right)^{q}2^{-n\epsilon}+\sum_{n=N}^{\infty}2^{n\epsilon/2}2^{-n\epsilon}\right)<\infty.
\end{align*}
Hence $\inf\left\{ q>0:\sum_{C\in\mathcal{\mathcal{D}_{\nu}}}\left(\nu\left(C\right)\Lambda\left(C\right)\right)^{q}<\infty\right\} \leq\overline{q}_{\nu}$.

Now we prove the last equality. Note that, for every $q'>\overline{q}_{\nu}$,
we have $\beta_{\nu}(q')\leq q'$ and therefore,
\[
\sup_{\alpha\geq0}\frac{\widehat{\beta}_{\nu}\left(\alpha\right)}{1+\alpha}\leq\sup_{\alpha\geq0}\frac{q'+q'\alpha}{1+\alpha}=q'.
\]
This shows $\sup_{\alpha\geq0}\widehat{\beta}_{\nu}\left(\alpha\right)/\left(1+\alpha\right)\leq\overline{q}_{\nu}$
and equality holds if $\overline{q}_{\nu}=0$. Assuming $\overline{q}_{\nu}>0$,
we already know that $\overline{q}_{\nu}$ is the unique fixed point
of $\beta_{\nu}.$ If $-a$ is in the subdifferential $\partial\beta_{\nu}\left(\overline{q}_{\nu}\right)$
of $\beta$ in $\overline{q}_{\nu}=\beta_{\nu}\left(\overline{q}_{\nu}\right)$
with $a\geq0$, then (see Fig. \ref{fig:Moment-generating-function})
\begin{align*}
\sup_{\alpha\geq0}\frac{\widehat{\beta}_{\nu}\left(\alpha\right)}{1+\alpha} & \geq\frac{\inf_{q\in\R}\beta_{\nu}(q)+aq}{1+a}\geq\frac{\beta_{\nu}(\overline{q}_{\nu})-a\left(q-\overline{q}_{\nu}\right)+aq}{1+a}=\overline{q}_{\nu}.
\end{align*}
\end{proof}
\begin{prop}[{\citep{Riedi_diss},\citep[Prop.  2]{MR1312056}}]
\label{Prop:bi_Lipschitz_Lq-Spectrum}Let $\nu$ be a Borel measure
on $[0,1]$ and $g:\left[0,1\right]\to\left[a,b\right]$ be a bi-Lipschitz
mapping. Then 
\[
\beta_{\nu}=\beta_{\nu\circ g^{-1}}.
\]
\end{prop}

\begin{proof}
For completeness we give a short proof. First, by our assumption we
have that there exist $c_{1},c_{2}>0$ such that for all $x,y\in[0,1]$
\[
c_{1}\left|x-y\right|\leq\left|g(x)-g(y)\right|\leq c_{2}\left|x-y\right|.
\]
For $C\in G_{\nu\circ g^{-1},2^{-n}}$, 
\[
\frac{1}{c_{2}}\text{diam}\left(C\right)\leq\text{diam}\left(g^{-1}\left(C\right)\right)\leq\frac{1}{c_{1}}\text{diam}\left(C\right)=\frac{1}{c_{1}2^{n}}.
\]
This shows that there are at most two elements $C_{1},C_{2}$ of $G_{\nu,\left(2^{n}c_{1}\right)^{-1}}$
intersecting $g^{-1}(C)$. Furthermore, each element of $G_{\nu,\left(2^{n}c_{1}\right)^{-1}}$
that intersects $g^{-1}(C)$ can intersect at most $\left\lceil c_{2}/c_{1}\right\rceil $
other element $g^{-1}(C')$ with $C'\in G_{\nu\circ g^{-1},2^{-n}}$.
This implies
\[
\sum_{C\in G_{\nu\circ g^{-1},2^{-n}}}\nu\left(g^{-1}\left(C\right)\right)^{q}\leq\left\lceil \frac{c_{2}}{c_{1}}\right\rceil 2^{q}\sum_{C\in G_{\nu,\left(2^{n}c_{1}\right)^{-1}}}\nu\left(C\right)^{q}.
\]
By Fact \ref{Fact IndependenceOFChoiceCordinates_and_factor}, we
therefore have
\[
\beta_{\nu\circ g^{-1}}(q)\leq\limsup_{n\rightarrow\infty}\frac{\log\left(\sum_{C\in G_{\nu,\left(2^{n}c_{1}\right)^{-1}}}\nu\left(C\right)^{q}\right)}{\log\left(c_{1}2^{n}\right)}=\beta_{\nu}(q).
\]
The same argument applied to $\nu\circ g^{-1}$ instead of $\nu$
and $g^{-1}$ instead of $g$ gives the reverse inequality.
\end{proof}
Next, we have to ensure that with out loss of generality we can assume
that $\nu$ has no atoms in the set of the set $\mathcal{D}^{*}\coloneqq\left\{ 2^{-n}k\mid k\in\mathbb{Z},n\in\N\right\} $
of dyadic numbers.
\begin{lem}
\label{Lemma:OBdA_keine_atome_diadiyische_Zahlen}Let $\nu$ be a
Borel measure on $\left(0,1\right)$. Then there exists an arbitrary
small constant $\alpha>0$ such that for $g_{\alpha}:x\mapsto x+\alpha$,
we have $\nu\circ g_{\alpha}^{-1}$ has no atoms in $\mathcal{D}^{*}$.
\end{lem}

\begin{proof}
Fix $\epsilon>0$. By Fubini's Theorem and the fact that $\mathcal{D}^{*}$
is countable, we have 
\begin{align*}
\int_{0}^{\epsilon}\int_{(0,1)}\1_{\mathcal{D}^{*}}(y+\alpha)\d\nu(y)\d\Lambda(\alpha) & =\int_{(0,1)}\int_{0}^{\epsilon}\1_{\mathcal{D}^{*}}(y+\alpha)\d\Lambda(\alpha)\d\nu(y)\\
 & =\int_{(0,1)}\Lambda|_{\left[0,\epsilon\right]}\left(\mathcal{D}^{*}-y\right)\d\nu(y)=0.
\end{align*}
Hence, $\int_{(0,1)}\1_{\mathcal{D}^{*}}(y+\alpha)\d\nu(y)=0$, for
$\Lambda$-a.~e\@. $\alpha\in\left[0,\epsilon\right]$. In particular,
there exists $\alpha\in\left[0,\epsilon\right]$ such that $\int_{(0,1)}\1_{\mathcal{D}^{*}}(y+\alpha)\d\nu(y)=\nu\circ g_{\alpha}^{-1}\left(\mathcal{D}^{*}\right)=\sum_{y\in\mathcal{D}^{*}}\nu\circ g_{\alpha}^{-1}\left(\left\{ y\right\} \right)=0$.
\end{proof}
Let $\nu$ be a bounded Borel probability measure with bounded support
contained in $\left(0,1\right)$. Then we find a bi-Lipschitz map
$f$ mapping the support of $\nu$ into $\left(\epsilon,1-\epsilon\right)$
for some $0<\epsilon<1$ and after a small affine linear transformation
via $x\mapsto x+\alpha$ the support of $\nu\circ f^{-1}\circ g_{\alpha}^{-1}$
is still subset of $\left(0,1\right)$ and $\nu\circ f^{-1}\circ g_{\alpha}^{-1}$
has no atoms in the dyadic numbers $\mathcal{D}^{*}$. In virtue of
Proposition \ref{Prop:bi_Lipschitz_trafo}, Lemma \ref{Lemma:OBdA_keine_atome_diadiyische_Zahlen}
and Corollary \ref{cor:enlarge_Interval_does_not_change_spectral_dimension}
these transformations and enlarging the domain also does not change
the spectral properties and the $L^{q}$-spectrum Hence, we will from
now on assume without loss of generality that $\nu$ is a probability
measure on $\left(0,1\right)$ such that $\nu\left(\mathcal{D}^{*}\right)=0$.

\subsection{\label{subsec:Upper-bounds-in}Upper bounds in terms of the $L^{q}$-spectrum}

In this section we let $\nu$ be an arbitrary Borel probability measure
on $\left(0,1\right)$ with $\card\left(\supp\nu\right)=\infty$ and
$\nu\left(\mathcal{D}^{*}\right)=0$. The idea is now to approximate
the quantity $\mathcal{N}^{R}$ by constructing $\nu$-partitions
of dyadic intervals. Let us define $\mathcal{D}\coloneqq\bigcup_{n}\mathcal{D}_{n}$
with $\mathcal{D}_{n}\coloneqq\left\{ A_{k}^{n}:k=0,...,2^{n}-1\right\} $.
The following lemma is obvious.
\begin{lem}
\label{lem:Partitionq0}For every $1>t>0$ 
\[
P_{t}\coloneqq\left\{ C\in\mathcal{D}:\nu\left(C\right)\Lambda\left(C\right)<t\,\&\,\exists C'\in\mathcal{D}_{\left|\log_{2}\left(\Lambda(C)\right)\right|-1}:C'\supset C\,\&\,\nu\left(C'\right)\Lambda\left(C'\right)\geq t\right\} 
\]
 is a $\nu$-partition.
\end{lem}

\begin{prop}
\label{prop:UpperBoundFixpoint}We always have
\[
\overline{s}_{\nu}\leq\overline{h}_{\nu}\leq\overline{q}_{\nu}.
\]
\end{prop}

\begin{proof}
We will show $\overline{q}_{\nu}\geq\overline{h}_{\nu}$ which together
with Proposition \ref{prop:estimateOfSpectralDim_N^L_m_N^R} proves
the claim. Let us choose $q_{1}>\overline{q}_{\nu}$. By Lemma \ref{lem:Partitionq0}
for $t\in\left(0,1\right)$, we have
\[
P_{t}=\left\{ C\in\mathcal{D}:\nu\left(C\right)\Lambda\left(C\right)<t\,\&\,\exists C'\in\mathcal{D}_{\left|\log_{2}\left(\Lambda(C)\right)\right|-1}:C'\supset C\,\&\,\nu\left(C'\right)\Lambda\left(C'\right)\geq t\right\} 
\]
is a $\nu$-partition of $\left[0,1\right]$ by dyadic intervals.
We also need the set 
\[
Q_{t}\coloneqq\left\{ C\in\mathcal{D}:\nu\left(C\right)\Lambda\left(C\right)\geq t\right\} .
\]
Note that for $C\in P_{t}$ there is exactly one $C'\in Q_{t}\cap\mathcal{D}_{\left|\log_{2}\left(\Lambda(C)\right)\right|-1}$
with $C\subset C'$ and for each $C'\in Q_{t}\cap\mathcal{D}_{\left|\log_{2}\left(\Lambda(C)\right)\right|-1}$
there at most two elements of $P_{t}\cap\mathcal{D}_{\left|\log_{2}\left(\Lambda(C)\right)\right|}$
such that they are subsets of $C'.$ This shows the following crucial
inequality
\[
\card P_{t}\leq2\card Q_{t}.
\]
This allows us to estimate
\begin{align*}
t^{q_{1}}\card P_{t} & =\sum_{C\in P_{t}}t^{q_{1}}\leq2\sum_{C\in Q_{t}}t^{q_{1}}\leq2\sum_{C\in Q_{t}}\left(\nu\left(C\right)\Lambda\left(C\right)\right)^{q_{1}}\leq2\sum_{C\in\mathcal{D}}\left(\nu\left(C\right)\Lambda\left(C\right)\right)^{q_{1}}.
\end{align*}
With this at hand and with the help of Fact \ref{fact:CriticalExponentBETA}
we find that the right hand side is finite and therefore we conclude
$\limsup_{t\searrow0}-\log\left(\card P_{t}\right)/\log\left(t\right)\leq q_{1}$.
By Proposition \ref{prop:estimateOfSpectralDim_N^L_m_N^R} we finally
conclude
\[
\overline{s}_{\nu}\leq\overline{h}_{\nu}\leq\limsup_{t\searrow0}\frac{\log\card P_{t}}{-\log t}\leq\overline{q}_{\nu}.
\]
\end{proof}
\begin{proof}
[Proof of Proposition  \ref{prop:upperBoundForUnderline_s<Mdim}]Since
\[
\frac{\log\left(\mathcal{N}^{R}\left(2^{n-1}\right)\right)}{\log\left(2^{n-1}\right)}\leq\frac{\log\left(2^{n\beta_{n}(0)+1}+1\right)}{\log\left(2^{n-1}\right)}
\]
it follows that $\underline{s}\leq\underline{h}_{\nu}\leq\underline{\dim}_{M}(\nu).$
\end{proof}
\begin{prop}
\label{prop:UppperboundForLinear_s_<liminfq_n} Under the assumption
that there exists a subsequence $(n_{k})_{k\in\N}\in\N^{\N}$ and
a constant $K>0$ such that for all $k\in\N$ 
\[
\max_{C\in\mathcal{D}_{\nu,n_{k}}}\nu(C)^{q_{n_{k}}}\leq\frac{K}{2^{\beta_{n_{k}}(0)n_{k}}}\sum_{C\in\mathcal{D}_{\nu,n_{k}}}\nu(C)^{q_{n_{k}}}
\]
and $\lim_{k\rightarrow\infty}q_{n_{k}}=\liminf_{n\rightarrow\infty}q_{n}$,
where $q_{n}\geq0$ is the unique solution to $\beta_{n}(q_{n})=q_{n}$,
we have 
\[
\underline{s}_{\nu}\leq\underline{h}_{\nu}\leq\liminf_{n\rightarrow\infty}q_{n}\leq\liminf_{n\rightarrow\infty}\frac{\beta_{n}(0)}{1+\beta_{n}(0)}=\frac{\underline{\dim}_{M}(\nu)}{1+\underline{\dim}_{M}(\nu)}.
\]
\end{prop}

\begin{proof}
We only consider the case $\underline{\dim}_{M}(\nu)>0.$ Since for
all $k\in\N$,
\[
2^{q_{n_{k}}n_{k}}=\sum_{C\in\mathcal{D}_{\nu,n_{k}}}\nu(C)^{q_{n_{k}}},
\]
we have
\[
\max_{C\in\mathcal{D}_{\nu,n_{k}}}\nu(C)^{q_{n_{k}}}\leq\frac{K}{2^{\beta_{n_{k}}(0)n_{k}}}\sum_{C\in\mathcal{D}_{\nu,n_{k}}}\nu(C)^{q_{n_{k}}}=K\frac{2^{n_{k}q_{n_{k}}}}{2^{n_{k}\beta_{n_{k}}(0)}}.
\]
Hence, 
\[
\max_{C\in\mathcal{D}_{\nu,n_{k}}}\nu(C)\Lambda(C)\leq\frac{K^{1/q_{n_{k}}}}{2^{n_{k}\beta_{n_{k}}(0)/q_{n_{k}}}}.
\]
Using $\underline{\dim}_{M}(\nu)>0$ we find $n_{k}\beta_{n_{k}}\left(0\right)\to\infty$,
$k\to\infty$. Thus we obtain 
\begin{align*}
\frac{\log\left(\mathcal{N}^{R}\left(\frac{2^{\beta_{n_{k}}(0)n_{k}/q_{n_{k}}}}{2K^{1/q_{n_{k}}}}\right)\right)}{\log\left(\frac{2^{\beta_{n_{k}}(0)n_{k}/q_{n_{k}}}}{2K^{1/q_{n_{k}}}}\right)} & \leq\frac{\log\left(22^{n_{k}\beta_{n_{k}}(0)}+1\right)}{\log\left(2^{n_{k}\beta_{n_{k}}(0)}\right)/q_{n_{k}}-\left(\log(K^{1/q_{n_{k}}})-\log(2)\right)}.\\
 & \leq\frac{\log\left(4\right)}{\log\left(2^{n_{k}\beta_{n_{k}}(0)}\right)/q_{n_{k}}-\left(\log(K^{1/q_{n_{k}}})-\log(2)\right)}\\
 & \hspace*{1em}\hspace*{1em}\hspace*{1em}+\frac{1}{1/q_{n_{k}}-\left(\log(K^{1/q_{n_{k}}})-\log(2)\right)/\log\left(2^{\beta_{n_{k}}(0)n_{k}/q_{n_{k}}}\right)}.
\end{align*}
Therefore, $\underline{s}_{\nu}\leq\underline{h}_{\nu}\leq\liminf_{n\rightarrow\infty}q_{n}.$
\end{proof}

\subsection{\label{subsec:Lower-bounds-in}Lower bounds in terms of the optimize
coarse multifractal dimension}

Here we again assume that $\nu$ is a finite Borel measure on $\left(0,1\right)$
with $\card\left(\supp\nu\right)=\infty$ and $\nu\left(\mathcal{D}^{*}\right)=0$.
To find a good lower bound, we use the concept of the \emph{upper
}and \emph{lower optimize coarse multifractal dimension $\underline{F}_{\nu}$},
$\overline{F}_{\nu}$ as defined in the introduction.
\begin{lem}
\label{lem:CoarseEstimate} For fixed $\alpha>0$ and with $n_{x}^{\alpha}\coloneqq\left\lfloor \log\left(\frac{x}{2}\right)/\left(\log(2)(\alpha+1)\right)\right\rfloor $
we have for every $x>2^{3+\alpha}$ 
\[
N_{\nu}(x)\geq\mathcal{N}_{3}^{L}\left(x\right)\geq\frac{\mathcal{N_{\alpha}}\left(n_{x}^{\alpha}\right)}{3}-3,
\]
or with $x_{n}^{\alpha}\coloneqq2^{n\left(\alpha+1\right)+1}$ we
have for every $n\in\N$
\[
N_{\nu}(x_{n}^{\alpha})\geq\frac{\mathcal{N_{\alpha}}\left(n\right)}{3}-3.
\]
\end{lem}

\begin{proof}
For fixed $n\in\N$, $\alpha>0$ let $\left(c_{1},d_{1}\right],\dots,\left(c_{\mathcal{N}_{\alpha}\left(n\right)},d_{\mathcal{N}_{\alpha}\left(n\right)}\right]$
denote the intervals of $M_{n}\left(\alpha\right)$ ordered in the
natural way, i.~e\@. $d_{i}<c_{i+1}$. For $i=0,\dots,\left\lfloor \mathcal{N}_{\alpha}\left(n\right)/3-2\right\rfloor -1$
define
\[
D_{n,i}\coloneqq\left(c_{2+3i}-\frac{1}{2^{n}},c_{2+3i}\right]\cup\left(c_{2+3i},d_{2+3i}\right]\cup\left(d_{2+3i},d_{2+3i}+\frac{1}{2^{n}}\right].
\]
Note that we have
\[
\nu\left(\left\langle D_{n,i}\right\rangle _{3}\right)\Lambda\left(\left\langle D_{n,i}\right\rangle _{3}\right)=\nu\left(\left(c_{2+3i},d_{2+3i}\right]\right)\Lambda\left(\left(c_{2+3i},d_{2+3i}\right]\right)\geq2^{-n(\alpha+1)},
\]
hence for $n_{x}^{\alpha}\coloneqq\left\lfloor \log\left(\frac{x}{2}\right)/\left(\log(2)(\alpha+1)\right)\right\rfloor $
and $x>2^{3+\alpha}$, we have
\[
\nu\left(\left\langle D_{n_{x}^{\alpha},i}\right\rangle _{3}\right)\Lambda\left(\left\langle D_{n_{x}^{\alpha},i}\right\rangle _{3}\right)\geq\frac{2}{x}
\]
which shows 
\begin{align*}
\mathcal{N}_{3,\nu}^{L}\left(x\right) & =\sup\left\{ \card\left(P\right):P\in\Pi_{0}:\min_{C\in P}\nu\left(\left\langle I\right\rangle _{3}\right)\Lambda\left(\left\langle I\right\rangle _{3}\right)\geq\frac{2}{x}\right\} \\
 & \geq\left\lfloor \frac{\mathcal{N}_{\alpha}\left(n_{x}^{\alpha}\right)}{3}-2\right\rfloor \geq\frac{\mathcal{N}_{\alpha}\left(n_{x}^{\alpha}\right)}{3}-3.
\end{align*}
\end{proof}
Now we can give a lower bound on $\underline{s}_{\nu}$ and $\overline{s}_{\nu}$
in terms of the lower and upper optimize coarse multifractal dimension.
\begin{prop}
\label{prop:lower_Bound_Multifractal_F_=00005Cnu} As a general lower
(upper) bound for the lower (upper) spectral dimension for all $m\in\left(1,3\right]$,
we have
\[
\underline{F}_{\nu}\leq\underline{h}_{\nu}^{m}\leq\underline{s}_{\nu}\;\text{and }\;\overline{F}_{\nu}\leq\overline{h}_{\nu}^{m}\leq\overline{s}_{\nu}.
\]
\end{prop}

\begin{proof}
First note that Proposition \ref{prop:estimateOfSpectralDim_N^L_m_N^R}
gives $\underline{h}_{\nu}^{m}\leq\underline{s}_{\nu}$ and $\overline{h}_{\nu}^{m}\leq\overline{s}_{\nu}$
for all $m>1$. With $n_{x}^{\alpha}\coloneqq\left\lfloor \frac{\log\left(\frac{x}{2}\right)}{\log(2)(\alpha+1)}\right\rfloor $
and $x>2^{3+\alpha}$ with Lemma \ref{lem:CoarseEstimate} for every
$\alpha\geq0$ and $x>2^{3+\alpha}$, we have
\[
\frac{\log^{+}\left(3\mathcal{N}_{3}^{L}\left(x\right)+3\right)}{\log(x)}\geq\frac{\log^{+}\left(\mathcal{N}_{\alpha}\left(n_{x}^{\alpha}\right)\right)}{\log x}.
\]
Hence, for all $1<m\leq3$,
\begin{align*}
\underline{h}_{\nu}^{m}\geq\underline{h}_{\nu}^{3} & =\liminf_{x}\frac{\log^{+}\left(3\mathcal{N}_{3}^{L}\left(x\right)+3\right)}{\log x}\geq\liminf_{x}\frac{\log^{+}\left(\mathcal{N}_{\alpha}\left(n_{x}^{\alpha}\right)\right)}{\log x}\\
 & =\liminf_{x}\frac{\log^{+}\left(\mathcal{N}_{\alpha}\left(n_{x}^{\alpha}\right)\right)}{(1+\alpha)\log\left(2^{\log_{2}\left(\frac{x}{2}\right)/(1+\alpha)}\right)+\log\left(2\right)}\\
 & \geq\liminf_{x}\frac{\log^{+}\left(\mathcal{N}_{\alpha}\left(n_{x}^{\alpha}\right)\right)}{(1+\alpha)\log2^{n_{x}^{\alpha}}+\log\left(2\right)(2+\alpha)}\\
 & =\liminf_{x}\frac{\log^{+}\left(\mathcal{N}_{\alpha}\left(n_{x}^{\alpha}\right)\right)}{\left(1+\alpha\right)\log2^{n_{x}^{\alpha}}}\geq\liminf_{n\rightarrow\infty}\frac{\log^{+}\left(\mathcal{N}_{\alpha}\left(n\right)\right)}{\left(1+\alpha\right)\log2^{n}},
\end{align*}
which implies $\underline{h}_{\nu}^{m}\geq\underline{F}_{\nu}$. We
also have for $x_{m}^{\eta}\coloneqq2^{m\left(\eta+1\right)+1}$ with
$\eta>0$,
\begin{align*}
\overline{h}_{\nu}^{m} & \geq\overline{h}_{\nu}^{3}\geq\limsup_{m\rightarrow\infty}\frac{\log^{+}\left(3\mathcal{N}_{3}^{L}\left(x_{m}^{\eta}\right)+3\right)}{\log(x_{m}^{\eta})}\geq\limsup_{m\rightarrow\infty}\frac{\log^{+}\left(\mathcal{N}_{\eta}\left(n_{x_{m}^{\eta}}^{\eta}\right)\right)}{\log x_{m}^{\eta}}\\
 & \geq\limsup_{m\rightarrow\infty}\frac{\log^{+}\left(\mathcal{N}_{\eta}\left(n_{x_{m}^{\eta}}^{\eta}\right)\right)}{\left(1+\eta\right)\log2^{n_{x_{m}^{\eta}}^{\eta}}}=\limsup_{m\rightarrow\infty}\frac{\log^{+}\left(\mathcal{N}_{\eta}\left(m\right)\right)}{\left(1+\eta\right)\log2^{m}},
\end{align*}
where we used\textbf{ $n_{x_{m}^{\eta}}^{\eta}=m$}. Thus, $\overline{h}_{\nu}^{m}\geq\overline{F}_{\nu}$.
\end{proof}
We need the following elementary observation from large deviation
theory which seems not to be standard in the relevant literature.
\begin{lem}
\label{lem:exponential decay-1} Suppose $\left(X_{n}\right)_{n\in\N}$
are real-valued random variables on some probability spaces $\left(\Omega_{n},\mathcal{A}_{n},\mu_{n}\right)$
such that the rate function $\mathfrak{c}\left(t\right)\coloneqq\limsup_{n\to\infty}\mathfrak{c}_{n}\left(t\right)$
is a proper convex function with $\mathfrak{c}_{n}\left(t\right)\coloneqq a_{n}^{-1}\log\int\exp tX_{n}\d\mu_{n}$,
$t\in\R$, $a_{n}\rightarrow\text{\ensuremath{\infty}}$ and such
that $0$ belongs to the interior of the domain of finiteness $\left\{ t\in\R\colon\mathfrak{c}\left(t\right)<\infty\right\} $.
Let $I=(a,d)$ be an open interval containing the subdifferential
$\partial\mathfrak{c}\left(0\right)=[b,c]$ of $\mathfrak{c}$ in
$0$. Then there exists $r>0$ such that for all $n$ sufficiently
large, 
\[
\mu_{n}\left(a_{n}^{-1}X_{n}\notin I\right)\leq2\exp\left(-ra_{n}\right).
\]
\end{lem}

\begin{proof}
We assume that $\partial\mathfrak{c}\left(0\right)=\left[b,c\right]$
and $I=\left(a,d\right)$ with $a<b\leq c<d$. First note that the
assumptions ensure that $-\infty<b\leq c<\infty.$ By Chebyshev inequality
for all $q>0$,
\begin{align*}
\mu_{n}\left(a_{n}^{-1}X_{n}\geq d\right) & =\mu_{n}\left(qX_{n}\geq qa_{n}d\right)\leq\exp\left(-qa_{n}d\right)\int\exp\left(qX_{n}\right)\d\mu_{n}
\end{align*}
 implying 
\[
\limsup a_{n}^{-1}\log\mu_{n}\left(a_{n}^{-1}X_{n}\geq d\right)\leq\inf_{q>0}\mathfrak{c}\left(q\right)-qd=\inf_{q\in\R}\mathfrak{c}\left(q\right)-qd,
\]
where the equality follows from the assumption $d\notin\partial\mathfrak{c}\left(0\right)$
and $\mathfrak{c}\left(q\right)-qd\geq(c-d)q\geq0$ for all $q\leq0.$
Similarly, we find 
\[
\limsup a_{n}^{-1}\log\mu_{n}\left(a_{n}^{-1}X_{n}\leq a\right)\leq\inf_{q<0}\mathfrak{c}\left(q\right)-qa=\inf_{q\in\R}\mathfrak{c}\left(q\right)-qa.
\]
We are left to show that both upper bounds are negative. We show the
first case by contradiction -- the other case follows in exactly
the same way. Assuming $\inf_{q\in\R}\mathfrak{c}\left(q\right)-qd=0$
implies for all $q\in\R$ that $\mathfrak{c}\left(q\right)-qd\geq0$,
or after rearranging, $\mathfrak{c}\left(q\right)-\mathfrak{c}\left(0\right)\geq dq$.
This means, according to the definition of the subdifferential, that
$d\in\partial\mathfrak{c}\left(0\right)$, contradicting our assumptions.
\end{proof}
\begin{prop}
\label{prop:GeneralBound.}Assume for some $q\in\left[0,1\right]$
that $\beta_{\nu}(q)=\lim\beta_{n_{k}}^{\nu}\left(q\right)$ on a
subsequence $\left(n_{k}\right)$. Then we have $[a,b]=-\partial\beta_{\nu}\left(q\right)\supset\left[a',b'\right]=-\partial\limsup_{k}\beta_{n_{k}}\left(q\right)$
and 
\begin{align*}
\frac{aq+\beta_{\nu}(q)}{1+b} & \leq\frac{a'q+\beta_{\nu}(q)}{1+b'}\leq\sup_{t>b}\liminf_{k\to\infty}\frac{\log\mathcal{N}_{t}\left(n_{k}\right)}{\log2^{n_{k}}(1+t)}\leq\sup_{t>b}\limsup_{k\to\infty}\frac{\log\mathcal{N}_{t}\left(n_{k}\right)}{\log2^{n_{k}}(1+t)}.
\end{align*}
\end{prop}

\begin{proof}
Without of loss of generality we can assume $b<\infty.$ The fact
that $-\partial\beta_{\nu}\left(q\right)\supset-\partial\limsup_{k}\beta_{n_{k}}\left(q\right)$
follows immediately from $\limsup_{k}\beta_{n_{k}}\leq\beta_{\nu}$.
It suffices to prove the second inequality. Now, for all $k\in\N$
and $s<a\leq b<t$, we have
\begin{align*}
\mathcal{N}_{t}\left(n_{k}\right) & \geq\card\underbrace{\left\{ 0\leq\ell<2^{n_{k}}:2^{-sn_{k}}>\nu\left(A_{n_{k}}^{\ell}\right)>2^{-tn_{k}}\right\} }_{\coloneqq L_{n_{k}}^{s,t}}\\
 & \geq\sum_{\ell\in L_{n_{k}}^{s,t}}\nu\left(A_{n_{k}}^{\ell}\right)^{q}2^{sn_{k}q}\geq2^{sn_{k}q+n_{k}\beta_{n_{k}}\left(q\right)}\sum_{1\leq\ell\leq2^{n_{k}}}\1_{L_{n_{k}}^{s,t}}\left(\ell\right)\nu\left(A_{n_{k}}^{\ell}\right)^{q}2^{-n_{k}\beta_{n_{k}}\left(q\right)}\\
 & =2^{sn_{k}q+n_{k}\beta_{n_{k}}\left(q\right)}\left(1-\sum_{1\leq\ell\leq2^{n_{k}}}\1_{\left(L_{n_{k}}^{s,t}\right)^{\complement}}\left(\ell\right)\nu\left(A_{n_{k}}^{\ell}\right)^{q}2^{-n_{k}\beta_{n_{k}}\left(q\right)}\right).
\end{align*}
We use the lower large deviation principle for the process $X_{k}\left(\ell\right)\coloneqq\log^{+}\nu\left(A_{n_{k}}^{\ell}\right)$
with probability measure on $\left\{ 1,\ldots,2^{n_{k}}\right\} $
given by $\mu_{k}\left(\left\{ \ell\right\} \right)\coloneqq\nu\left(A_{n_{k}}^{\ell}\right)^{q}2^{-n_{k}\beta_{n_{k}}\left(q\right)}$.
We find for the free energy function
\begin{align*}
\mathfrak{c}\left(x\right) & \coloneqq\limsup_{n_{k}}\frac{1}{\log2^{n_{k}}}\log\left(\mathbb{E}_{\mu_{k}}\left(\exp xX_{k}\right)\right)=\limsup_{n_{k}}\frac{1}{\log2^{n_{k}}}\log\left(\sum_{C\in\mathcal{D}_{n_{k}}}\nu\left(C\right)^{x+q}/2^{n_{k}\beta_{n_{k}}\left(q\right)}\right)\\
 & =\limsup_{n_{k}}\beta_{n_{k}}(q+x)-\beta_{\nu}\left(q\right),
\end{align*}
with $-\partial\mathfrak{c}\left(0\right)=\left[a',b'\right]\subset[a,b]\subset(s,t)$
and hence, by Lemma \ref{lem:exponential decay-1}, there exists a
constant $r>0$ depending on $s,t$ and $q$ such that for $k$ large,
\[
\sum_{1\leq\ell\leq2^{n_{k}}}\1_{\left(L_{n_{k}}^{s,t}\right)^{\complement}}\left(\ell\right)\nu\left(A_{n_{k}}^{\ell}\right)^{q}/2^{n_{k}\beta_{n_{k}}\left(q\right)}=\mu_{k}\left(\frac{X_{k}}{\log\left(2^{n_{k}}\right)}\notin(-t,-s)\right)\leq2\exp\left(-rn_{k}\right).
\]
Therefore,
\begin{align*}
\liminf_{k\to\infty}\frac{\log\mathcal{N}_{t}\left(n_{k}\right)}{\log2^{n_{k}}} & \geq sq+\beta_{\nu}(q)
\end{align*}
for all $s<a$ and $t>b$ and therefore
\[
\sup_{t>b}\liminf_{k\to\infty}\frac{\log\mathcal{N}_{t}\left(n_{k}\right)}{\log2^{n_{k}}(1+t)}\geq\sup_{t>b}\frac{aq+\beta_{\nu}(q)}{1+t}=\frac{aq+\beta_{\nu}(q)}{1+b}.
\]
\end{proof}
\begin{prop}
\label{prop:FixpointLowerBound}For all $m\in(1,3]$, we have 
\[
\overline{q}_{\nu}=\overline{F}_{\nu}=\overline{h}_{\nu}^{m}=\overline{s}_{\nu}.
\]
\end{prop}

\begin{proof}
We only have to consider the case $\overline{q}_{\nu}>0$. First note
that the family of convex function $\left(\beta_{n}\right)$ restricted
to $\left[0,1\right]$ only takes values in $\left[0,1\right]$ and
on any compact interval $\left[c,d\right]\subset\left(0,1\right)$
and for all $c\leq x\leq y\leq d$, we have

\[
\frac{\beta_{n}\left(x\right)-\beta_{n}\left(0\right)}{x-0}\leq\frac{\beta_{n}\left(y\right)-\beta_{n}\left(x\right)}{y-x}\leq\frac{\beta_{n}\left(1\right)-\beta_{n}\left(y\right)}{1-y},
\]
 which implies 
\[
\left|\beta_{n}\left(y\right)-\beta_{n}\left(x\right)\right|\leq\max\left\{ \frac{1}{c},\frac{1}{1-d}\right\} \left|x-y\right|
\]
and hence $\left(\beta_{n}\right)$ is uniformly Lipschitz and by
Arzelà--Ascoli relatively compact. Using this fact we find a subsequence
$\left(n_{k}\right)$ such that $\lim_{k}\beta_{n_{k}}\left(\overline{q}_{\nu}\right)=\limsup\beta_{n}\left(\overline{q}_{\nu}\right)=\overline{q}_{\nu}$
and $\beta_{n_{k}}$ converges uniformly to the convex function $\beta^{\sharp}$
on $[\overline{q}_{\nu}-\delta,\overline{q}_{\nu}+\delta]\subset[0,1]$
for $\delta$ sufficiently small and put $[a,b]=-\partial\beta^{\sharp}\left(\overline{q}_{\nu}\right)$.
Then by the convexity of $\beta^{\sharp}$ we find for every $\delta>\epsilon>0$
an element $q\in\left(\overline{q}_{\nu}-\epsilon,\overline{q}_{\nu}\right)$
such that $\beta^{\sharp}$ is differentiable in $q$ with $-\left(\beta^{\sharp}\right)'\left(q\right)\in[b,b+\epsilon]$
since the points where $\beta^{\sharp}$ is differentiable on $(0,1)$
lie dense in $(0,1)$ which follows from the fact that $\beta^{\sharp}$
is a decreasing function and left-hand derivatives of convex functions
are left-continuous and non-decreasing. In the same way as in the
proof Proposition \ref{prop:GeneralBound.}, we obtain,
\begin{align*}
\sup_{t>b}\limsup_{n\to\infty}\frac{\log\mathcal{N}_{t}\left(n\right)}{\log2^{n}(1+t)} & \geq\sup_{t>b}\limsup_{k\to\infty}\frac{\log\mathcal{N}_{t}\left(n_{k}\right)}{\log2^{n_{k}}(1+t)}\\
 & \geq\frac{(-\left(\beta^{\sharp}\right)'\left(q\right)q+\beta^{\sharp}\left(q\right)}{1-\left(\beta^{\sharp}\right)'\left(q\right)}\\
 & \geq\frac{b(\overline{q}_{\nu}-\varepsilon)+\beta^{\sharp}\left(\overline{q}_{\nu}\right)}{1+b+\varepsilon}\\
 & =\frac{b(\overline{q}_{\nu}-\varepsilon)+\overline{q}_{\nu}}{1+b+\varepsilon}.
\end{align*}
Taking $\epsilon\to0$ 
\[
\overline{F}_{\nu}=\sup_{\alpha\geq0}\frac{\overline{F}_{\nu}\left(\alpha\right)}{1+\alpha}\geq\overline{q}_{\nu}.
\]
The second part of Proposition \ref{prop:lower_Bound_Multifractal_F_=00005Cnu}
and Proposition \ref{prop:UpperBoundFixpoint} give the claimed equalities.
\end{proof}

\subsection{Proof of main results}

Now, we are able to prove the main results by combining the results
of Section \ref{subsec:Lower-bounds-in} and Section \ref{subsec:Upper-bounds-in}.
Recall from the introduction that we call $\nu$ \emph{regular} if
$\underline{F}_{\nu}=\overline{F}_{\nu}$ and $L^{q}$\emph{-regular},
if $\beta_{\nu}\left(q\right)=\liminf_{n}\beta_{n}^{\nu}\left(q\right)$
for $q\in(\overline{q}_{\nu}-\varepsilon,\overline{q}_{\nu})$, for
some $\varepsilon>0$, or $\beta_{\nu}\left(\overline{q}_{\nu}\right)=\liminf_{n}\beta_{n}^{\nu}\left(\overline{q}_{\nu}\right)$
and $\beta_{\nu}$ is differentiable in $\overline{q}_{\nu}$.
\begin{proof}
[Proof of Theorem \ref{thm:LqRegularImpliesRegular}]Since by Proposition
\ref{prop:FixpointLowerBound} we have $\overline{F}_{\nu}=\overline{q}_{\nu}$,
it suffices to show that $L^{q}$-regularity implies
\[
\underline{F}_{\nu}\geq\overline{q}_{\nu}.
\]
Indeed, let us assume $[a,b]=-\partial\beta_{\nu}\left(\overline{q}_{\nu}\right)$.
Then by the convexity of $\beta_{\nu}$ we find for every $\epsilon\in\left(0,\overline{q}_{\nu}\right)$
an element $q\in\left(\overline{q}_{\nu}-\epsilon,\overline{q}_{\nu}\right)$
such that $\beta_{\nu}$ is differentiable in $q$ with $-\beta_{\nu}'\left(q\right)\in[b,b+\epsilon]$
since the points where $\beta_{\nu}$ is differentiable on $(0,1)$
lie dense in $(0,1)$ which follows from the fact that $\beta_{\nu}$
is a decreasing function and the left-hand continuity of left-hand
derivative of the convex function $\beta_{\nu}$. Then by Proposition
\ref{prop:GeneralBound.}, we have
\begin{align*}
\sup_{t\geq0}\liminf_{n\to\infty}\frac{\log^{+}\left(\mathcal{N}_{t}\left(n\right)\right)}{\log2^{n}(1+t)} & \geq\sup_{t>-\beta'\left(q\right)}\liminf_{n\to\infty}\frac{\log^{+}\left(\mathcal{N}_{t}\left(n\right)\right)}{\log2^{n}(1+t)}\\
 & \geq\frac{-\beta_{\nu}'\left(q\right)q+\beta_{\nu}\left(q\right)}{1-\beta_{\nu}'\left(q\right)}\geq\frac{b(\overline{q}_{\nu}-\varepsilon)+\overline{q}_{\nu}}{1+b+\varepsilon}.
\end{align*}
Taking the limit $\epsilon\to0$ 

\[
\sup_{\alpha\geq0}\liminf_{n\to\infty}\frac{\log^{+}\left(\mathcal{N}_{\alpha}\left(n\right)\right)}{\log2^{n}(1+\alpha)}\geq\frac{b\overline{q}_{\nu}+\overline{q}_{\nu}}{1+b}=\overline{q}_{\nu}.
\]

The case $\beta_{\nu}$ exists as limit in $\overline{q}_{\nu}$ and
is differentiable in $\overline{q}_{\nu}$ is covered by Proposition
\ref{prop:GeneralBound.}.
\end{proof}
\begin{proof}
[Proof of Theorem \ref{Thm:MainChain_of_Inequalities+Regularity}]The
first claim follows from Propositions \ref{prop:UpperBoundFixpoint},
\ref{prop:lower_Bound_Multifractal_F_=00005Cnu}, \ref{prop:estimateOfSpectralDim_N^L_m_N^R}
and \ref{prop:FixpointLowerBound}.
\end{proof}
\begin{proof}
[Proof of Corollary \ref{cor:Sum_of_Measures}] The first part follows
immediately from Theorem \ref{Thm:MainChain_of_Inequalities+Regularity}
in tandem with Fact \ref{fact:MGF_for_sum}. For the second part note
that by Lemma \ref{lem general} 
\[
s_{\nu_{1}}\leq\max\left\{ \underline{s}_{\nu_{1}},\underline{s}_{\nu_{1}}\right\} \leq\underline{s}_{\nu_{1}+\nu_{2}}\leq\overline{s}_{\nu_{1}+\nu_{2}}=\overline{F}_{\nu_{1}+\nu_{2}}=\max\left\{ \overline{s}_{\nu_{1}},\overline{s}_{\nu_{2}}\right\} =s_{\nu_{1}}
\]
and using monotonicity properties\textbf{ }of the coarse optimized
multifractal dimension
\[
\overline{F}_{\nu_{1}}=\max\left\{ \underline{F}_{\nu_{1}},\underline{F}_{\nu_{2}}\right\} \leq\underline{F}_{\nu_{1}+\nu_{2}}\leq\overline{F}_{\nu_{1}+\nu_{2}}=s_{\nu_{1}}=\overline{F}_{\nu_{1}}.
\]
\end{proof}
\begin{proof}
[Proof of  Corollary \ref{cor:spec_absolutely_cont}] First, we
show that for an absolutely continuous measure $\nu$ the $L^{q}$-spectrum
always exists as limit on $[0,1]$. By Hölder's inequality for all
$C\in\mathcal{D}_{\nu,n}$ and $q\in\left(0,1\right)$, we have
\begin{align*}
2^{n}\int_{C}\rho^{q}\d\Lambda & \leq\left(\int_{C}2^{n/(1-q)}\d\Lambda\right)^{1-q}\left(\int_{C}\rho\d\Lambda\right)^{q}=\left(2^{n\frac{q}{1-q}}\right)^{1-q}\left(\int_{C}\rho\d\Lambda\right)^{q}=2^{nq}\left(\int_{C}\rho\d\Lambda\right)^{q}.
\end{align*}
This leads to
\[
\beta_{n}^{\nu}(q)\geq\frac{\log\left(\int_{[0,1]}\rho^{q}\d\Lambda\right)}{\log(2^{n})}+1-q,
\]
implying $1-q\leq\liminf_{n}\beta_{n}^{\nu}(q)\leq\limsup_{n}\beta_{n}^{\nu}(q)\leq1-q$
for all $q\in[0,1].$ Now, the remaining claims follow from Fact \ref{fact:MGF_for_sum}
and Corollary \ref{cor:q-1 implies s0=00003D1/2}.
\end{proof}
\begin{proof}
[Proof of  Proposition \ref{prop:loverbound_by_diff_in_1}]By Proposition
\ref{prop:GeneralBound.}, 
\[
\frac{aq+\beta_{\nu}(q)}{1+b}\leq\sup_{t>b}\liminf_{n\to\infty}\frac{\log\mathcal{N}_{t}\left(n\right)}{\log2^{n}(1+t)}\leq\underline{F}_{\nu}\leq\underline{s}_{\nu}.
\]
\end{proof}
\begin{proof}
[Proof of Corollary \ref{cor:spectralDimGeneralUpperBound}] For
the case $\delta_{\nu}^{*}=$0 we combine Fact \ref{fact:SmallSupport}
and the upper estimate from Proposition \ref{prop:UpperBoundFixpoint}.
If $\delta_{\nu}^{*}>0$, this is a direct consequence of Proposition
\ref{prop:UpperBoundFixpoint} and the general properties of $\beta_{\nu}$
on $\left[0,1\right]$ as a convex non-increasing function not exceeding
the line $t\mapsto\delta_{\nu}^{*}\left(1-t\right)$ which intersects
the bisector in $\delta_{\nu}^{*}/\left(1+\delta_{\nu}^{*}\right)$.
The last claim follows from the fact that $\beta_{\nu}$ is linear
on $\left[0,1\right]$ if and only if $\overline{\delta}_{\nu}=\delta_{\nu}^{*}$
and in this case $\overline{q}_{\nu}=\delta_{\nu}^{*}/\left(1+\delta_{\nu}^{*}\right)$.
The lower bound follows from Proposition \ref{prop:loverbound_by_diff_in_1}
using the fact that $\beta_{\nu}$ always exists as a limit in $1$.
\end{proof}
\begin{rem}
The case $\delta_{\nu}^{*}=0$ immediately gives $\overline{s}_{\nu}=0$.
If we use more information on $\beta_{\text{\ensuremath{\nu}}}$ for
the case $\delta_{\nu}^{*}>0$ we find a better upper bound; namely,
with $q_{1}\coloneqq\inf\left\{ s:\beta_{\nu}\left(s\right)\leq0\right\} $
we have
\[
\overline{s}_{\nu}\leq\overline{q}_{\nu}\leq\frac{q_{1}\delta_{\nu}^{*}}{q_{1}+\delta_{\nu}^{*}}.
\]
\end{rem}

\begin{proof}
[Proof of  Corollary \ref{cor:Strongly regular max dim}] By Theorem
\ref{Thm:MainChain_of_Inequalities+Regularity} we have $\overline{q}_{\nu}=\overline{s}_{\nu}$
and if $\nu$ is $L^{q}$-regular, the claim then follows from Corollary
\ref{cor:spectralDimGeneralUpperBound}. In \citep[Theorem 3.1]{MR2334791}
it has been shown that $\overline{\delta}_{\nu}\geq\dim_{P}\left(E\right)\geq\dim_{H}\left(\nu\right)\geq\underline{\dim}_{H}\left(\nu\right)\geq\underline{\delta}_{\nu}$
and this gives the second claim as a consequence of the first equivalence.

The last claim follows from the convexity of $\beta_{\nu}$, $\beta_{\nu}(0)\leq1$
and $\beta_{\nu}(1)=0$ in tandem with $\beta_{\nu}\left(\delta_{\nu}^{*}/\left(1+\delta_{\nu}^{*}\right)\right)=\delta_{\nu}^{*}/\left(1+\delta_{\nu}^{*}\right)$
which implies $\beta_{\nu}(q)=\delta_{\nu}^{*}(1-q)$ for $q\in[0,1]$.
Now, \citep[Theorem 3.1]{MR2334791} implies $\delta_{\nu}^{*}=\dim_{H}(\nu)=\dim_{P}\left(\nu\right).$
\end{proof}
\begin{proof}
[Proof of  Corollary \ref{cor:q-1 implies s0=00003D1/2}] If $\overline{s}_{\nu}=1/2$
then it follows from Proposition \ref{prop:UpperBoundFixpoint} that
$1/2=\overline{s}_{\nu}\leq\overline{q}_{\nu}\leq1/2.$ The convexity
of $\beta_{\nu}$ and the fact that $\beta_{\nu}\left(1\right)=0$
and $\beta_{\nu}\left(0\right)\leq1$ forces $\beta_{\nu}(q)=1-q$
for all $q\in[0,1]$. The second statement is an immediate consequence
of Proposition \ref{prop:GeneralBound.} by observing that, as in
case (1), by convexity we have $\beta_{\nu}(q)=1-q$ for all $q\in[0,1]$.
This implies the differentiability of $\beta_{\nu}$ in the particular
point $q\in\left(0,1\right)$, where by assumption $\beta_{\nu}\left(q\right)=\lim\beta_{n}^{\nu}\left(q\right)$.
Now, applying Proposition \ref{prop:GeneralBound.} gives
\[
\frac{1}{2}=\frac{q+1-q}{1+1}\leq\underline{s}_{\nu}.
\]
Since we always have $\overline{s}_{\nu}\leq1/2$, our claim follows.
\end{proof}

\section{Examples\label{sec:Examples}}

\subsection{$C^{1}$-self-conformal measures\label{sec:Applicaiton-to-self-conformal}}

For simplicity, let us start with the self-similar case.
\begin{example}
Here, we consider for fixed $m\in\N$ self-similar probability measure
$\nu$ with probability vector $(p_{1},\dots,p_{m})\in(0,1)^{m}$
and contractions $T_{i}(x)=\sigma_{i}x+b_{i}$ for $i=1,\dots,m$
with $\sigma_{i}\in(0,1)$ which maps from $[0,1]$ to $[0,1]$ and
fulfills the open set condition (i.~e\@. $T_{i}\left(\left(0,1\right)\right)\cap T_{j}\left(\left(0,1\right)\right)=\emptyset$
for all $j\neq i)$. That is the unique Borel measure which satisfy
\[
\nu(A)=\sum_{i=1}^{m}p_{i}\nu\left(T_{i}^{-1}(A)\right),A\mathfrak{\in B}([0,1]).
\]
 Using \citep[Theorem 15]{MR1312056} we obtain for $q\geq0$ that
$\beta_{\nu}(q)$ is given by 
\[
\sum_{i=1}^{m}p_{i}^{q}\sigma_{i}^{\beta_{\nu}(q)}=1.
\]
Thus, we can conclude that $\overline{q}_{\nu}$ is determined by
the equality $\sum_{i=1}^{m}\left(p_{i}\sigma_{i}\right)^{\overline{q}_{\nu}}=1$.
Moreover, let us consider the natural choice of weights given by $p_{i}:=\sigma_{i}^{u}$,
where $u$ is the zero of the associated pressure function, i.~e\@.
$\sum_{i=1}^{m}\sigma_{i}^{u}=1$. Then for $q\geq0$, we have $\beta_{\nu}\left(q\right)=u-uq$
and consequently $u=\delta_{\nu}=\delta_{\nu}^{*}$. Hence, Corollary
\ref{cor:Strongly regular max dim} gives
\[
s_{\nu}=\overline{q}_{\nu}=\underline{q}_{\nu}=\frac{\delta_{\nu}^{*}}{1+\delta_{\nu}^{*}}.
\]
Now, we consider $C^{1}$-self-conformal measures as in \citep{MR2322179}
which do not have to satisfy any separation conditions. Let $U\subset\R$
be a open set and $S:U\rightarrow\R$ a $C^{1}$-map with $\left|S'(x)\right|>0$
for all $x\in U$. Moreover, assume that $S$ is a contracting, i.~e\@.
there exists $0<c<1$ such for all $x,y\in U$ we have $\left|S(x)-S(y)\right|\leq c|x-y|$.
We call $\left\{ S_{i}:[0,1]\rightarrow[0,1]\right\} _{i=1,\dots\ell}$
a $C^{1}$-conformal iterated function system on $[0,1]$, if each
$S_{i}$ extends to an injective contracting $C^{1}$-map $S_{i}:U\rightarrow U$
on an open set $U\supset\left[0,1\right].$
\end{example}

Given a $C^{1}$-conformal iterated function system $\left\{ S_{i}:\left[0,1\right]\rightarrow\left[0,1\right]\right\} _{i=1,\dots,\ell}$
then there exists a non-empty compact set $K\subset[0,1]$ such that
\[
K=\bigcup_{i=1}^{\ell}S_{i}(K).
\]
For a probability vector $(p_{1},\dots,p_{\ell})$ there is a unique
Borel probability measure $\nu$ satisfying
\[
\nu=\sum_{i=1}^{\ell}p_{i}\nu\circ S_{i}^{-1}.
\]
The measure $\nu$ with support $\supp\nu=K$ is called \emph{self-conformal
measure}. Furthermore, we assume that $S_{i}$ do not have all the
same fixed point, then $\nu$ is atomless (\citep{MR1838304}). This
can be seen as follows. Fix $x\in\left[0,1\right]$ such that $\nu\left(\left\{ x\right\} \right)=\max\left(\nu\left(\left\{ y\right\} \right):y\in\left[0,1\right]\right)\eqqcolon m$.
Since with $I\coloneqq\left\{ 1,\ldots,\ell\right\} $
\[
\nu\left(\left\{ x\right\} \right)=\sum_{\omega\in I^{n}}p_{\omega}\nu\left(T_{\omega}^{-1}\left\{ x\right\} \right)
\]
we find $\text{\ensuremath{\nu}}\left(T_{\omega}^{-1}\left\{ x\right\} \right)=m$
for every $\omega\in I^{n}$ and $n\in\N$. Hence, if $m>0$, we have
$x\in T_{\omega}\left(\left[0,1\right]\right)$ for all $\omega\in I^{n}$
and $n\in\N$ and $x$ is the common fixed point for all contractions.
In \citep[Corollary 4.5.]{MR2322179}, building on the earlier work
\citep{MR1838304}, it has been proved that the $L^{q}$-spectrum
$\beta_{\nu}$ of the self-conformal measure $\nu$ exists as a limit.
\begin{prop}[{\citep[Corollary 4.5.]{MR2322179},\citep{MR1838304}}]
 For every $C^{1}$-self-conformal measure $\nu$ we have that the
$L^{q}$-spectrum $\beta_{\nu}$ exists as a limit on $\R_{>0}$.
\end{prop}

\begin{proof}
[Proof of  Theorem \ref{thm:self-conforma-casel}] Combining Theorems
\ref{thm:LqRegularImpliesRegular} and \ref{Thm:MainChain_of_Inequalities+Regularity}
readily proves Theorem \ref{thm:self-conforma-casel}.
\end{proof}
Note, that every measure on $[0,1]$ which is absolutely continuous
with respect to the Lebesgue measure has spectral dimension $1/2$.
Hence, we obtain the following corollary.
\begin{cor}
If the fixed point $\overline{q}_{\nu}$ of $\beta_{\nu}$ is smaller
than $1/2$, then $\nu$ is singular with respect to the Lebesgue
measure.
\end{cor}

\begin{rem}
Note that Example \ref{exmpl:RegularBehaviour_full_dimension=000023}
shows that there exist singular measures with spectral dimension $1/2.$
Hence, in general we can not expect the equivalence of $\overline{q}_{\nu}<1/2$
and $\nu\perp\Lambda$.
\end{rem}

\subsection{Homogeneous Cantor measures \label{NonExistenceSpectralDim}}

Let us recall the construction of general homogeneous Cantor measures
as in \citep{Arzt_diss,MR4048458,MR1473565} allowing us to construct
examples for which the spectral dimension does not exist. Let $J$
be finite or countably infinite subset of $\N$. For every $j\in J$
we define an iterated function system $\mathcal{S}^{(j)}.$ For $i=1,2$
let $S_{i}^{(j)}:[a,b]\rightarrow[a,b]$ be defined by 
\[
S_{i}^{(j)}(x)=r_{i}^{(j)}x+c_{i}^{(j)}
\]
with $r_{i}^{(j)}\in(0,1)$ and $c_{i}^{(j)}\in\R$ are chosen such
that 
\[
a=S_{1}^{(j)}(a)<S_{1}^{(j)}(b)\leq S_{2}^{(j)}(a)<S_{2}^{(j)}(b)=b.
\]
This ensures the open set condition. We define $\mathcal{S}^{(j)}=\left(S_{1}^{(j)},S_{2}^{(j)}\right).$
Moreover let us define an $\textit{environment sequence}$ $\xi\coloneqq\left(\text{\ensuremath{\xi}}_{i}\right)_{i\in\N}\in J^{\N}$.
Each $\xi_{i}$ represent an iterated function system $\mathcal{S}^{(\xi_{i})}$.
To give a suitable coding we define a word space $W_{n}\coloneqq\left\{ 1,2\right\} ^{n}$
of words with length $n$. For $n\in\N$ and $\omega\coloneqq(\omega_{1},\dots,\omega_{n})\in W_{n}$
we set 
\[
S_{\omega}^{(\xi)}\coloneqq S_{\omega_{1}}^{(\xi_{1})}\circ S_{\omega_{2}}^{(\xi_{2})}\circ\dots\circ S_{\omega_{n}}^{(\xi_{n})}.
\]
Now, for any environment sequence $\xi$, we construct a probability
measure $\nu^{(\xi)}$ on $[a,b]$ with support $K^{(\xi)}$ defined
by 
\[
K^{(\xi)}\coloneqq\bigcap_{n=1}^{\infty}\bigcup_{\omega\in W_{n}}\left(S_{\omega_{1}}^{(\xi_{1})}\circ S_{\omega_{2}}^{(\xi_{2})}\circ\dots\circ S_{\omega_{n}}^{(\xi_{n})}\right)\left(\left[a,b\right]\right).
\]
For any $j\in J$, let $\left(p_{i}^{j}\right)_{i=1,2}\in(0,1)^{2}$
with $\sum_{i=1}^{2}p_{i}^{j}=1.$ Moreover, let $p_{\omega}^{(\xi)}\coloneqq\prod_{i=1}^{n}p_{\omega_{i}}^{(\xi_{i})}$
be for $n\in\N$ and $\omega\coloneqq(\omega_{1},\dots,\omega_{n})\in W_{n}$.
Then define the following sequence of probability measures: $\nu_{0}\coloneqq\frac{1}{b-a}\Lambda|_{[a,b]}$
and for $n\in\N$ and $A\in\mathfrak{B}\left(\left[a,b\right]\right)$
\[
\nu_{n}(A)\coloneqq\sum_{\omega\in W_{n}}p_{\omega}^{(\xi)}\nu_{0}\left(\left(S_{\omega}^{(\xi)}\right)^{-1}(A)\right)
\]
then $\nu^{(\xi)}$ is defined by
\[
\nu^{(\xi)}(A)=\lim_{n\rightarrow\infty}\nu_{n}(A).
\]
The theorem of Vitali-Hahn-Saks ensures that $\nu^{(\xi)}$ is a probability
measure on $\mathfrak{B}\left(\left[a,b\right]\right)$ (see \citep[Lemma 3.1.2.]{Arzt_diss})
and for every $n\in\N$ and $\omega\in W_{n}$, we have $\nu^{(\xi)}\left(S_{\omega}^{(\xi)}([a,b])\right)=p_{\omega}^{(\xi)}.$
\begin{example}[Homogeneous Cantor measure with non-converging $L^{q}$-spectrum]
 \label{exa:Non_ex1}Now, let us consider the following environment
\[
\xi_{i}\coloneqq\begin{cases}
1 & ,\exists\ell\in\N_{0}:2^{2\ell}<i\leq2^{2\ell+1}\\
2 & ,\exists\ell\in\N_{0}:2^{2\ell+1}<i\leq2^{2\ell+2}
\end{cases}
\]
and $S_{1}^{(1)}(x)\coloneqq\frac{x}{4},S_{2}^{(1)}(x)\coloneqq\frac{x}{4}+\frac{3}{4}$,
$S_{1}^{(2)}(x)\coloneqq\frac{x}{16},S_{2}^{(2)}(x)\coloneqq\frac{x}{16}+\frac{15}{16}$,
$x\in[0,1]$. Furthermore, let $p_{1},p_{2}\in(0,1)$ be with $p_{1}+p_{2}=1$
then define for every $j\in\left\{ 1,2\right\} $ $p_{1}^{j}=p_{1}$
and $p_{2}^{j}=p_{2}$. First, observe that for $n\in\N$ and $\omega\in W_{n}$,
\[
\Lambda\left(S_{\omega}^{(\xi)}\left(\left[0,1\right]\right)\right)=\begin{cases}
2^{-2^{2l+1}\cdot8/3+10/3-4\left(n-2^{2l+1}\right)} & ,2^{2l+1}<n\leq2^{2l+2}\\
2^{-10/3\cdot2^{2l+2}+10/3-2\left(n-2^{2l+2}\right)} & ,2^{2l+2}<n\leq2^{2l+3}
\end{cases}.
\]
Hence, we obtain for all $n\in\N$ with $2^{2l+1}<n\leq2^{2l+2}$,
$q\geq0$ and $\omega\in W_{n}$ 
\begin{align*}
 & \!\!\!\!\!\!\!\!\!\!\!\!\!\!\frac{1}{-\log\left(\Lambda\left(S_{\omega}^{(\xi)}\left(\left[0,1\right]\right)\right)\right)}\log\left(\sum_{l=0,\nu\left(A_{l}^{n}\right)>0}^{2^{-\log_{2}\left(\Lambda\left(S_{\omega}^{(\xi)}\left(\left[a,b\right]\right)\right)\right)}-1}\nu^{(\xi)}\left(A_{l}^{n}\right)^{q}\right)\\
 & =\frac{1}{-\log\left(\Lambda\left(S_{\omega}^{(\xi)}\left(\left[0,1\right]\right)\right)\right)}\log\left(\sum_{l\in W_{n}}\nu^{(\xi)}\left(S_{l}^{(\xi)}\left(\left[a,b\right]\right)\right)^{q}\right)\\
 & =\frac{\log_{2}\left(p_{1}^{q}+p_{2}^{q}\right)}{-4/3\cdot2^{2l+1}/n+4-10/(3n)}
\end{align*}
and, for $2^{2l+2}<n\leq2^{2l+3}$,
\begin{align*}
\frac{1}{-\log\left(\Lambda\left(S_{\omega}^{(\xi)}\left(\left[0,1\right]\right)\right)\right)}\log\left(\sum_{l\in W_{n}}\nu^{(\xi)}\left(S_{l}^{(\xi)}\left(\left[a,b\right]\right)\right)^{q}\right) & =\frac{\log\left(p_{1}^{q}+p_{2}^{q}\right)}{-\log\left(\Lambda\left(S_{\omega}^{(\xi)}\left(\left[0,1\right]\right)\right)\right)}\\
 & =\frac{\log_{2}\left(p_{1}^{q}+p_{2}^{q}\right)}{4/3\frac{2^{2l+2}}{n}+2-10/(3n)}.
\end{align*}
Therefore, by Fact \ref{Fact IndependenceOFChoiceCordinates_and_factor},
\[
\beta_{\nu^{(\xi)}}\left(q\right)=\begin{cases}
\frac{3}{8}\log_{2}\left(p_{1}^{q}+p_{2}^{q}\right) & \text{for 0\ensuremath{\le}}q\leq1\\
\frac{3}{10}\log_{2}\left(p_{1}^{q}+p_{2}^{q}\right) & \text{for }q>1
\end{cases}
\]
and 
\[
\underline{\beta}_{\nu^{(\xi)}}\left(q\right)\coloneqq\liminf\beta_{n}\left(q\right)=\begin{cases}
\frac{3}{10}\log_{2}\left(p_{1}^{q}+p_{2}^{q}\right) & \text{for }0\leq q\leq1\\
\frac{3}{8}\log_{2}\left(p_{1}^{q}+p_{2}^{q}\right) & \text{for }q>1
\end{cases}.
\]
Hence, by Theorem \ref{Thm:MainChain_of_Inequalities+Regularity},
the upper spectral dimension of $\Delta_{\nu^{(\xi)}}$ is given by
the unique solution of
\[
p_{1}^{\overline{q}_{\nu^{(\xi)}}}+p_{2}^{\overline{q}_{\nu^{(\xi)}}}=2^{\frac{8\overline{q}_{\nu^{(\xi)}}}{3}}.
\]
Furthermore, $\underline{q}_{\nu^{(\xi)}}$is the unique solution
of 
\[
\left(2^{-\frac{10}{3}}p_{1}\right)^{\underline{q}_{\nu^{(\xi)}}}+\left(2^{-\frac{10}{3}}p_{2}\right)^{\underline{q}_{\nu^{(\xi)}}}=1.
\]
See Figure \ref{SumofMeasure} for the two graphs. 
\begin{figure}
\begin{tikzpicture}[line cap=round,line join=round,>=triangle 45,x=1cm,y=1cm]
\begin{axis}[ x=5cm,y=5cm, axis lines=middle, axis line style={very thick},ymajorgrids=false, xmajorgrids=false, grid style={thick,densely dotted,black!20}, xmin=-0.55 , xmax=1.4 , ymin=-0.2, ymax=0.5, xtick={0,1}, ytick={0},] \clip(-1,-0.2) rectangle (2,1.8); 
 \draw[line width=1pt, smooth,samples=100,domain=-0.05:1.4]   
plot(\x,{(3/10)*log10((0.25)^((\x))+(0.75)^((\x)))/log10(2)}); 
 \draw[line width=1pt, smooth,samples=100,domain=-0.2:1.4]   
plot(\x,{(\x))}); 
\draw [line width=1pt, domain=-0.05 :1.4] 
plot(\x,{(3/8)*log10((0.25)^((\x))+(0.75)^((\x)))/log10(2)}); 
\draw (1.3 ,0.1) node[anchor=north west] {$\Large q$};
\draw (0.01 ,0.52 ) node[anchor=north west] {$\Large \displaystyle{\beta_{\nu^{(\xi)}}(q)}$};
 \draw (-0.43 ,0.467 ) node[anchor=north west] {$\Large \displaystyle{\overline{\dim}_M\left(\nu^{(\xi)}\right)}$}; 
\end{axis} 
\end{tikzpicture}

\caption{$\beta_{\nu^{(\xi)}}$ and $\liminf\beta_{n}^{\nu^{(\xi)}}$with $p_{1}=0.25$.
\label{SumofMeasure}}
\end{figure}
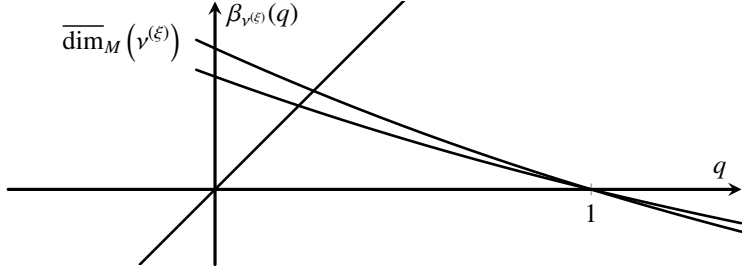
For the special case $p_{1}=1/2$ we obtain 
\[
\beta_{\nu^{(\xi)}}\left(q\right)=\begin{cases}
\frac{3}{8}(1-q) & \text{for 0\ensuremath{\le}}q\leq1\\
\frac{3}{10}(1-q) & \text{for }q>1
\end{cases}
\]
and 
\[
\underline{\beta}_{\nu^{(\xi)}}\left(q\right)\coloneqq\liminf\beta_{n}\left(q\right)=\begin{cases}
\frac{3}{10}(1-q) & \text{for }0\leq q\leq1\\
\frac{3}{8}(1-q) & \text{for }q>1
\end{cases}.
\]
In particular in this special case, we have, $\underline{\dim}_{M}\left(\supp\nu\right)=3/10$
and $\overline{\dim}_{M}\left(\nu\right)=3/8$. Now, applying Theorem
\ref{Thm:MainChain_of_Inequalities+Regularity} in tandem with Proposition
\ref{prop:UppperboundForLinear_s_<liminfq_n}, we conclude
\[
\underline{s}_{\nu^{(\xi)}}\leq\frac{\underline{\dim}_{M}\left(\nu\right)}{1+\underline{\dim}_{M}\left(\nu\right)}=\frac{3}{13}<\frac{3}{11}=\frac{\overline{\dim}_{M}\left(\nu\right)}{1+\overline{\dim}_{M}\left(\nu\right)}=\overline{s}_{\nu^{(\xi)}}.
\]
Now, let us prove that $\underline{s}_{\nu^{(\xi)}}=\frac{3}{13}$.
Define for $n\in\N$
\[
R(n)=\begin{cases}
-2^{2\ell+1}\cdot8/3+10/3-4\left(n-2^{2\ell+1}\right) & ,2^{2\ell+1}<n\leq2^{2\ell+2},\\
-10/3\cdot2^{2\ell+2}+10/3-2\left(n-2^{2\ell+2}\right) & ,2^{2\ell+2}<n\leq2^{2\ell+3}.
\end{cases}
\]
First note that for $m=1+\frac{1}{2^{7}}$ and all $\omega\in W_{n}$,
there exists $\eta\in W_{2}$ such that
\[
\text{\ensuremath{S_{\omega\eta}^{(\xi)}\left(\left[0,1\right]\right)\subset\left\langle S_{\omega}^{(\xi)}\left(\left[0,1\right]\right)\right\rangle _{m}}}.
\]
Hence, we obtain
\begin{equation}
\nu\left(\left\langle S_{\omega}^{(\xi)}\left(\left[0,1\right]\right)\right\rangle _{m}\right)\Lambda\left(\left\langle S_{\omega}^{(\xi)}\left(\left[0,1\right]\right)\right\rangle _{m}\right)\geq\frac{2^{-n+R(n)-2}}{m}.\label{eq:-2}
\end{equation}
Recall that by Proposition \ref{prop:estimateOfSpectralDim_N^L_m_N^R},
for $x>0$,
\[
\mathcal{N}_{m}^{L}\left(x\right)=\sup\left\{ \card\left(P\right):P\in\Pi_{0}:\min_{C\in P}\nu\left(\left\langle I\right\rangle _{m}\right)\Lambda\left(\left\langle I\right\rangle _{m}\right)\geq\frac{4}{x(m-1)}\right\} \leq N_{\nu}(x).
\]
Now, with $x_{n}=m\frac{2^{n-R(n)+4}}{m-1}$ by Eq. \eqref{eq:-2}
we have that $\mathcal{N}_{m}^{L}(x_{n})\geq2^{n}$, which implies
\begin{align*}
\frac{\log\left(N_{\nu}(x_{n})\right)}{\log(x_{n})} & \geq\left(1-\frac{R(n)}{n}+4/n+\frac{\log\left(m/\left(m-1\right)\right)}{n\log(2)}\right)^{-1}\\
 & =\begin{cases}
\left(1+\frac{2^{2l+1}\cdot8/3-10/3+4\left(n-2^{2l+1}\right)}{n}+4/n-\frac{\log\left(m/\left(m-1\right)\right)}{n\log(2)}\right)^{-1} & ,2^{2l+1}<n\leq2^{2l+2},\\
\left(1+\frac{10/3\cdot2^{2l+2}-10/3+2\left(n-2^{2l+2}\right))}{n}+4/n-\frac{\log\left(m/\left(m-1\right)\right)}{n\log(2)}\right)^{-1} & ,2^{2l+2}<n\leq2^{2l+3},
\end{cases}\\
 & \geq\begin{cases}
\left(1+\frac{8}{3}+\left(4-10/3\right)/n-\frac{\log\left(m/\left(m-1\right)\right)}{n\log(2)}\right)^{-1} & ,2^{2l+1}<n\leq2^{2l+2},\\
\left(1+\frac{10}{3}+\left(4-10/3\right)/n-\frac{\log\left(m/\left(m-1\right)\right)}{n\log(2)}\right)^{-1} & ,2^{2l+2}<n\leq2^{2l+3}.
\end{cases}
\end{align*}
This shows
\[
\liminf_{n\rightarrow\infty}\frac{\log\left(N_{\nu}(x_{n})\right)}{\log(x_{n})}\geq\frac{1}{1+10/3}.
\]
Since $x_{n}\leq x_{n+1}\leq2^{4}x_{n}$, 
\[
\liminf_{x\rightarrow\infty}\frac{\log\left(N_{\nu}(x)\right)}{\log(x)}\geq\liminf_{n\rightarrow\infty}\frac{\log\left(N_{\nu}(x_{n})\right)}{\log(x_{n})}\geq\frac{1}{1+10/3}.
\]
For $x>0$, choose $n\in\N$ such that $x_{n}<x\leq x_{n+1}$ and
therefore
\[
\liminf_{x\rightarrow\infty}\frac{\log\left(N_{\nu}(x)\right)}{\log(x)}\geq\liminf_{n\rightarrow\infty}\frac{\log\left(N_{\nu}(x_{n})\right)}{\log(x_{n})+\log(x_{n+1}/x_{n})}\geq\frac{1}{1+10/3}.
\]
\end{example}

\begin{example}[Homogeneous Cantor measure with vanishing $L^{q}$-spectrum]
 \label{exa:Atomic_s=00003D0}Let us now construct an example for
which the spectral dimension equals $0.$ Let us consider the following
environment $\xi_{n}\coloneqq n$, $n\in\N$, and $S_{1}^{(n)}(x)\coloneqq\frac{x}{4^{n}},S_{2}^{(n)}(x)\coloneqq\frac{x}{4^{n}}+\frac{4^{n}-1}{4^{n}}$,
$x\in[0,1]$, $p_{1}^{n}=p_{2}^{n}=1/2$. It follows that $\overline{\dim}_{M}\left(K^{(\xi)}\right)=0$
with Corollary \ref{cor:spectralDimGeneralUpperBound} that $s_{\nu}=0.$
Indeed, let us verify that $\dim_{M}\left(K^{(\xi)}\right)=0$. For
every $\omega\in W_{m}$, we have
\[
\Lambda\left(S_{\omega}^{(\xi)}\left(\left[0,1\right]\right)\right)=2^{-(m+1)m}.
\]
Now, for $n\in\N$, there exists $m(n)=m\in\N$ such that $2^{-(m+2)(m+1)}\leq2^{-n}\leq2^{-(m+1)m}$.
In particular, $m\left(n\right)\ll\sqrt{n}$ and we conclude
\[
\frac{\log\left(\beta_{n}(0)\right)}{\log(2^{n})}\leq\frac{\log\left(\beta_{(m+2)(m+1)}(0)\right)}{\log(2^{n})}=\frac{m(n)+1}{n}\rightarrow0\;\text{for }\,\,n\rightarrow\infty.
\]
\end{example}

\subsection{Pure point measures\label{subsec:The-pure-point-1}}

In this section we give examples of singular measures $\eta$ on $\left(0,1\right)$
of pure point type such that the spectral dimension attains any value
in $\left[0,1/2\right]$. To fix notation, throughout this section,
we write $\eta\coloneqq\sum p_{k}\delta_{x_{k}}$ with $\left(p_{k}\right)\in\left(\R_{>0}\right)^{\N}$,
$\sum p_{k}<\infty$ and $\left(x_{k}\right)\in\left(0,1\right)^{\N}$.

The first example shows that it is possible for the spectral dimension
to be $0$ even though the Minkowski dimension is $1$.
\begin{example}
\label{Example:fastDecay}In this example we consider pure atomic
measures $\eta$ with $\left(x_{n}\right){}_{n\in\N}\in\left(\mathbb{Q}\cap(0,1)\right)^{\N}$
such that $x_{n}\neq x_{m}$ for $m\neq n$, and $p_{n}\leq C_{1}\e^{-n}$,
for all $n\in\N$. We will show that the spectral dimension exists
and equals $s_{\eta}=0$. Define for $I_{k}^{n}\coloneqq\left(x_{k}-\frac{b_{n}}{\e^{n}},x_{k}+\frac{b_{n}}{\e^{n}}\right]\cap[0,1]$
with $b_{n}\coloneqq\min\left\{ \left|x_{l_{1}}-x_{l_{2}}\right|\colon l_{1}\neq l_{2},\ l_{1},l_{2}\leq n\right\} $
for all $k=1,\dots,n$. Then we have $x_{k}-\frac{b_{n}}{\e^{n}},x_{k}+\frac{b_{n}}{\e^{n}}\notin\mathbb{Q}$
and
\[
\max_{k=1,\dots,n}\Lambda\left(I_{k}^{n}\right)\nu\left(I_{k}^{n}\right)\leq\sum_{l=1}^{\infty}p_{l}\max_{k=1,\dots,n}\Lambda(I_{k}^{n})\leq\frac{C_{1}}{\e-1}\frac{2}{\e^{n}}.
\]
Let $A_{k}^{n}$ denote the disjoint half open intervals such that
$(0,1]\setminus\bigcup_{k=1}^{n}I_{k}^{n}=\bigcup_{k=1}^{m(n)}A_{k}^{n}$
such that $m(n)\leq n+1$. Then we conclude
\[
\max_{k=1,\dots,m(n)}\Lambda(A_{k}^{n})\nu(A_{k}^{n})\leq\sum_{l=n+1}^{\infty}p_{l}\leq C_{1}\frac{\e^{-n}}{\e-1},
\]
implying $\gamma_{2n+1}\leq\frac{2C_{1}}{\e-1}\e^{-n}$ and therefore,
for every $m\in\N$ with $2n+1\leq m\leq2n+3$, 
\[
\gamma_{m}\leq\text{\ensuremath{\gamma}}_{2n+1}\ll\e^{-(m-3)/2}.
\]
This shows $\underline{\alpha}_{\nu}=\infty$ implying by Lemma \ref{lem: SolomyakEquivalence}
that $s_{\eta}=0.$
\end{example}

If $\left(x_{k}\right)\in\left(0,1\right)^{\N}$ is strictly decreasing,
then in \citep{MR3648085} $\Delta_{\eta}$ is called Kre\u{\i}n-Feller
operator of Stieltjes type. We start with a general observation which
is a consequence of the Poincaré inequality as stated in Lemma \ref{lem:Poincare_inequality}.
\begin{lem}
\label{lem:StrongDecay-1}Assume that $\left(x_{k}\right)\in\left(0,1\right)^{\N}$
is strictly decreasing such that for an increasing function $f:\mathbb{Z}_{+}\rightarrow\R_{+}$
and all $k\in\N_{>1}$,
\[
\frac{x_{k}+x_{k-1}}{2}\sum_{l=k}^{\infty}p_{l}\leq1/f(k).
\]
 Then, for all $x\geq0$, we have $N_{\eta}(x)-2\leq\check{f}^{-1}\left(x\right)\coloneqq\inf\left\{ n\in\mathbb{\N}_{>0}:f\left(n\right)\geq x\right\} .$
\end{lem}

\begin{proof}
The Poincaré inequality (Lemma \ref{lem:Poincare_inequality}) guarantees
for all $k\in\N$
\[
\lambda_{\eta,\left[0,\frac{x_{k}+x_{k-1}}{2}\right]}^{1}\geq\frac{1}{\frac{x_{k}+x_{k-1}}{2}\sum_{l=k}^{\infty}p_{l}}.
\]
Since
\[
\frac{1}{\frac{x_{k}+x_{k-1}}{2}\sum_{l=k}^{\infty}p_{l}}\geq f\left(k\right)\geq x\implies k\geq\check{f}^{-1}\left(x\right),
\]
 using the superadditivity of the Neumann eigenvalue counting function
with $k_{x}\coloneqq\check{f}^{-1}\left(x\right)$, by Proposition
\ref{prop:SubandSupperAdditivity}, we have
\[
N_{\eta}(x)\leq N_{\eta,\left[0,\frac{x_{k_{x}+}x_{k_{x}-1}}{2}\right]}(x)+N_{\eta,\left[\frac{x_{k_{x}+}x_{k_{x}-1}}{2},1\right]}(x)+2\leq2+\check{f}^{-1}\left(x\right).
\]
\end{proof}
\begin{example}[Dirac comb with exponential decay]
 \label{exmpl:ExponentialDecay}If $\left(x_{n}\right)$ or $\left(\sum_{k\geq m}p_{k}\right)_{m}$
decays exponentially, then $N_{\eta}(x)\ll\log(x)$, hence the spectral
dimension $s_{\eta}$ equals $0$. Further, the $L^{q}$-spectrum
is given by
\[
\beta_{\eta}\left(q\right)=\begin{cases}
\overline{\dim}_{M}\left(\eta\right) & q=0,\\
0 & q>0.
\end{cases}
\]
\end{example}

\begin{example}[Dirac comb with at most power law decay]
\label{exmpl:PurePointoneLimit_at_most_power_law} Assume that $\left(x_{k}\right)\in\left(0,1\right)^{\N}$
is strictly decreasing and
\[
p_{n}\gg n^{-u_{1}}f_{1}(n),\;(x_{n}-x_{n+1})\gg n^{-u_{2}}f_{2}(n),
\]
with $u_{1},u_{2}\geq1$ and $\lim_{n\rightarrow\infty}\log(f_{i}(n))/\log(n)=0$
for $i=1,2$, then
\[
\frac{1}{u_{1}+u_{2}}\leq\underline{s}_{\eta}
\]
and, in particular, if $u_{1}+u_{2}=2$, we have $s_{\eta}=\frac{1}{2}.$
To see this, define
\[
I_{k}\coloneqq\left[x_{k}-\frac{\min\left\{ x_{k}-x_{k+1},x_{k-1}-x_{k}\right\} }{2},x_{k}+\frac{\min\left\{ x_{k}-x_{k+1},x_{k-1}-x_{k}\right\} }{2}\right]
\]
 for $k=1,\ldots,n$. Then we have $\nu\left(\left\langle I_{k}\right\rangle _{2}\right)=\nu\left(I_{k}\right)$
and for $n$ large enough
\[
Cn^{-(u_{1}+u_{2})}f_{1}(n)f_{2}(n)\leq\min_{k=1,\dots,n}\nu\left(\left\langle I_{k}\right\rangle _{2}\right)\Lambda\left(\left\langle I_{k}\right\rangle _{2}\right)
\]
with $C>0$ suitable, which implies for every $\varepsilon>0$ 
\[
\mathcal{N}_{2}^{L}\left(x\right)\gg x^{\frac{1}{u_{1}+u_{2}+\varepsilon}}
\]
proving the claim. Since, in the case $u_{1}+u_{2}=2$, we always
have $\overline{s}_{\eta}\leq1/2$, it readily follows that $s_{\eta}=1/2.$
\end{example}

\begin{example}[Dirac comb with at most geometric decay and full dimension]
\label{exmpl:RegularBehaviour_full_dimension=000023}Assume that
$\left(x_{k}\right)\in\left(0,1\right)^{\N}$ is strictly decreasing
and
\[
\frac{f_{1}\left(n\right)}{n}\ll p_{n},\:\frac{f_{2}\left(n\right)}{n}\ll\left(x_{n-1}-x_{n}\right)
\]
with $\lim_{n\rightarrow\infty}\log(f_{i}(n))/\log(n)=0$, for $i=1,2$,
then spectral dimension exists and equals 
\[
s_{\eta}=1/2
\]
and, for $q\in[0,1]$,
\[
\beta_{\eta}\left(q\right)=\begin{cases}
1-q & ,q\in\left[0,1\right],\\
0 & ,q>1.
\end{cases}
\]
Indeed, observe that Example \ref{exmpl:PurePointoneLimit_at_most_power_law}
implies $1/2\leq\underline{s}_{\eta}$ and by Theorem \ref{Thm:MainChain_of_Inequalities+Regularity}
and Corollary \ref{cor:spectralDimGeneralUpperBound} we also have
$\underline{s}_{\eta}\leq\overline{s}_{\eta}=\overline{q}_{\eta}\leq1/2$
which shows $s_{\eta}=1/2.$ The second statement is then a direct
consequence of the first part of Corollary \ref{cor:q-1 implies s0=00003D1/2}.
\end{example}

The following example shows that the spectral dimension attains every
value in $(\text{0,1/2)}.$

\begin{example}[Dirac comb with power law decay]
\label{exmpl:GeneralBound_power_law}If
\[
\lim_{n\to\infty}-\log p_{n}/\log n=u_{1}>1\;\text{and }x_{k}\coloneqq(k+1)^{-u_{2}},k\in\N,u_{2}>0,
\]
then the $L^{q}$-spectrum exists as a limit on the positive half-line
and we have
\[
\beta_{\eta}\left(q\right)=\begin{cases}
\frac{1}{u_{2}+1}-q\frac{u_{1}}{u_{2}+1} & \text{for }q\in\left[0,1/u_{1}\right],\\
0 & \text{for }q>1/u_{1}.
\end{cases}
\]
Consequently, $\eta$ is $L^{q}$-regular and the spectral dimension
exists and equals
\[
s_{\eta}=\frac{1}{u_{1}+u_{2}+1}.
\]
In particular, for $u_{1}=u_{2}+1$ we have
\[
s_{\eta}=\frac{\dim_{M}\left(\eta\right)}{2}.
\]
This can be seen as follows. For every $\varepsilon>0$ uniformly
in $n\in\N$, we have $n^{-(u_{1}+\varepsilon)}\ll p_{n}\ll n^{-u_{1}+\varepsilon}$.
For suitable $C>0$, 
\[
x_{m}-x_{m+1}=\frac{(m+1)^{u_{2}}-m^{u_{2}}}{m^{u_{2}}(m+1)^{u_{2}}}=\frac{1}{(m+1)^{u_{2}}m}\frac{\left(\frac{m+1}{m}\right)^{u_{2}}-1}{1/m}\geq\frac{C}{m^{u_{2}+1}}.
\]
If $2^{-n}<C(m+1)^{-(u_{2}+1)}$, then $m<\left(2^{n}C\right)^{\frac{1}{u_{2}+1}}$.
Combining these observations, we obtain

\begin{align*}
\sum_{C\in\mathcal{D}_{n}}\eta\left(C\right)^{q} & \geq\sum_{k=1}^{C^{1/(u_{2}+1)}2^{n/\left(u_{2}+1\right)}}p_{k}^{q}\gg\sum_{k=1}^{C^{1/(u_{2}+1)}2^{n/\left(u_{2}+1\right)}}k^{-\left(u_{1}+\epsilon\right)q}\asymp2^{n\left(-\left(u_{1}+\epsilon\right)q+1\right)/\left(u_{2}+1\right)}.
\end{align*}
For $q\in\left[0,1/(u_{1}+\varepsilon)\right)$, this gives
\[
\beta_{\eta}\left(q\right)\geq\liminf_{n}\beta_{n}^{\eta}\left(q\right)\geq\frac{1}{u_{2}+1}-q\frac{u_{1}+\varepsilon}{u_{2}+1}.
\]
 Letting $\varepsilon\rightarrow0$ this shows for $q\in[0,1/u_{1}]$
\[
\beta_{\eta}\left(q\right)\geq\liminf_{n}\beta_{n}^{\eta}\left(q\right)\geq\begin{cases}
\frac{1}{u_{2}+1}-q\frac{u_{1}}{u_{2}+1} & \text{for }q\in\left[0,1/u_{1}\right],\\
0 & \text{for }q>1/u_{1}.
\end{cases}
\]
Moreover, for $m\geq\left(2^{n}C\right)^{\frac{1}{u_{2}+1}}$ and
$k2^{-n}<x_{m}$, we have $k<2^{n}\left(2^{n}C\right)^{-\frac{u_{2}}{u_{2}+1}}$.
From this inequality, using the integral test for convergence, we
obtain for $q>1/\left(u_{1}-\varepsilon\right)$,
\begin{align*}
\sum_{k=0}^{2^{n}-1}\eta\left(A_{k}^{n}\right)^{q} & =\sum_{k=0}^{2^{n}-1}\left(\sum_{k2^{-n}<x_{m}\leq(k+1)2^{-n}}p_{m}\right)^{q}\\
 & \ll\sum_{m<C^{1/u_{2}}2^{n/u_{2}}}m^{q\left(-u_{1}+\varepsilon\right)}+\sum_{k=0}^{C^{-\frac{u_{2}}{u_{2}+1}}2^{\frac{n}{u_{2}+1}}}\left(\sum_{k2^{-n}<x_{m}\leq(k+1)2^{-n}}p_{m}\right)^{q}\\
 & \ll1+2^{n\left(\frac{1}{u_{2}+1}+q\left(-\frac{(u_{1}-\varepsilon-1)}{u_{2}}+\frac{(u_{1}-\varepsilon-1)-u_{2}}{u_{2}(u_{2}+1)}\right)\right)}=1+2^{n\left(\frac{1}{u_{2}+1}-q\frac{u_{1}-\varepsilon}{u_{2}+1}\right)}\ll1.
\end{align*}
Hence, $\beta_{\eta}(q)=0$ for $q\geq1/u_{1}$. Since $\beta_{\eta}\left(0\right)=\delta_{\eta}^{*}=1/(u_{2}+1)$,
by the convexity of $\beta_{\eta}$, it follows that for all $q\in[0,1/u_{1}]$,
\[
\frac{1}{u_{2}+1}-q\frac{u_{1}}{u_{2}+1}\leq\liminf_{n}\beta_{n}^{\eta}\left(q\right)\leq\beta_{\eta}\left(q\right)\leq\frac{1}{u_{2}+1}-q\frac{u_{1}}{u_{2}+1}.
\]
Theorem \ref{Thm:MainChain_of_Inequalities+Regularity} in tandem
with Theorem \ref{thm:LqRegularImpliesRegular} then gives $s_{\eta}=\overline{q}_{\eta}=\left(u_{1}+u_{2}+1\right)^{-1}$.
\end{example}

The last example demonstrates how one can improve Example \ref{exmpl:ExponentialDecay},
if one knows the exact exponential asymptotics of $p_{n}$ and $x_{n}$,
which in turn forces an exact logarithmic asymptotics for the eigenvalue
counting function. The following elementary lemma is provided for
preparation.

\begin{lem}
\label{lem:ExpliciFirsttEigenvalueDirc}Let $\nu\coloneqq p\delta_{z}$
be with $z\in\left(0,1\right)$, $p>\text{0 }$ and $0<a<z$, $0<b<1-z$.
Then the first eigenvalue of $\Delta_{\nu,[z-a,z-b]}$ is given by
\[
\lambda_{\nu,[z-a,z+b]}^{1}=\frac{a+b}{pab}.
\]
\end{lem}

\begin{proof}
Observe that each function $f\in\dom\left(\E_{\nu,[z-a,z+b]}\right)$
is of triangular shape, i\@.\,e\@. $f(x)=f_{a,b,c}(x)\coloneqq c\frac{(x-(z-a))}{a}\1_{[z-a,z]}+c\frac{(z+b-x)}{b}\1_{(z,z+b]}.$
Therefore, using the min-max principle we get
\begin{align*}
\lambda_{\eta,[z-a,z+b]}^{1} & =\inf\left\{ \frac{\E(\psi,\psi)}{\langle\psi,\psi\rangle_{\eta}}:\psi\in\dom\left(\E_{\nu,[z-a,z+b]}\right)^{\star}\right\} =\inf_{c>0}c^{2}\left(\frac{1}{a}+\frac{1}{b}\right)/\left(c^{2}p\right)=\frac{a+b}{pab}.
\end{align*}
\end{proof}
\begin{example}[Dirac comb with exponential decay -- precise asymptotics]
\label{exmpl:ExplicitAssymptoticCase_Exp_decay}Let us consider the
case $\sum_{k=1}^{n}p_{k}=1-\e^{-\alpha n}$, and $x_{k}\coloneqq\e^{-\beta k}$
for some $\alpha,\mathbf{\beta>\text{0}}$. Then 
\[
\text{\ensuremath{\lim_{x\rightarrow\infty}}}\frac{N_{\eta}(x)}{\log\left(x\right)}=\frac{1}{\alpha+\beta}.
\]
To see this, we consider the intervals $A_{n}\coloneqq\left[x_{n}-\left(x_{n}-x_{n+1}\right)/2,x_{n}+\left(x_{n-1}-x_{n}\right)/2\right],$for
$n>1$. Then by Lemma \ref{lem:ExpliciFirsttEigenvalueDirc}
\begin{align*}
\lambda_{\eta,A_{n}}^{1} & =\frac{\e^{(\alpha+\beta)n}2\left(\e^{\beta}-\e^{-\beta}\right)}{\left(1-\e^{-\beta}\right)\left(\e^{\alpha}-1\right)\left(\e^{\beta}-1\right)}
\end{align*}
implying 
\[
N_{\eta,[0,1]}\left(x\right)\geq\left\lfloor \frac{1}{(\alpha+\beta)}\log\left(x\frac{\left(1-\e^{-\beta}\right)\left(\e^{\alpha}-1\right)\left(\e^{\beta}-1\right)}{2\left(\e^{\beta}-\e^{-\beta}\right)}\right)\right\rfloor .
\]
For the upper bound, we observe that for $A^{k}\coloneqq[0,\frac{x_{k}+x_{k-1}}{2}]$,
we have 
\[
\lambda_{\eta,A^{k}}^{1}\geq\frac{1}{\frac{x_{k}+x_{k-1}}{2}\sum_{\ell=k}^{\infty}p_{\ell}}=\frac{2}{\e^{-\beta k+(k-1)\alpha}(1+\e^{-\beta})}=2\frac{\e^{(k-1)\alpha+k\beta}}{1+\e^{-\beta}}.
\]
Then, for $x>4$, we obtain 
\begin{align*}
2\frac{\e^{(k-1)\alpha+k\beta}}{1+\e^{-\beta}}>x & \Longleftrightarrow k>\frac{1}{(\alpha+\beta)}\log\left(\left(1+\e^{-\beta}\right)\frac{\e^{\alpha}x}{2}\right).
\end{align*}
This gives for $n_{x}=\left\lceil \frac{1}{(\alpha+\beta)}\log\left(\left(1+\e^{-\beta}\right)\frac{\e^{\alpha}x}{2}\right)\right\rceil $
\[
N_{\eta,[0,1]}(x)\leq N_{\eta,A^{k}}(x)+N_{\eta,\left[\frac{x_{n_{x}}+x_{n_{x}-1}}{2},1\right]}(x)+2\leq3+\left(n_{x}-1\right)\leq2+\frac{1}{(\alpha+\beta)}\log\left(\left(1+\e^{-\beta}\right)\frac{\e^{\alpha}x}{2}\right)
\]
and the claim follows.
\end{example}

\bibliographystyle{amsalpha}
\bibliography{../../bib}

\end{document}